\documentclass[12pt]{amsart}
\usepackage[normalem]{ulem} 
\usepackage{lmodern}
\usepackage{amsmath}
\usepackage{amssymb}
\usepackage{dsfont}  
\usepackage{hyperref}
\usepackage{graphicx}
\usepackage[all]{xy}
\usepackage{latexsym}
\usepackage{epstopdf}
\usepackage{pgf}
\usepackage{tikz}
\usepackage{tikz-cd}
\usepackage{pgf}
\usepackage{float}
\usepackage{subcaption} 
\usepackage{xcolor}
\usepackage{algorithm2e}

\usetikzlibrary{matrix,arrows,backgrounds,shapes.misc,shapes.geometric,patterns,calc,positioning}
\usetikzlibrary{calc,shapes}
\usetikzlibrary{decorations.pathmorphing}
\usetikzlibrary{positioning,arrows}
\usetikzlibrary{decorations.markings}

\usepackage[left=3cm,right=3cm,top=2.8cm,bottom=2.8cm]{geometry}
\newtheorem{theorem}{Theorem}[section] 
\newtheorem{thm}{Theorem}
        
\newtheorem{lemma}[theorem]{Lemma}
\newtheorem{corollary}[theorem]{Corollary}

\newtheorem{proposition}[theorem]{Proposition} 
\newtheorem{conjecture}[theorem]{Conjecture}   
\newtheorem{conj}[thm]{Conjecture}

\theoremstyle{remark}
\newtheorem{remark}[theorem]{Remark}

\theoremstyle{definition}
\newtheorem{definition}[theorem]{Definition} 
\newtheorem{example}[theorem]{Example} 

\numberwithin{equation}{section}

\newcommand{\rk}{\operatorname{rank}\nolimits}

\newcommand{\Hom}{\operatorname{Hom}\nolimits}
\newcommand{\Id}{\operatorname{Id}\nolimits}
\newcommand{\Ext}{\operatorname{Ext}\nolimits}

\newcommand{\CC}{\mathbb{C}}
\newcommand{\ZZ}{\mathbb{Z}}

\newcommand{\Gr}{\mathrm{Gr}}

\newcommand{\con}{{\rm con}}

\newcommand{\dec}{{\rm dec}}
\newcommand{\inc}{{\rm inc}}

\newcommand{\thfrac}[3]{ 	
\begin{aligned} #1\\ \hline \\[-3\jot] #2 \\ \hline \\[-3\jot] #3 \end{aligned}
}

\newcommand{\ffrac}[4]{ 	
\begin{aligned} #1\\ \hline \\[-3\jot] #2 \\ \hline \\[-3\jot] #3 \\ \hline \\[-3\jot] #4 \end{aligned}
}

\newcommand\scalemath[2]{\scalebox{#1}{\mbox{\ensuremath{\displaystyle #2}}}}

\usepackage{color}

\title{Rigid indecomposable modules in Grassmannian cluster categories}

\author{Karin Baur, Dusko Bogdanic, Ana Garcia Elsener, and Jian-Rong Li}

\address{Karin Baur, School of Mathematics, University of Leeds,
Leeds, LS2 9JT, UK\\
On leave from the University of Graz}
\email{K.U.Baur@leeds.ac.uk}

\address{Dusko Bogdanic, Faculty of Natural Sciences and Mathematics, University of Banja Luka, Mladena Stojanovica 2, 78000 Banja Luka, Bosnia and Herzegovina}
\email{dusko.bogdanic@pmf.unibl.org}

\address{Ana Garcia Elsener, Universidad Nacional de Mar del Plata, Facultad de Ciencias Exactas y Naturales, Departamento de Matem\'atica. Dean Funes 3350 CP7600 Buenos Aires, Argentina}
\email{elsener@mdp.edu.ar}

\address{Jian-Rong Li, Faculty of Mathematics, University of Vienna, Oskar-Morgenstern-Platz 1, 1090 Vienna, Austria.}
\email{lijr07@gmail.com}

\newcommand{\monthword}[1]{\ifcase#1\or January\or February\or March\or April\or May\or 
June\or July\or August\or September\or October\or November\or December\fi}
\date{\monthword{\the\month} \the\day, \the\year } 

\begin{document}

\maketitle 

\begin{abstract}
The coordinate ring  
of the Grassmannian variety of $k$-dimensional subspaces in $\CC^n$ 
has a cluster algebra structure with Pl\"ucker relations giving rise to exchange relations. 
In this paper, we study indecomposable modules of the corresponding 
Grassmannian cluster categories ${\rm CM}(B_{k,n})$.\  Jensen, King, and Su have associated a Kac-Moody  root system 
$J_{k,n}$ to ${\rm CM}(B_{k,n})$ and shown that in the finite types, rigid indecomposable modules 
correspond to roots. In general, the link between the category ${\rm CM}(B_{k,n})$ and 
the root system $J_{k,n}$ remains mysterious and it is an open question whether indecomposables 
always give roots. 
In this paper, we provide evidence for this association in the infinite types: 
we show that every indecomposable rank 2 module corresponds to a root of the 
associated root system. 
We also show that indecomposable rank 3 modules in ${\rm CM}(B_{3,n})$ all give rise to 
roots of $J_{3,n}$. 
For the rank 3 modules in ${\rm CM}(B_{3,n})$ corresponding to real roots, we show that their 
underlying profiles are cyclic permutations of a certain canonical one. 
We also characterize the rank 3 modules in ${\rm CM}(B_{3,n})$ 
corresponding to imaginary roots. 
By proving that there are exactly 225 profiles of 
rigid indecomposable rank 3 modules in ${\rm CM}(B_{3,9})$ we 
confirm the link between the Grassmannian cluster category and the associated root system in this case. 
We conjecture that the profile of any rigid indecomposable module in ${\rm CM}(B_{k,n})$ corresponding to a real root is a cyclic permutation of a canonical profile. 
\end{abstract}

\tableofcontents

\section{Introduction}

Consider the homogeneous coordinate ring $\CC[\Gr(2,n)]$ of the Grassmannian of 2-dimensional 
subspaces of $\CC^n$. This is one of the key initial examples of Fomin and Zelevinsky's 
theory of cluster algebras, \cite[\S 12.2]{fz}: the cluster variables are the Pl\"ucker coordinates, 
the exchange relations arise from the 
short Pl\"ucker relations, and clusters are in bijection with triangulations of a convex $n$-gon. 
Scott then proved in~\cite{scott} 
that this cluster structure can be generalized to the coordinate ring $\CC[\Gr(k,n)]$, 
where additional cluster variables appear (in general, infinitely many) and more exchange relations. 
This has sparked a lot of research activities in cluster theory, 
e.g.  \cite{SW, gls, HL10, gssv, MuS, BKM, JKS, mr,fraser}.

In particular, 
Jensen, King and Su showed in~\cite{JKS} that the category ${\rm CM}(B_{k,n})$ 
of Cohen-Macaulay modules over a quotient $B_{k,n}$ of a preprojective algebra of affine type 
$A$ provides an additive categorification of Scott's cluster algebra structure. 
The category ${\rm CM}(B_{k,n})$ is called the Grassmannian cluster category. 
They also show that there is a cluster character 
on this category, sending rigid indecomposable objects to cluster variables (\cite[Section 9]{JKS}). 
Without loss of generality, we will assume $1\le k\le n/2$ from now on.

Through this categorification, the classification of rigid indecomposable modules in ${\rm CM}(B_{k,n})$ 
(i.e., indecomposable modules $M$ with Ext$^1(M,M)=0$) 
becomes an important tool towards characterising cluster variables in $\CC[\Gr(k,n)]$ as well as in the 
classification of real prime modules of quantum affine algebras of type $A$, \cite{HL10, KKKO18}. 

In this paper, we study indecomposable modules of ${\rm CM}(B_{k,n})$ with the goal of 
providing an understanding of the associated cluster algebras. A first contribution to this is the fact that 
in the infinite types, all components in the Auslander-Reiten quiver are tubes, Proposition \ref{propos:AR-tubes}. With this, 
we have some control over certain types of indecomposable modules.

Among the indecomposable modules are the rank 1 modules which are known to be in 
bijection with $k$-subsets of $\{1,2,\dots, n\}$. These are the building blocks of the 
category as any module 
in ${\rm CM}(B_{k,n})$ can be filtered by rank 1 modules (see Section~\ref{ssec:CM-setup}). 
Using this, in \cite[Section 8]{JKS}, a map  is 
defined from indecomposable modules of ${\rm CM}(B_{k,n})$ to a root lattice 
by associating a module with its class in the Grothendieck group and identifying the 
latter with a root lattice (see Section~\ref{ssec:roots-Jkn}). 
Let $J_{k,n}$ be the graph with nodes $1,2,\dots, n-1$ on a line 
and node $n$ attached to node $k$, see Figure~\ref{fig:diagram-Jkn}. 
If $k=2$ or $k=3$ and $n\in \{6,7,8\}$, $J_{k,n}$ is a Dynkin diagram, 
in general, it gives rise to a Kac-Moody algebra. 
In the Dynkin cases, 
the categories ${\rm CM}(B_{k,n})$ have only finitely many 
indecomposable objects and they are known to correspond to positive real roots for the associated 
root system of $J_{k,n}$, \cite[Section 2]{JKS}. 
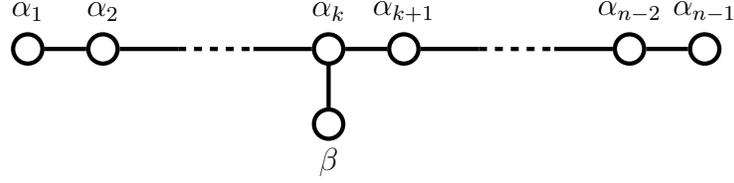
\begin{figure}[H]
\begin{center}
\begin{tikzpicture}[ultra thick]
\def\a{1}
\tikzset{dynkin/.style={circle,draw,minimum size=2mm}}
\path
(0,0)      node[dynkin] (N1) {} +(+90:.5) node{$\alpha_1$}
++(0:\a)   node[dynkin] (N2) {} +(+90:.5) node{$\alpha_2$}
++(0:\a)   coordinate (A1) 
++(0:\a)   coordinate (B1)
++(0:\a) node[dynkin] (N3) {} +(+90:.5) node{$\alpha_k$}
++(0:\a) node[dynkin] (N4) {} +(+90:.5) node{$\alpha_{k+1}$}
++(0:\a)   coordinate (A) 
++(0:\a)   coordinate (B)
++(0:\a) node[dynkin] (Nn-2) {} +(+90:.5) node{$\alpha_{n-2}$}
++(0:\a) node[dynkin] (Nn-1) {} +(+90:.5) node{$\alpha_{n-1}$};
\path 
(4*\a,-\a) node[dynkin] (Nn) {} +(-90:.5) node{$\beta$};
 
\draw  (N1)--(N2);
\draw (N3)--(N4);
\draw (N4)--(A);
\draw (N2)--(A1);
\draw (B1)--(N3);
\draw (B)--(Nn-2)--(Nn-1);
\draw[dashed] (A1)--(B1);
\draw[dashed] (A)--(B);
\draw (N3)--(Nn);
\end{tikzpicture}
\end{center}
\caption{The diagram of the root system $J_{k,n}$ associated with $\Gr(k,n)$, we write $\beta$ for $\alpha_n$.}
\label{fig:diagram-Jkn}
\end{figure}

In general, it is a very difficult problem to describe the structure of the category or to classify indecomposable 
modules in ${\rm CM}(B_{k,n})$. 
In contrast to the finite case, it is not clear how the 
correspondence between modules in the Grassmannian cluster categories 
and the root system arises. 
However, Jensen, King and Su \cite{JKS} suspect that the classes of rigid indecomposable 
modules indeed are roots for $J_{k,n}$. Evidence for this was given in the small rank cases 
for certain infinite cases in ~\cite{BBG}. 
The authors give a construction of rank 2 modules via short exact 
sequences, and find conditions that filtration factors of these modules have to fulfill. 
They also show that ${\rm CM}(B_{k,n})$ has at most 
$2{n \choose 6} {n-6 \choose k-3}$ 
(profiles of) rigid indecomposable rank 2 modules which correspond to real roots. 
Here, we also consider the imaginary roots and show in Theorem~\ref{thm:rank2-bound} (2) that 
{there are at least} 
\begin{align*}
N_{k,n} = \sum_{r=3}^{k} \left( \frac{2r}{3} \cdot p_1(r) +  2r \cdot p_2(r) + 4r \cdot p_3(r) \right) \cdot {n \choose 2r} {n-2r \choose k-r}
\end{align*} 
profiles of 
rigid rank 2 indecomposable modules in ${\rm CM}(B_{k,n})$, 
where $p_i(r)$ is the number of partitions $r=r_1+r_2+r_3$ such that $r_1,r_2,r_3 \in \ZZ_{\ge 1}$ and 
$|\{r_1,r_2,r_3\}|=i$. {Furthermore, every rank 2 indecomposable module where the rims of the filtration factors form three rectangular boxes is rigid.}

Moreover, any indecomposable rank 2 module corresponds to a root of $J_{k,n}$ and for 
$k=3$, we show that all rank 3 modules in ${\rm CM}(B_{3,n})$ map to roots for $J_{3,n}$.

\begin{thm}[Lemma~\ref{lm:rk2-condition} and Theorem~\ref{thm:rank3-in-CM-3-n}]
(1) Every indecomposable rank $2$ module in ${\rm CM}(B_{k,n})$ corresponds to a root for 
$J_{k,n}$. \\
(2) Every indecomposable module of rank $3$ in ${\rm CM}(B_{3,n})$ corresponds to 
a root for $J_{3,n}$. 
\end{thm}

Recall that the modules in ${\rm CM}(B_{k,n})$ can be filtered by rank 1 modules which in turn correspond 
to $k$-subsets of $[n]=\{1,\ldots,n\}$ (see Section~\ref{ssec:CM-setup}). 
The {\bf profile} of a module is the collection of $k$-subsets corresponding to rank 1 modules in the generic 
filtration (\cite[\S 8]{JKS}). If the filtration of $M$ has rank 1 factors $I_1,\dots, I_m$, 
for some $m>0$, where the rank 1 module $I_m$ is a submodule of $M$, 
we write $M=I_1\mid I_2\mid \dots \mid I_m$. We also write 
$P_M$ for the profile of $M$. 

We denote by $\mathcal{P}_{k,n}$ the set of profiles of 
indecomposable modules in ${\rm CM}(B_{k,n})$. Its elements have $m$ rows with $k$ entries in 
$[n]$. 

Let $P \in \mathcal{P}_{k,n}$ be a profile with $m$ rows. 
Write $P=(P_{ij})$, $1\le i\le m$, $1\le j\le k$, for the profile $P$ where the entries 
in every row are written in increasing order. 
Then $P$ is called {\bf weakly column decreasing} if for every $j\in [k]$ and every $i\in [m-1]$, 
$P_{i,j}\ge P_{i+1,j}$. 
We call $P$ {\bf canonical}, if $P$ is weakly column decreasing and if, in addition,  
$P_{m,j}\ge P_{1,j-1}$ for all $j\in[2,k]$. 
We write $C_{k,n}^{\rm re}$ to denote the set of all canonical profiles in $\mathcal{P}_{k,n}$ such that 
the corresponding module gives rise to a real root for $J_{k,n}$ (see Section~\ref{ssec:roots-Jkn}).

With this notion, we are able to prove the following.  
\begin{thm}[Theorem \ref{thm:k=3_rk=3_real_canonical}]
\label{thm:k=rk=3_real_rigid_canonical}
If $M$ is an indecomposable rank $3$ module in  ${\rm CM}(B_{3,n})$ such that 
$M$ corresponds to a real root for $J_{3,n}$, 
then the profile $P_M$ is a cyclic permutation of a canonical profile (i.e., a 
cyclic permutation of the rows of the profile of $M$ is canonical). 
\end{thm}

We expect that Theorem~\ref{thm:k=rk=3_real_rigid_canonical} is true for all indecomposable 
modules corresponding to real roots.

\begin{conj} [Conjecture \ref{conj:real-roots-interlacing}]
If $M\in{\rm CM}(B_{k,n})$ is rigid indecomposable and corresponds to a real root for $J_{k,n}$, 
then $P_M$ is a cyclic permutation of a canonical profile. 
\end{conj}

We have the following result about modules corresponding to imaginary roots.

\begin{thm} [Theorem \ref{thm:k=3-mod-imaginary}]\label{thm:modules-imaginary}
Suppose that the indecomposable module $M \in {\rm CM}(B_{3, n})$ 
corresponds to an imaginary root of $J_{3,n}$. 
Then $P_M$ is one of the following:
\begin{align*}
\thfrac{ i_1   i_5   i_7}{ i_3   i_6   i_9}{ i_2   i_4   i_8 },\thfrac{ i_2   i_6   i_8}{ i_1   i_4   i_7}{ i_3   i_5   i_9 },\thfrac{ i_3   i_7   i_9}{ i_2   i_5   i_8}{ i_1   i_4   i_6 },\thfrac{ i_1   i_4   i_8}{ i_3   i_6   i_9}{ i_2   i_5   i_7 },\thfrac{ i_2   i_5   i_9}{ i_1   i_4   i_7}{ i_3   i_6   i_8 },\thfrac{ i_1   i_3   i_6}{ i_2   i_5   i_8}{ i_4   i_7   i_9 },\thfrac{ i_2   i_4   i_7}{ i_3   i_6   i_9}{ i_1   i_5   i_8 },\thfrac{ i_3   i_5   i_8}{ i_1   i_4   i_7}{ i_2   i_6   i_9 },\thfrac{ i_4   i_6   i_9}{ i_2   i_5   i_8}{ i_1   i_3   i_7 },
\thfrac{i_1 i_4 i_7}{i_3 i_6 i_9}{i_2 i_5 i_8}, \thfrac{i_2 i_5 i_8}{i_1 i_4 i_7}{i_3 i_6 i_9}, \thfrac{i_3 i_6 i_9}{i_2 i_5 i_8}{i_1 i_4 i_7},
\end{align*}
where $1\le i_1<i_2<\cdots <i_9\le n$. 
\end{thm} 

Note that for $n=9$, there are exactly 12 indecomposable rank 3 modules 
corresponding to an imaginary root. 
The first 9 in the list are rigid and the last three are non-rigid; we show that in this case, there are 
exactly 225 rigid indecomposable rank 3 modules corresponding to roots of $J_{3,9}$. 

For arbitrary $n>9$, we expect that modules such as the last three in the list of 
Theorem~\ref{thm:modules-imaginary} are always non-rigid. 

A module $M' \in {\rm CM}(B_{k,n})$ is said to be an $a$-shift of the module 
$M \in {\rm CM}(B_{k,n})$ if the profile of $M'$ is obtained from the profile $P$ of $M$ 
by adding a fixed number (mod $n$) 
to every entry of $P$ (Definition~\ref{def:shift}). 

\begin{thm}[Theorem~\ref{thm:rigid-ind-rk3-39} and Corollary~\ref{cor:count-rigid-3-9}]
\label{thm:225-rigid-ind-rk3-39}
Consider the category ${\rm CM}(B_{3,9})$.  Every rigid indecomposable rank $3$ module 
maps to a root of $J_{3,9}$. 
Among the rigid indecomposable rank $3$ modules, $216$ correspond to a real root and the profile of 
each of them is a cyclic permutation of a canonical one. 
Furthermore, there are $9$ modules mapping to an imaginary root and their profiles are 
all a shift of $\thfrac{157}{369}{248}$. 
\end{thm}

Theorem~\ref{thm:225-rigid-ind-rk3-39} is proved by studying the tubes of the Auslander-Reiten quiver of 
${\rm CM}(B_{3,9})$ containing rigid rank 3 modules.

We also prove the 
following result and its dual version (Theorem 
\ref{thm:AR-sequence-3-peaks-dual}). 
If $I$ is a $k$-subset of $[n]$, we write $L_I$ for the corresponding rank 1 module. If $I$ and $J$ are 
two $k$-subsets, we write $L_I\mid L_J$ for the rank 2 module with submodule $L_J$ and quotient 
$L_I$. 

\begin{thm} [{Theorem \ref{thm:AR-sequences-3-peaks}}]
We consider the category ${\rm CM}(B_{3,n})$. Let $I=\{i_1, i_2, i_3\}$ be a $3$-subset where 
$i_j<i_{j+1}-1$ for $j=1,2,3$ (reducing modulo $n$). Let $X=\{i_1+1, i_2+1, i_3+1\}$ and 
$Y=\{i_1+2, i_2+2, i_3+2\}$. 
Then there is an Auslander-Reiten sequence 
\[
L_I \hookrightarrow M \twoheadrightarrow \frac{L_X}{L_Y}
\]
where $M$ is a rank $3$ module. If $M$ is indecomposable, its profile is $X\mid I\mid Y$. 
\end{thm}

The paper is organized as follows. In Section \ref{sec:Grassmannian cluster categories}, we recall key definitions and results about Grassmannian cluster categories. 
In Section \ref{sec:canonical profiles and interlacing property}, we study profiles and show that a weakly 
column decreasing profile is canonical if and only if any two rows of its profile are interlacing. 
In Section~\ref{sec:rk2-rigid}, we give a lower bound for the number of 
 indecomposable rank 2 modules ${\rm CM}(B_{k,n})$ corresponding to roots for $J_{k,n}$. 
In Section~\ref{sec:subspace configurations}, we concentrate on rank 3 modules and study the 
subspace configurations. In Section~\ref{sec:Auslander-Reiten quiver}, we provide short exact sequences  
to describe the structure of the Auslander-Reiten quiver for ${\rm CM}(B_{3,n})$. We use this to 
determine the number of rigid indecomposable rank 3 modules in the infinite case ${\rm CM}(B_{3,9})$. 

\subsection*{Acknowledgments} 
We thank Christof Geiss, Alastair King, Matthew Pressland, Markus Reineke, Andrei Smolensky, Sonia Trepode, and Michael Wemyss for helpful discussions.  
K. B. was supported by a Royal Society Wolfson Fellowship. 
She is currently on leave from the University of Graz. D.B.\ was supported by the Austrian Science Fund Project Number P29807-35. J.-R.L. was supported by the Austrian Science 
Fund (FWF): M 2633-N32 Meitner Program and P 34602 Einzelprojekt. A.G.E. was supported by PICT(2017-2533) Agencia Nacional de Promoci\'on Cient\'ifica. 

\section{Grassmannian cluster categories} \label{sec:Grassmannian cluster categories}

In this section, we recall the definition of the Grassmannian cluster categories from~\cite{JKS} 
and some of their properties, see also~\cite{BBG}. Let $n\ge 4$ and recall that we always assume 
$1\le k\le \frac{n}{2}$.

\subsection{Cohen-Macaulay modules}\label{ssec:CM-setup}

Denote by $C=(C_0, C_1)$ the circular graph with 
vertex set $C_0=\ZZ_{n}$ clockwise around the circle, and with the edge set $C_1=\ZZ_n$, with edge 
$i$ joining vertices $i-1$ and $i$, see Figure \ref{fig:graph C and quiver Q_C}. 
Denote by $Q_C$ the quiver with the same vertex set $C_0$ and with arrows 
$x_i: i-1 \to i$, $y_i: i \to i-1$ for every $i \in C_0$, see Figure \ref{fig:graph C and quiver Q_C}.

\begin{figure}
\centering
    \begin{minipage}{.45\textwidth}
\begin{tikzpicture}[
    myedge/.style={thick,draw=black, postaction={decorate},
      decoration={markings,mark=at position .6 with {\arrow[black]{triangle 45}}}},
    myshorten/.style={shorten <= 2pt, shorten >= 2pt}]
    
    \node (A1) at (360/6*1:2cm) [circle,fill,inner sep=2pt, label=above:6] {}; 
    \node (A2) at (360/6*2:2cm) [circle,fill,inner sep=2pt, label=above:5] {}; 
    \node (A3) at (360/6*3:2cm) [circle,fill,inner sep=2pt, label=left:4] {};
    \node (A4) at (360/6*4:2cm) [circle,fill,inner sep=2pt, label=below:3] {}; 
    \node (A5) at (360/6*5:2cm) [circle,fill,inner sep=2pt, label=below:2] {}; 
    \node (A6) at (360/6*6:2cm) [circle,fill,inner sep=2pt, label=right:1] {};

  \draw (A1) to [bend right=25] node[midway,above] {6} (A2);
  \draw (A2) to [bend right=25] node[midway,left] {5} (A3);
  \draw (A3) to [bend right=25] node[midway,left] {4} (A4);
  \draw (A4) to [bend right=25] node[midway,below] {3} (A5);
  \draw (A5) to [bend right=25] node[midway,right] {2} (A6);
  \draw (A6) to [bend right=25] node[midway,right] {1} (A1);

\end{tikzpicture}
\end{minipage}
\centering
    \begin{minipage}{.45\textwidth}
\begin{tikzpicture}[
    myedge/.style={thick,draw=black, postaction={decorate},
      decoration={markings,mark=at position .6 with {\arrow[black]{triangle 45}}}},
    myshorten/.style={shorten <= 2pt, shorten >= 2pt}]
    
    \node (A1) at (360/6*1:2cm) [circle,fill,inner sep=2pt, label=above:6] {}; 
    \node (A2) at (360/6*2:2cm) [circle,fill,inner sep=2pt, label=above:5] {}; 
    \node (A3) at (360/6*3:2cm) [circle,fill,inner sep=2pt, label=left:4] {};
    \node (A4) at (360/6*4:2cm) [circle,fill,inner sep=2pt, label=below:3] {}; 
    \node (A5) at (360/6*5:2cm) [circle,fill,inner sep=2pt, label=below:2] {}; 
    \node (A6) at (360/6*6:2cm) [circle,fill,inner sep=2pt, label=right:1] {};

  \draw[<-] (A1) to [bend right=25] node[midway,above] {$x_6$} (A2);
  \draw[<-] (A2) to [bend right=25] node[midway,left] {$x_5$} (A3);
  \draw[<-] (A3) to [bend right=25] node[midway,left] {$x_4$} (A4);
  \draw[<-] (A4) to [bend right=25] node[midway,below] {$x_3$} (A5);
  \draw[<-] (A5) to [bend right=25] node[midway,right] {$x_2$} (A6);
  \draw[<-] (A6) to [bend right=25] node[midway,right] {$x_1$} (A1);
  
  \draw[<-] (A2) to [bend right=25] node[midway,below] {$y_6$} (A1);
  \draw[<-] (A3) to [bend right=25] node[midway,right] {$y_5$} (A2);
  \draw[<-] (A4) to [bend right=25] node[midway,right] {$y_4$} (A3);
  \draw[<-] (A5) to [bend right=25] node[midway,above] {$y_3$} (A4);
  \draw[<-] (A6) to [bend right=25] node[midway,left] {$y_2$} (A5);
  \draw[<-] (A1) to [bend right=25] node[midway,left] {$y_1$} (A6);

\end{tikzpicture}
\end{minipage}

\caption{The graph $C$ and the quiver $Q_C$, $n=6$.}
\label{fig:graph C and quiver Q_C}
\end{figure}
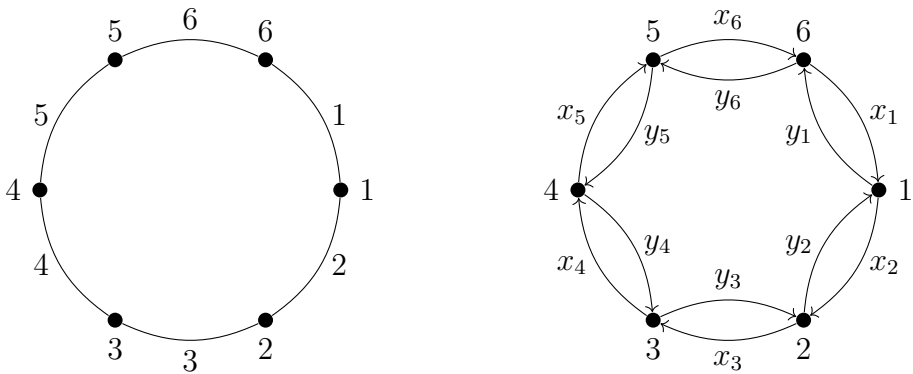

Denote by $B_{k,n}$ the quotient of the complete path algebra $\widehat{\CC Q_C}$ by the ideal generated 
by the $2n$ relations $x y = y x$, $x^{k} = y^{n-k}$, where $x, y$ are arrows of the form $x_i, y_j$ 
for appropriate $i,j$ (two relations for each vertex of $Q_C$). 

The center $Z$ of $B_{k,n}$ is the ring of formal power series $\CC[[t]]$, where $t = \sum_{i=1}^n x_i y_i$. A (maximal) Cohen-Macaulay $B_{k,n}$-module is given by a representation $\{M_i: i \in C_0\}$ of $Q_C$, where each $M_i$ is a free $Z$-module of the same rank (cf. \cite[Section 3]{JKS}).

\begin{definition}[{\cite[Definition 3.5]{JKS}}]
For any $B_{k,n}$-module $M$ and $K$ the field of fractions of $Z$, the {\bf rank} 
of $M$, denoted by ${\rm rk}(M)$,  is defined 
to be ${\rm rk}(M) = {\rm len}(M \otimes_Z K)$. 
\end{definition}

Jensen, King, and Su proved that the category ${\rm CM}(B_{k,n})$ 
is an additive categorification of the cluster algebra structure on $\CC[\Gr(k,n)]$. 

The category ${\rm CM}(B_{k,n})$ is exact and Frobenius with projective-injective objects given by the $B_{k,n}$ projective modules, and it has an Auslander-Reiten quiver (\cite[Remark 3.3]{JKS}). We 
denote by $\tau(M)$ the Auslander-Reiten translation of $M$ and by $\tau^{-1}(M)$ the inverse 
Auslander-Reiten translation of $M$. 

\begin{remark}\label{rem:rigid-tau}
A module $M$ in ${\rm CM}(B_{k,n})$ is {\bf rigid} if $ \Ext^1_{{\rm CM}} (M,M)=0 $. Since ${\rm CM}(B_{k,n})$ is Frobenius, \cite[Corollary 3.7]{JKS} we have
\[ 
\Ext^i_{{\rm CM}} (M,N)= \underline{\Hom}_{\rm CM}(M,\Omega^{-i}N), 
\] 
where $\underline{\Hom}_{\rm CM}(M,\Omega^{-i}N)$ is the stable space of $\Hom_{\rm CM}(M,\Omega^{-i}N)$. 
By construction, $\Omega^{-1}$ is the shift for the triangulated category $\underline{\rm CM}(B_{k,n})$ (underline 
indicating the stable category). Note that $\tau=\Omega^{-1}$ as the category 
$\underline{\rm CM}(B_{k,n})$ is 2-Calabi--Yau.
It follows that a non-projective module $M$ in ${\rm CM}(B_{k,n})$ is rigid if and only if $\tau (M)$ is rigid, as $\tau$ is 
an autoequivalence for the triangulated category $\underline{\rm CM}(B_{k,n})$.
\end{remark}

A special class of objects of ${\rm CM}(B_{k,n})$ are the rank 1 modules which are known to be rigid, \cite[Proposition 5.6]{JKS}. 

\begin{definition}[{\cite[Definition 5.1]{JKS}}]\label{def:rank1}
For any $k$-subset $I$ of $C_1$, a rank 1 module $L_I$ in ${\rm CM}(B_{k,n})$ 
is defined by
\begin{align*}
L_I=(U_i, i \in C_0; x_i, y_i, i \in C_1),
\end{align*}
where $U_i = \CC[[t]]$, $i \in C_0$, $e_i$ acts as the identity on $U_i$ and $e_iU_j=0$ for $i\neq j$, and  
\begin{align*}
x_i: U_{i-1} \to U_i \text{ is given by multiplication by $1$ if $i \in I$, and by $t$ if $i \not\in I$,} \\
y_i: U_{i} \to U_{i-1} \text{ is given by multiplication by $t$ if $i \in I$, and by $1$ if $i \not\in I$.}
\end{align*}
\end{definition}

By \cite[Proposition 5.2]{JKS}, every rank 1 module is isomorphic to $L_I$ for some 
$k$-subset $I$ of $[n]$. So there is a bijection between the rank 1 modules in ${\rm CM}(B_{k,n})$ 
with the $k$-subsets of $[n]$ and the cluster variables of $\CC[\Gr(k,n)]$ which are Pl\"ucker coordinates. 

It is convenient to represent the module $L_I$ by a lattice diagram, see 
Figure~\ref{fig:lattice diagram of L156 in B39}. The spaces $U_0, \ldots, U_n$ are represented by 
columns from left to right and $U_0$ and $U_n$ are identified. The vertices in each column correspond 
to the natural monomial $\CC$-basis of $\CC[t]$. 
The column corresponding to $U_{i+1}$ is displaced half a step vertically downwards (resp.\ upwards) in 
relation to $U_i$ if $i + 1 \in I$ (resp. $i+1 \not\in I$).

\begin{figure}
\includegraphics[width=6cm]{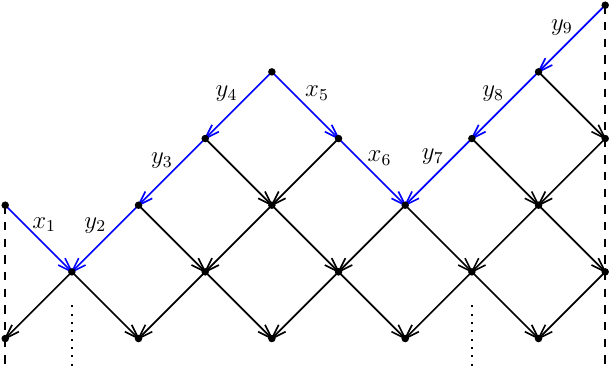}
\caption{The lattice diagram of $L_{\{1,5,6\}}$ in ${\rm CM}(B_{3,9})$ 
with its rim (blue arrows, with labels $x_1,y_2,y_3,y_4,x_5,x_6,y_7,y_8,y_9$).}
\label{fig:lattice diagram of L156 in B39}
\end{figure}

The upper boundary of the lattice diagram of $L_I$ is called the {\bf rim of $L_I$} (we usually 
omit the arrows when drawing the rim). The $k$ subset $I \subset [n]$ of the rank $1$ module $L_I$ can be read off as the set of labels on 
the edges 
going down to the right which are on the rim of $L_I$, i.e., the labels of the $x_i$'s appearing in the rim. 
For simplicity, we usually do not draw the labels of the rim in the figures. 

We recall the notion of peaks and valleys of rank 1 modules from~\cite{bb16}. 

\begin{definition}\label{def:peak}
If $I$ is a $k$-subset of $[n]$, the set of {\bf peaks of $L_I$} is the set 
$\{i\mid i\notin I, i+1\in I\}$. 
The set of {\bf valleys of $L_I$} is $\{i\mid i\in I, i+1\notin I\}$. 
\end{definition}

The rank 1 modules can be viewed as building blocks for the category as every module in 
${\rm CM}(B_{k,n})$ has a filtration with factors which are rank $1$ modules (cf. \cite[Proposition 6.6]{JKS}).
Let $M$ be a rank $m$ module in ${\rm CM}(B_{k,n})$ with factors $L_{I_1},\dots, L_{I_m}$ in its generic filtration, 
where $L_{I_m}$ is a submodule of $M$. 
We write $M=L_{I_1}|L_{I_2}| \cdots |L_{I_m}$ or 
$M=\begin{small}\begin{array}{c} L_{I_1}\\ \hline \vdots \\ \hline L_{I_m} \end{array}\end{small}$\,.
The ordered 
collection of $k$-subsets $I_1, \ldots, I_m$ in the generic filtration of $M$ is called the {\bf profile} of $M$, denoted $P_M$. We write 
$P_M=\begin{small}\begin{array}{c} I_1\\ \hline \vdots \\ \hline I_m \end{array}\end{small}$ 
or $P_M=I_1| \cdots |I_m$ if $M$ has a filtration having factors 
$L_{I_1}, \ldots, L_{I_m}$ (in this order). 
We sometimes write $M=P_M$ to indicate that $M$ is a module with profile $P_M$. 
Note that in general, such a filtration is not unique, but in case $M$ is rigid, the filtration is unique 
in the sense that it gives a canonical ordered set of rank 1 composition factors (see next subsection, \cite[Definition 6.4]{JKS}, and \cite[Proposition 6.6]{JKS} for more details on the uniqueness of the profile of the Cohen-Macaulay modules which are given as subspace configurations). 


\subsection{Subspace configurations}\label{ssec:subspace}

Here, we follow the set-up of~\cite[Section 6]{JKS}. 
Let $M$ be a module in ${\rm CM}(B_{k,n})$ and consider a filtration of $M$. We draw the lattice 
diagrams of all the rank 1 modules appearing in the filtration, in the filtration order. The 
rims of the filtration factors in this picture, together with the multiplicities of the elements in the monimial 
bases are called the set of {\bf contours of $M$}. 
For indecomposable modules, contours are close-packed in the sense that there can be no 
``walk'' (unoriented path) going around between two layers of the contours of $M$. 
The result is a lattice diagram where the multiplicities of the elements of the monomial basis are equal 
to the rank of $M$ below the lowest contour and equal to 1 between the top two contours. 
For the representation $M$, connected regions with constant multiplicity in the lattice diagram 
correspond to isomorphic vector spaces, e.g., we can consider the region below the lowest contour to be 
$\CC^{\rk M}$. This view-point gives a subspace configuration associated with $M$, drawn as a poset, with 
the convention that an edge $i\-- j$ with $i<j$ denotes an inclusion of an $i$-dimensional space into 
a $j$-dimensional space. 
With slight abuse of notation, we also call the poset of a subspace 
configuration a subspace configuration. The poset (or subspace configuration) is a graph where the 
labels on the vertices are viewed as multiplicities, they are elements of $\{1,2,\dots, \rk M\}$, the 
multiplicity $\rk M$ only appears in one vertex. 

The poset can thus be viewed as a representation for the underlying graph, with the orientations on the 
edges coming from the inclusions in $M$, i.e., pointing from the smaller to the larger numbers. 

\begin{example}\label{ex:contours-poset}
Let $M$ be a module in ${\rm CM}(B_{3,9})$ with profile $169\mid 147\mid 358$. 
Figure \ref{fig:subspace-poset-simplify} (A) shows the contours of $M$. 
In (A), below the first rim, every region has an upper boundary which is a part of the rim of the 
quotient $M$ by the rank 1 modules above it. 
In (B), we show the poset of the subspace configuration of $M$. This graph can be simplified to 
(C) as we explain below. 
\end{example}

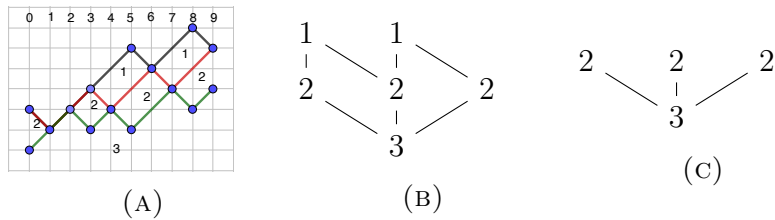
\begin{figure}
\begin{center} 
\begin{subfigure}{0.23\textwidth}
\includegraphics[width=3cm]{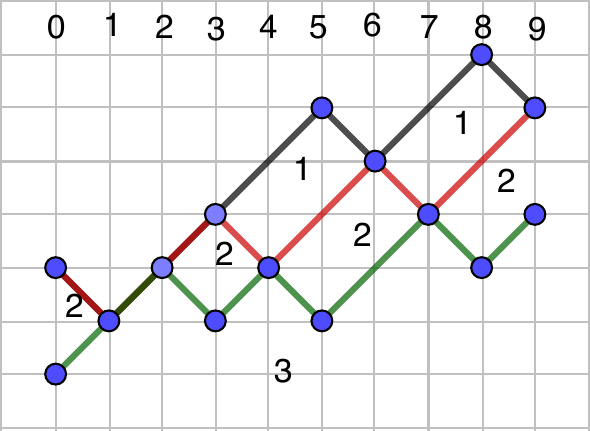}
\caption{}
\end{subfigure}
\begin{subfigure}{0.23\textwidth}
\begin{tikzpicture}[scale=0.6]
\tikzstyle{state}=[thick,minimum size=2mm]
		\node [state] (1) at (-10, 3) {$1$};
		\node [state] (2) at (-8, 3) {$1$};
		\node [state] (3) at (-10, 1.75) {$2$};
		\node [state] (4) at (-8, 1.75) {$2$};
		\node [state] (5) at (-6, 1.75) {$2$};
		\node [state] (6) at (-8, 0.5) {$3$};
		\draw (1) to (3);
		\draw (1) to (4);
		\draw (2) to (4);
		\draw (2) to (5);
		\draw (3) to (6);
		\draw (4) to (6);
		\draw (5) to (6); 
\end{tikzpicture}
\caption{} 
\end{subfigure}  
\begin{subfigure}{0.23\textwidth}
\begin{tikzpicture}[scale=0.6]
\tikzstyle{state}=[thick,minimum size=2mm]
		\node [state] (3) at (-10, 1.75) {$2$};
		\node [state] (4) at (-8, 1.75) {$2$};
		\node [state] (5) at (-6, 1.75) {$2$};
		\node [state] (6) at (-8, 0.5) {$3$};
		\draw (3) to (6);
		\draw (4) to (6);
		\draw (5) to (6); 
\end{tikzpicture}
\caption{} 
\end{subfigure}  
\end{center}
\caption{The subspace configuration (B) corresponds to the graph of contours (A). 
The subspace configuration (B) simplifies to (C).}
\label{fig:subspace-poset-simplify}
\end{figure}


The following lemma follows from the discussion below Definition 6.4 in \cite{JKS}.
\begin{lemma} \label{lem:if M is indecomposable then the subspace configuration is indecomposable}
If $M$ is indecomposable, then the subspace configuration is indecomposable, i.e., the poset diagram is an indecomposable representation for the oriented poset graph. 
\end{lemma} 

Using this, we get that if the 
poset diagram is a direct sum of two non-zero representations for the poset graph, it 
induces a direct sum decomposition of $M$ in ${\rm CM}(B_{k,n})$.

To see whether the subspace configuration is indecomposable, it is convenient to simplify 
the poset diagrams. 

The poset diagrams of subspace configurations can be simplified as follows.
A subspace can be identified 
as the intersection of several subspaces it maps into and may be omitted: 
for example, in Figure~\ref{fig:subspace-poset-simplify} (B), 
each of the 1-dimensional subspaces 
has to map in two different 2-dimensional subspaces of the 3-dimensional space at the bottom. 
In each case, these two different 2-dimensional subspaces intersect in a subspace of dimension 1. 
Hence this 1-dimensional subspace mapping into them is already determined and can be omitted from the diagram. 

If after applying simplifications, the underlying graph of the poset is the graph of a Dynkin diagram, 
we know that if the poset diagram is indecomposable, it has to be a positive root for it. 

Under the above rules, 
the subspace configuration in Figure \ref{fig:subspace-poset-simplify} (B) from 
Example~\ref{ex:contours-poset} 
simplifies to (C). This is of Dynkin type $D_4$ and the simplified poset is not a root for $D_4$, hence 
the poset diagram is not indecomposable.

\subsection{Auslander-Reiten translations and shifts}\label{sec:AR-and-shift}

By the symmetry of the algebra $B_{k,n}$, adding a fixed number 
to every element in every $k$-subset of the profile of a module 
and changing the linear maps in the representation accordingly preserves indecomposability and rigidness.

\begin{definition}\label{def:shift} 
Let $M$ and $M'$ be indecomposable modules in ${\rm CM}(B_{k,n})$, with profiles $P$ and 
$P'$. If there is a number $a\in [n]$ such that $P'$ is obtained by adding $a$ to each entry of $P$ 
(reducing modulo $n$), we say that $M'$ is a {\bf shift of} $M$ {\bf by} $a$ 
and that the profile $P'$ is a 
{\bf shift of} $P$ ({\bf by} $a$). 
\end{definition}

\begin{lemma} \label{lem:shift-rigid}
Let $M$ and $M'\in {\rm CM}(B_{k,n})$ be indecomposable and let $M'$ be a shift of $M$. 
Then $M$ is rigid if and only if $M'$ is rigid. 
\end{lemma}

If the Auslander-Reiten translation leaves a module invariant, it cannot be rigid.

\begin{lemma} \label{lem:tauM=M_M-non-rigid} 
Let $M$ be a module in $\underline{{\rm CM}}(B_{k,n})$. If $\tau^{-1}(M)=M$, then $M$ is non-rigid. 
\end{lemma}

\begin{proof}
If $M=\tau^{-1}(M)$, then there is an almost split sequence 
\[
M\to E \to M 
\]
in ${\rm CM}(B_{k,n})$. This sequence does not split, so $M$ is not rigid. 
\end{proof}

\begin{lemma} \label{lem:tauM and M have the same profile implies that M is non-rigid}
For any indecomposable module $M$ in $\underline{{\rm CM}}(B_{k,n})$, if the profile of 
$\tau^{-1}(M)$ is the same as 
the profile of $M$, then $M$ is non-rigid. 
\end{lemma}

\begin{proof}
Suppose that $\tau^{-1}(M)$ and $M$ have the same profile. If $M$ is rigid, then $\tau^{-1}(M)=M$ since rigid indecomposable modules are uniquely determined by their profiles. By Lemma \ref{lem:tauM=M_M-non-rigid} , 
$M$ is non-rigid. This is a contradiction. Therefore $M$ is non-rigid.
\end{proof}

Consider the subcategory ${\rm Sub}\, Q_k$ of the module category   
of the preprojective algebra of type $A_{n-1}$ (on vertices $1,2,\dots,n-1$) 
consisting of modules with socle concentrated at the vertex $k$. In other words,  ${\rm Sub}\, Q_k$ 
is the exact subcategory of modules having injective envelope in the additive hull of the injective at 
$k$. It has been studied in the work \cite{gls} of Geiss-Leclerc-Schr\"oer on the cluster algebra 
structure for the coordinate ring of the affine open cell in the Grassmannian. 
There is a triangle equivalence between the 
stable categories $\underline{{\rm CM}}(B_{k,n})$ and $\underline{{\rm Sub}}\, Q_k$, \cite[Corollary 4.6]{JKS}.

\begin{proposition}\label{propos:AR-tubes} 
Let $\Gamma$ be the Auslander--Reiten quiver of $\underline{{\rm CM}}(B_{k,n})$ and let $C$ be a component of $\Gamma$. Then 
\begin{itemize}
\item $\Gamma =C =\mathbb{Z}\Delta/G$ where $\Delta$ is a Dynkin quiver and $G$ is an automorphism group of $\mathbb{Z}\Delta$ in the finite cases $(2,n), (3,6), (3,7)$ and $(3,8)$. 
\item  each $C$ is of the form $\mathbb{Z}A_{\infty} / G_{\lambda}$, where each $G_{\lambda}$ is a power of $\tau$, in the infinite cases.
\end{itemize}
\end{proposition}

\begin{proof}
Let $C$ be a connected component of $\Gamma$. 
The category ${\underline{\rm CM}}(B_{k,n})$ is periodic under $\tau$ by \cite[Section 3]{BBG}. 
So the triangle equivalent category $\underline{{\rm Sub}}(Q_k)$ is periodic under $\tau$. We also write $C$ for the connected component 
of the Auslander-Reiten component of $\underline{{\rm Sub}}(Q_k)$ corresponding to $C$. 
We can apply Liu's result \cite[Theorem 5.5]{Liu}, since 
$\underline{{\rm Sub}}(Q_k)$ is a Krull--Schmidt and Hom-finite $k$-category, to deduce that  $C$ is of the form $\mathbb{Z}\Delta/H$ where $H$ is an automorphism group of $\mathbb{Z}\Delta$ 
containing a power of the translation $\tau$ or $C$ is a stable 
tube, i.e., isomorphic to  $\mathbb{Z}A_{\infty} / \tau^m$ for some $m>0$.

Let $T$ be a cluster-tilting object in the stable category ${\underline{\rm CM}}(B_{k,n})$ (for example, 
as given in~\cite[Section 3]{BKM}, without the projective summands). Let $A$ be its endomorphism algebra. 
If the Auslander--Reiten quiver of ${\rm mod}\, A$ contains 
a finite connected component, then this is the only component and ${\rm mod} \, A$ is representation finite 
by \cite[Theorem IV.5.4]{ASS}. 
Then by \cite[Section 2.1]{KR}  we have the following equivalence 
${\rm mod} \, A \simeq {\underline{\rm CM}}(B_{k,n}) / (T)$ and hence $C$ is as claimed. 

If ${\rm mod} \, A$ does not have a finite component, $C$  has to be a tube by the above-mentioned result by Liu. 
\end{proof}

For further results on $\tau$ periodicity in  these categories the reader is referred to \cite{DL} and \cite{Keller}. From the previous proposition it follows that in the infinite cases each Auslander--Reiten sequence has a middle term with at most two indecomposable summands. 

We point out that while the Auslander--Reiten components for ${\rm CM}(B_{k,n})$ of 
infinite type  are always tubes, outside of the tame cases, there are many morphisms 
(those in the radical infinite) connecting the tubes. In these cases it is not true that 
the indecomposable rigid objects are 
precisely those sitting in the tubes at a level lower than the rank of the tube.

\subsection{A root system for the Grassmannian} \label{ssec:roots-Jkn} 

Recall the graph $J_{k,n}$ with vertices $1,2,\dots, n-1$ on a line and 
an additional vertex $n$ attached to the vertex $k$, see Figure~\ref{fig:diagram-Jkn}, as introduced in~\cite{JKS} in the 
study of ${\rm CM}(B_{k,n})$. 
This graph gives rise to a Kac-Moody root system (we also call this root system $J_{k,n}$). 
It is of Dynkin type $D_n$ for $k=2$ and 
of Dynkin type $E_6, E_7$,  and $E_8$, respectively, for $k=3$ and $n=6,7,$ and $8$, respectively. 
The root lattice of the Kac-Moody algebra of $J_{k,n}$ can be identified with the lattice 
$\ZZ^n(k):=\{x=(x_i)\in \ZZ^n\mid \mbox{ $k$ divides } \sum_i x_i\}$. 
We equip this with the inner product 
\begin{align} \label{eq:inner product for root lattice of Jkn}
(x,y)=\sum_{i=1}^n x_iy_i + \frac{2-k}{k^2} (\sum_{i=1}^n x_i) (\sum_{i=1}^n y_i)
\end{align}
and with the quadratic form $q(x):=(x,x)=\sum_{i=1}^n x_i^2 + \frac{2-k}{k^2} (\sum_{i=1}^n x_i)^2$. The roots of this root system satisfy $q(x)\le 2$, among them, real roots have $q(x)=2$ and imaginary ones have $q(x)<2$.

Let $e_1,\dots, e_n$ be the standard basis vector of $\mathbb{R}^n$. Define 
$\alpha_i=-e_i+e_{i+1}$ for $i\in [n-1]$ and $\beta=e_1+e_2+\dots+e_k$. Then the set 
$\{\alpha_1,\dots, \alpha_{n-1},\beta\}$ is a system of simple roots for the root system $J_{k,n}$. 
It will also be convenient to use $\alpha_n$ for $\beta$ and we switch freely between the two. The inner products of the simple roots are all equal 
to $2$ and 
\begin{align*}
(\alpha_i,\alpha_j) = \begin{cases}
-1, & \text{if } |i-j|=1, \\
0, & \text{otherwise,}
\end{cases} \qquad
(\beta,\alpha_i) = \begin{cases}
-1, & \text{if } i=k, \\
0, & \text{otherwise}.
\end{cases} 
\end{align*}

Denote by $s_{\alpha}$ the reflection about the hyperplane perpendicular to a root $\alpha$ of $J_{k,n}$ and we write $s_i = s_{\alpha_i}$ with $i \in [n-1]$ and $s_{\beta} = s_{\alpha_n}$. For any root $\alpha$ of $J_{k,n}$ and $v \in \mathbb{R}^n$, $s_{\alpha}(v)=v - 2 \frac{(v, \alpha)}{(\alpha, \alpha)} \alpha$. 
It is customary to abbreviate the scalar in front of $\alpha$ and to write 
$\langle v,\alpha^{\vee}\rangle:=2\frac{(v,\alpha)}{(\alpha,\alpha)}$ as we will do later. The {\bf Weyl group $W$ of $J_{k,n}$} is the group generated by the simple reflections.

In Section 2 in \cite{JKS}, there is a basis $e_1, \ldots, e_n$ of $\ZZ^n$ such that $\alpha_i=e_{i+1}-e_i$, $i \in [n-1]$, and $\beta=e_1+\ldots + e_k$. We write an element $x=\sum_{i=1}^n x_i e_i$ in $\ZZ^n$ as $x=(x_1, \ldots, x_n)$. 

\begin{lemma} \label{lem:action of Weyl group on root system}
For $i \in [n-1]$, $x=(x_1,\ldots, x_n)=\sum_{i=1}^n x_i e_i \in \ZZ^n$, we have that $s_i(x)$ is obtained from $x$ by interchanging $x_i$ and $x_{i+1}$, and $s_n(x) = (x_1+r,\ldots,x_k+r,x_{k+1},\ldots,x_n$, where $r = x_{k+1} + \ldots + x_n - 2 \frac{1}{k}\sum_{i=1}^n x_i$. 
\end{lemma}

\begin{proof}
For $i \in [n-1]$, by (\ref{eq:inner product for root lattice of Jkn}), we have that
\begin{align*}
s_i(\sum_{j=1}^n x_je_j) & =
 \sum_{j=1}^n x_je_j - 2 \frac{(\sum_{j=1}^n x_je_j, \alpha_i)}{(\alpha_i, \alpha_i)} \alpha_i \\
& = \sum_{j=1}^n x_je_j - (x_{i+1}-x_i)(e_{i+1}-e_i) = \sum_{j \in [n]\setminus \{i,i+1\}} x_je_j + x_{i+1}e_{i} + x_{i}e_{i+1}.
\end{align*}
Therefore $s_i(x)$ is obtained from $x$ by interchanging $x_i$ and $x_{i+1}$. Similarly, by (\ref{eq:inner product for root lattice of Jkn}),
\begin{align*}
s_n(x_1,\ldots,x_n) & = \sum_{j=1}^n x_je_j - 2 \frac{(\sum_{j=1}^n x_je_j, \alpha_n)}{(\alpha_n, \alpha_n)} \alpha_n \\
& = \sum_{j=1}^n x_je_j - \left( \sum_{i=1}^k x_i + \frac{2-k}{k^2} \cdot (\sum_{i=1}^n x_i) \cdot k \right)(e_1+\ldots + e_k ) \\
& = (x_1+r,\ldots,x_k+r,x_{k+1},\ldots,x_n),
\end{align*} 
where $r = x_{k+1} + \ldots + x_n - 2 \frac{1}{k}\sum_{i=1}^n x_i$. 
\end{proof}

Jensen, King and Su define in~\cite[Sections 2,8]{JKS} 
a map from the modules of ${\rm CM}(B_{k,n})$ to the root lattice of $J_{k,n}$ 
by associating a module with its class in the Grothendieck group of ${\rm CM}(B_{k,n})$ (the Grothendieck group is identified with the root lattice of $J_{k,n}$). This works as follows.
Let $M$ be a module and $P=P_M$ its profile. For $i=1,\dots, n$, 
let $x_i$ be the multiplicity of $i$ in the union of the $k$-subsets in the profile $P$. 
Then let $x(M)=x(P_M)=(x_1,x_2,\dots, x_n)\in \ZZ^n$ be the vector of the multiplicities. By definition, 
$k\mid \sum_i x_i$, so $x(M)\in \ZZ^n(k)$ and there is an element in 
$\text{Span}_{\mathbb{R}}\{\beta, \alpha_1, \ldots, \alpha_{n-1}\}$ corresponding to 
$x(M)$. We denote this element by $\varphi(M)$ and define $x(\varphi(M))=x(M)$. 
We say that $M\in {\rm CM}(B_{k,n})$ corresponds to a real (respectively, imaginary) root if 
$\varphi(M)$ is a real (respectively, imaginary) root of $J_{k,n}$. 
This terminology is motivated by the fact that in the finite types, all indecomposable  
modules map to roots 
and by our results on indecomposable rank 2 and rank 3 modules.

\begin{example}\label{ex:profile-root}
Let $M$ be a module in ${\rm CM}(B_{3,9})$ whose profile is 
$P_M=\begin{small}\begin{array}{c} 359\\ \hline 258\\ \hline 147\\ \hline  146\end{array}\end{small}$. 
Then  $x(M) = (2,1,1,2,2,1,1,1,1)$, $q(M)=2$, and 
$\varphi(M) = 4 \beta + 2 \alpha_1 + 5 \alpha_2 + 8 \alpha_3 + 6 \alpha_4
 + 4 \alpha_5 + 3 \alpha_6 + 2 \alpha_7 + \alpha_8$. 
\end{example}

\section{Canonical profiles and interlacing property} \label{sec:canonical profiles and interlacing property}

Modules in ${\rm CM}(B_{k,n})$ corresponding to real roots are closely related to canonical profiles 
which we will discuss in this section. In particular, we will prove that a weakly column 
decreasing profile is canonical if and only if any two rows of the profile are interlacing. 

\begin{definition}
We say that two sets $J_1, J_2$ of integers are $r$-interlacing if 
$|J_{1} \backslash J_{2}|=|J_{2} \backslash J_{1}|=r$ and 
$J_{1}\backslash J_{2} = \{a_1,a_2,\ldots,a_r\}$, $J_{2} \backslash J_{1} = \{b_1,b_2,\ldots,b_r\}$, 
$a_i < a_{i+1}$, $b_i < b_{i+1}$, $i \in [r-1]$, either $a_1 < b_1<a_2<b_2 < \cdots < b_{r-1} <a_r < b_r$ 
or $b_1<a_1<b_2<a_2 < \ldots < a_{r-1}<b_r<a_r$. 
In particular, if $J_1=J_2$, the sets are $0$-interlacing. 
We call $J_1, J_2$ interlacing if they are $i$-interlacing for some $i \in \ZZ_{\ge 0}$. 

A profile $P\in\mathcal P_{k,n}$ with $k$-subsets $J_1,\dots, J_m$ is called {\bf interlacing} if 
for any $i\ne j$, $J_i$ and $J_j$ are interlacing. 
\end{definition}

The profile $P_M$ from Example~\ref{ex:profile-root} is interlacing, e.g., 
the first row and the second row of $P_M$ are $2$-interlacing, 
the third row and the fourth row of $P_M$ are $1$-interlacing. 

\begin{remark}
The definition of $i$-interlacing we use in this paper is called tightly $i$-interlacing in \cite{BBG}.
\end{remark}

\begin{remark} 
Observe that there are profiles where any two successive rows and the first and the last row are interlacing but 
where the profile is not interlacing. An example is the profile 
$\begin{small}
\begin{array}{c}1 \hskip.2cm 8\ 15 \ 20 \\ \hline 2 \ 11 \ 18 \ 23 \\
 \hline  5 \ 17 \ 21 \ 30 \\ \hline 4 \ 13 \ 19 \ 29 \end{array} 
\end{small}$
where rows 1 and 3 are not interlacing. 
\end{remark}

Recall from the introduction that if $P \in \mathcal{P}_{k,n}$ is a profile with $m$ rows, we write 
$P=(P_{ij})$, $1\le i\le m$, $1\le j\le k$, for the profile $P$ with all rows written in increasing order.

\begin{definition}\label{def:canonical-real-profile}
Let $P$ be a profile in $\mathcal{P}_{k,n}$ with $m$ rows. 

\begin{enumerate}
\item $P$ is called {\bf weakly column decreasing} if 
for every $j\in[k]$ and for every $i\in[m-1]$, we have $P_{i,j} \ge P_{i+1,j}$. 
\item $P$ is called {\bf canonical}, if it is weakly column decreasing and if for all $j\in[2,k]$, 
$P_{m,j} \ge P_{1, j-1}$.
\item If $P$ is canonical and $q(P)=2$, we say it is {\bf real}. We write 
$C_{k,n}^{\rm re} \subset \mathcal{P}_{k,n}$ for the set of all canonical profiles $P$ 
with $q(P)=2$.
\end{enumerate}

\end{definition}

\begin{example}
The profile $P=\begin{small}\ffrac{359}{258}{147}{146}\end{small}$ 
of $\mathcal P_{3,9}$ is canonical and $P\in C_{3,9}^{\rm re}$. 
\end{example}

\begin{definition}\label{def:cyclic-perm}
Let $P$ be a profile in ${\mathcal P}_{k,n}$ with $m$ rows. 
Then a profile $P'$ is a {\bf cyclic permutation} of $P$ if $P'$ is obtained by cyclically 
permuting the rows of $P$. 
\end{definition}

\begin{proposition} \label{prop:canonical-interlace} 
Let $P \in \mathcal{P}_{k,n}$ be a weakly column decreasing profile. 
Then $P$ is canonical if and only if any two rows of $P$ are interlacing.
\end{proposition}

\begin{proof}
For $m=1$  there is nothing to show. So assume $m\ge 2$. \\
($\Rightarrow$) Suppose that $P$ is canonical and $P$ has two or more rows. Write $P=(P_{ij})_{1\le i\le m, 1\le j\le k}$, $P_{ij} < P_{i,j+1}$, $i \in [m]$, $j \in [k-1]$. Consider any two rows
\[
\begin{array}{ccc}
P_{i_1,1} & \cdots & P_{i_1, k} \\
P_{i_2,1} & \cdots & P_{i_2, k}
\end{array}
\]
with $i_1<i_2$. Let $I_i:=\{P_{i_j,1},\ldots, P_{i_j,k}\}$ for $j=1,2$. These are both $k$-subsets of $[n]$ as $P$ is a profile.
Note that $|I_2\setminus I_1|=|I_1\setminus I_2|=:\ell$ for some $\ell\ge 0$. 

To check the interlacing property, we consider $\widetilde{I_i}:=I_i\setminus (I_1\cap I_2)$ for $i=1,2$. 
So we can write $\widetilde{I_1}=\{P_{i_1,j_1},P_{i_1,j_2}, \dots, P_{i_1,j_{\ell}} \}$ and 
$\widetilde{I_2}=\{P_{i_2,j_1'},P_{i_2,j_2'}, \dots, P_{i_2, j_{\ell}'} \}$ for 
$1\le j_1<j_2<\dots <j_{\ell}\le k$ and $1\le j_1'<j_2'<\dots <j_{\ell}'\le k$. 
Since $P$ is canonical, $j_r' \le j_r$ for $r=1,\dots,\ell$ and $j_r< j'_{r+1}$ for $r=1,\dots, \ell-1$.

The $\ell$-subsets $\widetilde{I_1}$ and $\widetilde{I_2}$ are disjoint subsets by construction. 
We need to see that they are $\ell$-interlacing, i.e., that 
$P_{i_2,j_1'}<P_{i_1,j_1}<P_{i_2,j_2'}<\dots<P_{i_1, j_{\ell}}$.

In case $j_r'=j_r$, we have $P_{i_2,j_r'}\le P_{i_1,j_r}$ by the weakly column decreasing condition and 
$P_{i_1, j_r}\le P_{i_2, j'_{r+1}}$ since $P$ is canonical. And 
since $\widetilde{I_1}$ and $\widetilde{I_2}$ are disjoint, 
both inequalities have to be strict.

So assume that $j_r'<j_r$. Then $P_{i_1, j_r'}<P_{i_1, j_r}$. By the weakly column decreasing condition, we have $P_{i_2, j_r'} \le P_{i_1,j_r'}$. Therefore $P_{i_2, j_r'}<P_{i_1, j_r}$. Furthermore, $P_{i_1, j_r}\le P_{i_2, j_{r+1}'}$ as $P$ is canonical and so, since 
$\widetilde{I_1}$ and $\widetilde{I_2}$ are disjoint, we have $P_{i_1,j_r}< P_{i_2, j_{r+1}'}$.

\vskip.5cm

($\Leftarrow$) 
Let $P\in \mathcal{P}_{k,n}$ with $m\ge 2$ rows. Write 
$P=(P_{ij})_{1\le i\le m, 1\le j\le k}$, $P_{ij} < P_{i,j+1}$, $i \in [m]$, $j \in [k-1]$. For some $1\le j<k$, 
consider the four elements in rows $1$ and $m$ of $P$ and in the two successive columns $j,j+1$:
\[
\begin{array}{cc}
P_{1,j} &  P_{1,j+1} \\
\ & \ \\
P_{m,j} & P_{m,j+1}. 
\end{array}
\]
We have to show that $P_{1,j}\le P_{m,j+1}$. Denote by $I_1, I_m$ the $k$-subsets of the first and last row 
respectively. 

\noindent
a) If $P_{1,j}=P_{m,j}$, then we have $P_{1,j}<P_{m,j+1}$ since 
$P_{1,j}=P_{m,j}<P_{m,j+1}$. 

\noindent
b) Let $P_{1,j+1}=P_{m,j+1}$. Then $P_{1,j}<P_{m,j+1}$ as $P_{m,j+1}=P_{1,j+1}>P_{1,j}$. 

\noindent 
So we have to consider the cases where $P_{1,j}\ne P_{mj}$ and $P_{1,j+1}\ne P_{m,j+1}$. 
Since $P$ is weakly column decreasing, this means that $P_{1,j}>P_{m,j}$ 
and $P_{1,j+1}>P_{m,j+1}$. 

\noindent
c) Let $P_{1,j}>P_{m,j}$ and $P_{1,j+1}>P_{m,j+1}$. Suppose that $P_{1,j}>P_{m,j+1}$. If $j=1$, 
then the first row and the last row are not interlacing. This is a contradiction. Therefore $j>1$. 
Now we have a set $\{P_{m,j}, P_{m,j+1}\}$ in which every element is smaller than $P_{1,j}$. 
We define this set $\{P_{m,j}, P_{m,j+1}\}$ to be the {\bf active set} and the element 
$P_{1,j}$ the {\bf active maximum}. 
We now check the elements $P_{1,j-1}, P_{1,j-2}, \ldots, P_{1,1}$ one by one as follows: 
we compare each of these elements with the smallest element in the active set and with further elements 
of row $m$. 
Denote the minimum in the active set by $A$. 
In the beginning, we have $A=P_{m,j}$. 
\begin{itemize}
\item 
If $P_{1,j-1}<A$ and $P_{1,j-1} \ne P_{m,j-1}$, then the elements in the active set 
$\{P_{m,j}, P_{m,j+1}\}$ lie between two consecutive numbers 
$P_{1,j-1}, P_{1,j}$ in $I_1 \setminus I_m$. 
This contradicts the assumption that the first and the last row are interlacing. 

\item 
If $P_{1,j-1}<A$ and $P_{1,j-1} = P_{m,j-1}$, we keep the active set and the active maximum and 
continue with $P_{1,j-2}$. 

\item 
If $P_{1,j-1}=A$, then we have $P_{m,j-1}<A=P_{1,j-1}$. 
We replace the active set by removing $A$ and adding $P_{m,j-1}$. As $P_{1,j}$ is still larger than the 
elements of the new active set, we keep the active maximum. We then continue with $P_{1,j-2}$. 

\item 
If $P_{1,j-1}>A$ and $P_{1,j-1} \not\in I_1 \cap I_m$, then $P_{m,j-1}<P_{m,j} = A<P_{1,j-1}$. 
We replace the active set by $\{P_{m,j-1}, P_{m,j}\} = \{P_{m,j-1}, A\}$ and replace the active 
maximum by $P_{1,j-1}$. We then continue with $P_{1,j-2}$. 

\item 
If $P_{1,j-1}>A$ and $P_{1,j-1} \in I_1 \cap I_m$, then $P_{m,j-1}<P_{m,j} = A<P_{1,j-1}<P_{1,j}$. 
We change the active set by removing $P_{1,j-1}$ (if it is contained in it - otherwise we do not remove 
anything from the active set) and by 
adding $P_{m,j-1}$. We keep the active maximum 
and continue with $P_{1,j-2}$. 
\end{itemize}

Continuing this procedure, the active set always has two or more elements. By construction, 
every element in the active set is smaller than the active maximum. 
We have one of the following situations which both contradict the assumption that $I_1, I_m$ are interlacing: 
\begin{itemize}
\item 
At some step, the elements in the active set lie between two consecutive elements in $I_1 \setminus I_m$. 

\item 
In the last step, after we check $P_{1,1}$, the elements in the active set are smaller than the 
active maximum element and this active maximum is the smallest element in $I_1 \setminus I_m$. 
\end{itemize}

We thus get $P_{1,j}\le P_{m,j+1}$ as claimed. 
\end{proof}

\section{Rank 2 rigid indecomposable modules in ${\rm CM}(B_{k,n})$} 
\label{sec:rk2-rigid}

In this section, we give a lower bound for the number of indecomposable modules of rank 
$2$ corresponding to roots in ${\rm CM}(B_{k,n})$.

By \cite[Section 5]{BBG}, every rank $2$ indecomposable module in ${\rm CM}(B_{k,n})$ 
corresponding to a real root is of the form $L_I\mid L_J$ where $I$ and $J$ are 
3-interlacing. 
From that we can deduce that there are at most $2 {n \choose 6} {n-6 \choose k-3}$ such rigid 
modules (as they come in pairs).

Recall from Subsection~\ref{ssec:subspace} that we can draw modules as lattice diagrams or as collections of rims, 
as in 
Figure~\ref{fig:subspace-poset-simplify}. For simplicity, we often do not write the numbers (the 
dimensions of the vector spaces) in the regions of the lattice diagrams. 

By abuse of notation, we will call a profile or the lattice diagram of a rank $m$ module 
a {\bf profile of rank $m$} or a {\bf lattice diagram of rank $m$}. 

Consider a lattice diagram of rank 2: it has two rims which may meet several times (in particular, they 
will do so if the module is indecomposable). Similarly, if we consider a lattice diagram of higher 
rank, any two successive rims may meet several times.

\begin{definition}  
Consider a rank $m$ module with filtration $L_{I_1}\mid L_{I_2}\mid \dots\mid L_{I_m}$. 
Let $R_1$ and $R_2$ be the rims of $L_{I_j}$ and $L_{I_{j+1}}$ for some $1\le j<m$.  

We call the non-empty regions formed by the two rims, between any two successive 
meeting points (reducing modulo $n$, if necessary) the {\bf quasi-boxes between} the two rims. 
In particular, we say that $R_1$ and $R_2$ {\bf form $n$ quasi-boxes} if there are $n$ quasi-boxes between them. 
If a quasi-box is of rectangular shape, we call it a {\bf box}. 

We define the {\bf size (or right-size) of a (quasi-)box} to be the sum of the sizes of the 
intervals in the part of $I_j$ corresponding to this box.

The {\bf co-size (or left-size) of a quasi-box} is the sum of the sizes of the intervals of 
$I_j^c$ corresponding to this box. 
\end{definition}

\begin{example}\label{ex:boxes}
Figure~\ref{fig:branching} shows an example for $(k,n)=(7,16)$. 
The $7$-subsets forming the rims are $I=\{4,5,8,10,13,14,16\}$ 
and $J=\{1,2,6,7 ,11,13,15\}$. 
There are four quasi-boxes, two of them are boxes. 
To illustrate size and co-size: 
The quasi-box between branching points 
$5$ and $10$ has size 2 and co-size 3, the quasi-box between branching points $10$ and $14$ 
has size 2 and co-size 2. 
\end{example}

\begin{figure}
\includegraphics[scale=.7]{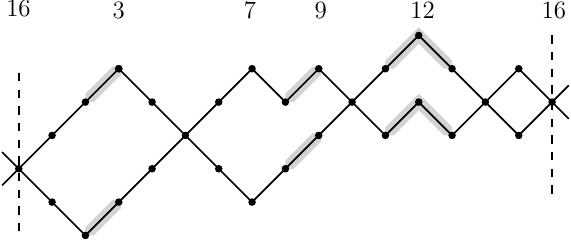}
\caption{Two rims formed by $7$-subsets of $n=16$, with branching points $\{5, 10,14,16\}$. 
Parallel segments are indicated by the shading.} 
\label{fig:branching} 
\end{figure}

The map from indecomposable modules to elements of the root system does not depend on the order 
of the factors in the filtration, it only depends on the elements of the $k$-subsets in its profile. 
We thus give this set a name.  
Let $P$ be the profile of an indecomposable module. The {\bf content of $P$}, ${\rm con}(P)$, 
is the multiset consisting of the elements of all $k$-subsets in $P$, for example, 
\[
{\rm con}
\left( \begin{small}\ffrac{147}{258}{358}{369} 
\end{small}\right)
=\{1,2,3,3,4,5,5,6,7,8,8,9\}.
\] 

If we have a rank 2 module $M=L_I\mid L_J$, adding a constant number to all entries in the involved 
$k$-subsets keeps some properties of the module (e.g., indecomposabilty, 
rigidness) but changes others, e.g., the content of $P_M$ (i.e., the labels appearing in 
$I\cup J$) will change in general.

We introduce operations on the profiles. 
The ``collapse'' operation removes parallel lines from profiles and the $a$-shift preserves the content as 
we will need to study different indecomposable modules which give rise to the same root.

\begin{definition}(Collapsing profiles and profile $a$-shift)\label{def:collapse-a-shift}

\begin{enumerate}

\item Let $P=I_1\mid I_2$ be a rank 2 profile in $\mathcal P_{k,n}$. 
Let $n'=n-| I_1^c\cap I_2^c|  -|I_1\cap I_2|$ and 
$\psi: [n]\setminus ((I_1^c\cap I_2^c) \cup (I_1\cap I_2))
\to [n']$ be the bijection respecting the order on these sets. We say that $\psi(I_1) \mid \psi (I_2)$ is the {\bf collapse}
of $I_1\mid I_2$.

\item Let $P=I_1\mid I_2\in \mathcal P_{k,n}$ be a rank 2 profile and $a\in[n]$. For $j=1,2$, 
define $I_j'$ as the set $ \psi^{-1}(\psi(I_j) + a \pmod {n'}) \cup (I_1 \cap I_2 ) $.  
Then $P' = I'_1\mid I'_2$ is called the {\bf $a$-shift } of $P$.
\end{enumerate}
 
\end{definition}

Observe that the elements of $I_1\cap I_2$ and the elements of 
$I_1^c\cap I_2^c$ are fixed under an $a$-shift for any $a\in [n]$. So for elements in $I_1\cap I_2$, the $a$-shift 
does not do anything whereas the elements in only one of the $k$-subsets might get moved to the other.

\begin{figure}
\[
\includegraphics[scale=.7]{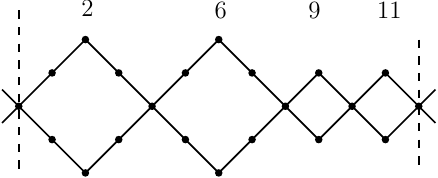}
\]
\caption{The collapse of the profile $I_1\mid I_2$ from Figure~\ref{fig:branching}.}\label{fig:collapsed}
\end{figure}

\begin{figure}
\includegraphics[scale=.7]{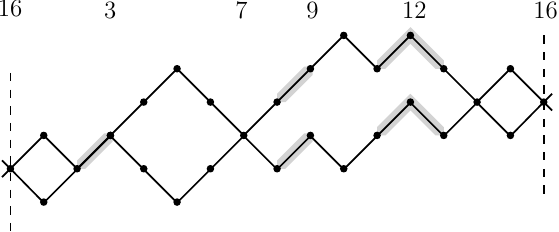}
\caption{The 2-shift of the profile from Figure~\ref{fig:branching}}\label{fig:2-shift}
\end{figure}

\begin{example} \label{ex:collapse-shift}
Let $P=I_1\mid I_2$ with $I_1=\{4,5,8,10,13,14,16\}$ and $I_2=\{1,2,6,7,11,13,15\}$. The collapse 
of $I_1\mid I_2$ is $\{3,4,7,8,10,12\}\mid\{1,2,5,6,9,11\}$, see Figure~\ref{fig:collapsed}. 
The $2$-shift $I_1\mid I_2$ is the profile 
$\{2,6,7,11,13,14,16\}\mid \{1,4,5,8,10,13,15\}$ it is shown in Figure~\ref{fig:2-shift}. 
\end{example}

The following is immediate from the definition of $a$-shifts. 
\begin{lemma}
Let $P$ be a rank $2$ profile in 
$\mathcal P_{k,n}$ and let $P'$ be an $a$-shift 
of $P$. Then ${\rm con}(P')={\rm con}(P)$. In particular, the root of $P$ is the same as the 
root of $P'$. Furthermore, if the $k$-subsets of $P$ are $r$-interlacing, then the $k$-subsets of $P'$ are 
also $r$-interlacing. 
\end{lemma}

Recall that we always assume $k\le n/2$.

\begin{lemma}\label{lm:rk2-condition}
Let $M$ be indecomposable of rank $2$ module in ${\rm CM}(B_{k,n})$ and assume that $M$ has a filtration $L_I\mid L_J$. 
Then the rims of $L_I$ and $L_J$ form at least three quasi-boxes. 
Furthermore, $q(M) \in \{2,0,-2,\ldots, 8-2k\}$ and $\varphi(M)$ is a root of $J_{k,n}$.
\end{lemma}

\begin{proof}
That the two rims have to form at least 
three quasi-boxes for indecomposability follows from the subspace configurations 
(cf.\ Remark 3.2 in~\cite{BBG}). Let $l$ be the number of quasi-boxes. 

Let $x(M)=x(P_M)=(x_1,\dots, x_n)$ be the vector of the multiplicities (Section~\ref{ssec:roots-Jkn}).  

Since $M$ has rank 2, the sum $\sum_{i=1}^nx_i$ is equal to $2k$ and $0\le x_i\le 2$ for all $i$. 
Also, since $I$ and $J$ form 
at least three quasi-boxes, we have 
$|I\setminus J|=|J\setminus I|\ge 3$ and so the number of $x_i$'s which are equal to 2 is at most $k-3$. 
But then $\sum x_i^2\le 4(k-3)+6$. 

Denote by $a$ (resp.\ $b$) the number of $2$'s (resp.\ $1$'s) in $\{x_i : i \in [n]\}$. 
Note that $b$ is even and $b\ge 6$ since there are at least three quasi-boxes. 
Then $\sum_i x_i^2=4a+b \le 4k-6$ and $\sum_i x_i=2a+b=2k$. 
Therefore $0 \le a \le k-3$, $b=2k-2a$. Hence 
$q(M) = \sum_{i=1}^n x_i^2 - 4(k-2) \in \{2, 0, -2, \ldots, 8-2k\}$.

First we consider the case when $q(M)=2$. That is, $a=k-3$ and $b=6$. We denote $x(M)=x(\varphi(M))=x=(x_1, \ldots, x_n)$. Up to Weyl group action, we may assume that 
$x(M) = (2,\ldots, 2, 1, \ldots, 1, 0, \ldots, 0)$ (see Lemma \ref{lem:action of Weyl group on root system}), where the number of $2$'s is $k-3$ and the number of $1$'s is $6$. Therefore 
\begin{align*} 
\varphi(M) = 2 \beta + \alpha_{k-2} + 2 \alpha_{k-1} + 3\alpha_k + 2 \alpha_{k+1} +  \alpha_{k+2}. 
\end{align*} 
By Lemma \ref{lem:action of Weyl group on root system}, we have that $s_ks_{k+1}s_{k+2}s_{k-1}s_{k}s_{k+1}s_{k-2}s_{k-1}s_{k}s_{\beta}(\varphi(M))=\beta$ which is a simple root. Therefore $\varphi(M)$ is a real root of $J_{k,n}$ by definition, 
\cite[Section 5.1]{Kac}.

Now we consider the case when $q(M)<2$. That is $a<k-3$ and $b=2k-2a$. Denote by 
$Q^+ = \oplus_{i=1}^n \mathbb{Z}_{\ge 0}\alpha_i$ the positive part of the root lattice of $J_{k,n}$ where 
we write $\alpha_n$ for $\beta$ for the moment. For $\alpha \in Q^+$, denote by ${\rm supp}(\alpha)$ 
the support of $\alpha$, i.e., the subdiagram of $J_{k,n}$ corresponding to the simple roots having 
non-zero coefficients in $\alpha$. 
Let $K = \{\alpha \in Q^+ \backslash \{0\} \mid  \langle \alpha, \alpha_i^{\vee} \rangle \le 0, i \in [n] \text{ and ${\rm supp}(\alpha)$ is connected} \}$. 
By \cite[Theorem 5.4]{Kac}, the set of all positive imaginary roots of $J_{k,n}$ is equal to 
$\cup_{w \in W} w(K)$. 
Up to Weyl group action, we may assume that $x(M) = (2,\ldots, 2, 1, \ldots, 1, 0, \ldots, 0)$ (see Lemma \ref{lem:action of Weyl group on root system}), 
where the number of $2$'s is $a$ and the number of $1$'s is $b$. Therefore 
\begin{align*}
\scalemath{0.9}{
\varphi(M) = 2 \beta + \alpha_{a+1} + 2 \alpha_{a+2} + \cdots + (k-a)\alpha_k + (k-a-1) \alpha_{k+1} + \cdots + 2 \alpha_{a+b-2} + \alpha_{a+b-1}. }
\end{align*}
We have that 
$\langle \varphi(M), \alpha_{a}^{\vee} \rangle = \langle \varphi(M), \alpha_{a+b}^{\vee} \rangle = -1$, 
$\langle \varphi(M), \alpha_{i}^{\vee} \rangle = 0$ for $i \in [n-1]\setminus \{a,a+b\}$ and 
$\langle \varphi(M), \beta^{\vee} \rangle = 4 - (k-a) \le 0$ since $a < k-3$. 
Moreover, ${\rm supp}(\varphi(M))$ is connected. Therefore $\varphi(M)$ is an imaginary root.
\end{proof}

Note that if $a=0$, $b=2k$ and $q(M)=8-2k$, whereas if 
$a=k-3$, we have $b=6$ and $q(M)=2$. 
So when $k=3$, all indecomposable rank 2 modules have $q(M)=2$ and for 
$k>3$, we can find rank 2 modules with the lower bound $q(M)=8-2k$ by taking two 
$k$-interlacing subsets.

\begin{theorem} \label{thm:rank2-bound}
Consider rank $2$ modules in ${\rm CM}(B_{k,n})$. \\
(1) Assume that $M$ has filtration $L_I\mid L_J$. Then $M$ is rigid indecomposable if  
the rims of 
$L_I$ and $L_J$ form three boxes. \\
(2) The number of profiles of  
rigid indecomposable rank $2$ modules in ${\rm CM}(B_{k,n})$ 
is at least  
\begin{align*}
N_{k,n}=\sum_{r=3}^{k} ( \frac{2r}{3} \cdot p_1(r) +  2r \cdot p_2(r) + 4r \cdot p_3(r)) \cdot {n \choose 2r} {n-2r \choose k-r},
\end{align*}
where $p_i(r)$ is the number of partitions $r=r_1+r_2+r_3$ such that 
$r_1,r_2,r_3 \in \ZZ_{\ge 1}$ and $|\{r_1,r_2,r_3\}|=i$.
\end{theorem}

\begin{proof}
Note that there are no indecomposable rank 2 modules for $k<3$ and $N_{k,n}=0$ in this case. 
So we can assume $k\ge 3$. 

To prove part (1), by Corollary~\ref{cor:reducing}, we can assume 
that $I\cap J=I^c\cap J^c=\emptyset$ (all parallel segments in the two rims can be removed). 
By Lemma~\ref{lm:rk2-condition}, the two rims have to form at least three quasi-boxes for a rank 2 module to be 
indecomposable. 

So assume  
that $L_I$ and $L_J$ form exactly three boxes. Since $I\cap J=I^c\cap J^c=\emptyset$, the syzygy $\Omega(M)$ 
is a rank 1 module, so $M$ is rigid: the projective cover of $M$ has exactly three summands and so $\Omega(M)=\tau^{-1}(M)$ is a rank 1 module. 
In particular, $M$ is rigid.

\medskip

For (2), we have to determine how many profiles with 2 rows exist with exactly three boxes, we claim this to be equal to $N_{k,r}$. 
To show this, we consider how many different profiles can arise from a given one  
through $a$-shifts and reordering of boxes. So assume $I$ and $J$ are the two rows of a rank 2 profile 
(of a rigid indecomposable module). 
{By Corollary~\ref{cor:reducing}, we can assume that $I$ and $J$ are fully reduced, 
i.e., that $I\cap J=I^c\cap J^c=\emptyset$}. 

We have that $|I\cup J|=2r$ for some $r\in [3,k]$. 
Let $r_1,r_2,r_3$ be the sizes of the three boxes (ordered such that the box of size $r_1$ 
starts with the smallest $i\in\{1,2\dots, n\}\setminus (I^c\cap J^c)$). 
If $|\{r_1,r_2,r_3\}|=1$, then $3\mid r$ and the partial shifts yield $2r/3$ different modules 
with the same content. 
If $|\{r_1,r_2,r_3\}|=2$, the partial shifts yield $2r$ different modules with the same content. 
If $|\{r_1,r_2,r_3\}|=3$, the partial shifts yield $2r$ additional different modules with the same content. 
Reordering the boxes so that the sizes are $r_1,r_3,r_2$ (i.e., interchanging two boxes) 
yields $2r$ additional modules 
with the same content (as the sizes are all different). So $4r$ different modules with the same 
content. 

The claim then follows with part (1), as every such profile with exactly three 
boxes gives a rigid indecomposable.
\end{proof}

\begin{conjecture}\label{conj:number-is-correct}
The number of profiles of rigid indecomposable  
rank $2$ modules in ${\rm CM}(B_{k,n})$ 
is $N_{k,n}$. 
\end{conjecture}

Note that in the tame cases $(3,9)$ and $(4,8)$, the number of rigid indecomposable rank 2 modules can already be 
deduced from the results of~\cite{BBG}. In the general infinite case, there are many morphisms between tubes and even though 
the rank 2 modules are sitting low in their tubes, a priori, the formula of Theorem~\ref{thm:rank2-bound} (2) only gives a lower bound.

\begin{example}
In case $k=4$ and $n=8$, the formula gives 
\begin{align*}
N_{4,8}=\sum_{r=3}^{4} ( \frac{2r}{3} \cdot p_1(r) +  2r \cdot p_2(r) + 4r \cdot p_3(r)) \cdot {8 \choose 2r} {8-2r \choose 4-r},
\end{align*} 
The only possibilities to form three boxes are $r=3$ with $r_1=r_2=r_3=1$ and $r=4$ with 
$r_1=r_2=1$, $r_3=2$. So $p_1(3)=1$, $p_2(4)=1$ and all other $p_i(r)$ are $0$. The formula gives 
$N_{4,8}=2p_1(3){8\choose 6}{2\choose 1} + 8p_1(4) {8\choose 8}{0\choose 0}$
$=2\cdot 28\cdot 2 +8=120$, which is equal to the number of rigid indecomposables in this 
case, cf.~\cite[Section 7]{BBG}. 
For $r=3$, all these modules correspond to real roots and the 8 modules with imaginary roots arise 
from $r=4$. 
\end{example}


\section{Subspace configurations of rank 3 modules} \label{sec:subspace configurations}

In this section, we use the subspace configurations to 
derive necessary conditions for indecomposability (Subsection~\ref{ssec:necessary}) and 
then use these to study the roots associated with indecomposable rank 3 modules 
(Subsection~\ref{ssec:rk3-roots}). We automatically have $k\ge 3$, as there are no indecomposable 
higher rank modules for $k\le 2$. 
To find these necessary conditions, we consider the number of (quasi-)boxes between the two pairs of 
successive rims of $M$; the number of valleys of the upper rim is an upper bound for this number.

Note that if a module $M$ is indecomposable, any two successive rims of $M$ are ``closely packed'' in the 
sense that there cannot be a walk between the two rims, \cite[Section 6]{JKS}.

\subsection{Necessary conditions for indecomposability}\label{ssec:necessary}

The first case is when $I=J$, then there are no boxes between the two rims of $L_I=L_J$.

\begin{lemma} \label{lm:IJK-indec-0-4-boxes}
Let $M=L_I|L_J|L_K \in {\rm CM}(B_{k,n})$ be indecomposable. 
If $I=J$, 
then the rims of $L_J$ and $L_K$ form at least four quasi-boxes.
\end{lemma}

\begin{proof}
Suppose that $M=L_I|L_J|L_K \in {\rm CM}(B_{k,n})$ is indecomposable, with $I=J$. 
In this case, the subspace configuration of $M$ looks like a star. It has one 
vertex with multiplicity 3 and $m\ge 0$ vertices with $2$, each with an arrow towards $3$. 
Such a subspace configuration is decomposable for $m\le 3$. 
\end{proof}

Let $I$ and $J$ be $k$-subsets. 
Instead of writing that (quasi-)boxes are formed by the rims of $L_I$ and $L_J$, we will also say 
that the (quasi-)boxes are formed by $L_I$ and $L_J$ or by $I$ and $J$.

\begin{lemma} \label{lm:IJK-indec-1-4-or-G}
Let  $M=L_I|L_J|L_K \in {\rm CM}(B_{k,n})$ be indecomposable. \\
(1) 
If $L_J$ and $L_K$ form up to three quasi-boxes, then $L_I, L_J$ form at least two quasi-boxes or the 
subspace configuration of $M$ is (G) in Figure~\ref{fig:IJ_one_box_JK_up_to_3}. \\
(2) 
If $L_I, L_J$ form one quasi-box, then $L_J, L_K$ form at least four quasi-boxes or the subspace configuration of 
$M$ is (G) in Figure \ref{fig:IJ_one_box_JK_up_to_3}. 
\end{lemma}

\begin{proof}
Suppose that $M=L_I|L_J|L_K \in {\rm CM}(B_{k,n})$ is indecomposable and $L_J, L_K$ form at 
most three quasi-boxes. 
By Lemma \ref{lm:IJK-indec-0-4-boxes}, $I\ne J$, so $L_I, L_J$ form at least one quasi-box. 
All possible subspace configurations of $M$ when $I$ and $J$ form one quasi-box are 
in Figure~\ref{fig:IJ_one_box_JK_up_to_3}. The only indecomposable one among them is (G). 
This proves (1). 

For the same reason, if 
$L_I, L_J$ form one quasi-box, then $L_J, L_K$ form at least four quasi-boxes or the subspace configuration 
of $M$ is (G).
\end{proof}

\begin{figure}
\begin{subfigure}{0.19\textwidth}
\centering
\begin{tikzpicture}[scale=0.6]
\tikzstyle{state}=[thick,minimum size=2mm] 
		\node [state] (1) at (-8, 3) {$1$};  
		\node [state] (2) at (-8, 0.5) {$3$}; 
		\draw (1) to (2);  
\end{tikzpicture}
\caption{} 
\end{subfigure}
\begin{subfigure}{0.19\textwidth}
\centering
\begin{tikzpicture}[scale=0.6]
\tikzstyle{state}=[thick,minimum size=2mm] 
		\node [state] (1) at (-8, 3) {$1$};  
		\node [state] (2) at (-8, 1.75) {$2$}; 
		\node [state] (3) at (-8, 0.5) {$3$};
		\draw (1) to (2);
		\draw (2) to (3);  
\end{tikzpicture}
\caption{} 
\end{subfigure}
\begin{subfigure}{0.19\textwidth}
\centering
\begin{tikzpicture}[scale=0.6]
\tikzstyle{state}=[thick,minimum size=2mm]
		\node [state] (1) at (-9, 3) {$1$};
		\node [state] (2) at (-7, 3) {$2$};  
		\node [state] (3) at (-8, 0.5) {$3$};
		\draw (1) to (3);
		\draw (2) to (3);  
\end{tikzpicture}
\caption{} 
\end{subfigure}
\begin{subfigure}{0.19\textwidth}
\centering
\begin{tikzpicture}[scale=0.6]
\tikzstyle{state}=[thick,minimum size=2mm]
		\node [state] (1) at (-8, 3) {$1$};
		\node [state] (2) at (-9, 1.75) {$2$};  
		\node [state] (3) at (-7, 1.75) {$2$};  
		\node [state] (4) at (-8, 0.5) {$3$};
		\draw (1) to (2);
		\draw (1) to (3);  
		\draw (2) to (4);  
		\draw (3) to (4);  
\end{tikzpicture}
\caption{} 
\end{subfigure}
\begin{subfigure}{0.19\textwidth}
\centering
\begin{tikzpicture}[scale=0.6]
\tikzstyle{state}=[thick,minimum size=2mm]
		\node [state] (1) at (-8, 3) {$1$};
		\node [state] (2) at (-9, 1.75) {$2$};  
		\node [state] (3) at (-7, 1.75) {$2$};  
		\node [state] (4) at (-8, 0.5) {$3$};
		\draw (1) to (2);   
		\draw (2) to (4);  
		\draw (3) to (4);  
\end{tikzpicture}
\caption{} 
\end{subfigure}
\begin{subfigure}{0.19\textwidth}
\centering
\begin{tikzpicture}[scale=0.6]
\tikzstyle{state}=[thick,minimum size=2mm]
		\node [state] (1) at (-9, 3) {$1$};
		\node [state] (2) at (-8, 3) {$2$};  
		\node [state] (3) at (-7, 3) {$2$};  
		\node [state] (4) at (-8, 0.5) {$3$};
		\draw (1) to (4);
		\draw (2) to (4);  
		\draw (3) to (4);   
\end{tikzpicture}
\caption{} 
\end{subfigure}
\begin{subfigure}{0.19\textwidth}
\centering
\begin{tikzpicture}[scale=0.6]
\tikzstyle{state}=[thick,minimum size=2mm]
		\node [state] (1) at (-10, 3) {$1$};
		\node [state] (2) at (-9, 3) {$2$};  
		\node [state] (3) at (-8, 3) {$2$};
		\node [state] (4) at (-7, 3) {$2$};  
		\node [state] (5) at (-8.5, 0.5) {$3$};
		\draw (1) to (5);
		\draw (2) to (5);  
		\draw (3) to (5);   
		\draw (4) to (5);   
\end{tikzpicture}
\caption{} 
\end{subfigure}
\begin{subfigure}{0.19\textwidth}
\centering
\begin{tikzpicture}[scale=0.6]
\tikzstyle{state}=[thick,minimum size=2mm]
		\node [state] (1) at (-9, 3) {$1$};
		\node [state] (2) at (-9, 1.75) {$2$};  
		\node [state] (3) at (-8, 1.75) {$2$};
		\node [state] (4) at (-7, 1.75) {$2$};  
		\node [state] (5) at (-8, 0.5) {$3$};
		\draw (1) to (2);
		\draw (2) to (5);  
		\draw (3) to (5);   
		\draw (4) to (5);   
\end{tikzpicture}
\caption{} 
\end{subfigure}
\begin{subfigure}{0.19\textwidth}
\centering
\begin{tikzpicture}[scale=0.6]
\tikzstyle{state}=[thick,minimum size=2mm]
		\node [state] (1) at (-9, 3) {$1$};
		\node [state] (2) at (-9, 1.75) {$2$};  
		\node [state] (3) at (-8, 1.75) {$2$};
		\node [state] (4) at (-7, 1.75) {$2$};  
		\node [state] (5) at (-8, 0.5) {$3$};
		\draw (1) to (2);
		\draw (1) to (3);
		\draw (2) to (5);  
		\draw (3) to (5);   
		\draw (4) to (5);   
\end{tikzpicture}
\caption{} 
\end{subfigure}
\begin{subfigure}{0.19\textwidth}
\centering
\begin{tikzpicture}[scale=0.6]
\tikzstyle{state}=[thick,minimum size=2mm]
		\node [state] (1) at (-8, 3) {$1$};
		\node [state] (2) at (-9, 1.75) {$2$};  
		\node [state] (3) at (-8, 1.75) {$2$};
		\node [state] (4) at (-7, 1.75) {$2$};  
		\node [state] (5) at (-8, 0.5) {$3$};
		\draw (1) to (2);
		\draw (1) to (3);
		\draw (1) to (4);
		\draw (2) to (5);  
		\draw (3) to (5);   
		\draw (4) to (5);   
\end{tikzpicture}
\caption{} 
\end{subfigure}
\caption{All subspace configurations for $L_I|L_J|L_K$, where $L_I, L_J$ form one quasi-box and 
$L_J, L_K$ form at most three quasi-boxes.}
\label{fig:IJ_one_box_JK_up_to_3}
\end{figure}
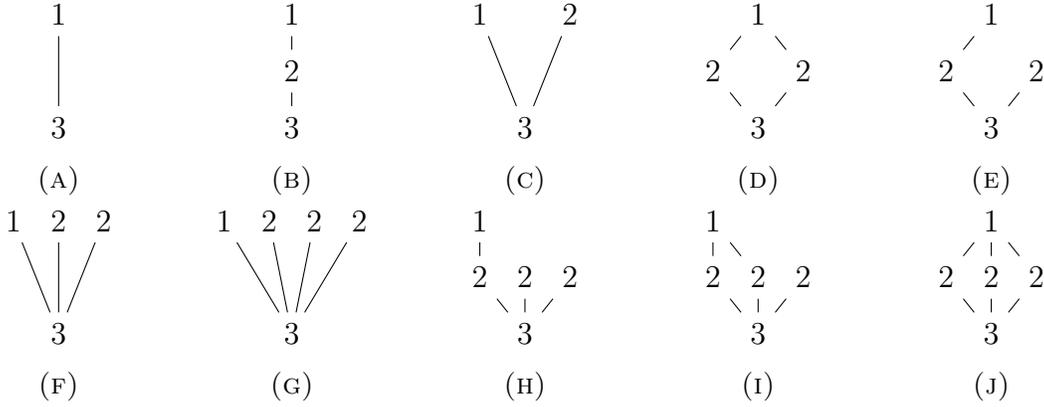

\begin{lemma} \label{lm:IJK-indec-2-3-or-K}
Let $M=L_I|L_J|L_K \in {\rm CM}(B_{k,n})$ be indecomposable, and assume that 
$L_I$ and $L_J$ form two quasi-boxes. 
Then $L_J, L_K$ form at least three quasi-boxes or the subspace configuration of $M$ is (K) from 
Figure~\ref{fig:IJ-2-boxes_JK_up_to_2}. 
\end{lemma}

\begin{proof}
Suppose that 
$M=L_I|L_J|L_K \in {\rm CM}(B_{k,n})$ is indecomposable and that $L_I, L_J$ form two quasi-boxes. 
Suppose that $L_J, L_K$ form at most two quasi-boxes. 
All subspace configurations for such $M$ where $L_J$ and $L_K$ form at most two quasi-boxes are 
in Figure~\ref{fig:IJ-2-boxes_JK_up_to_2}. The only indecomposable one of them is (K). 
This proves the claim. 
\end{proof}

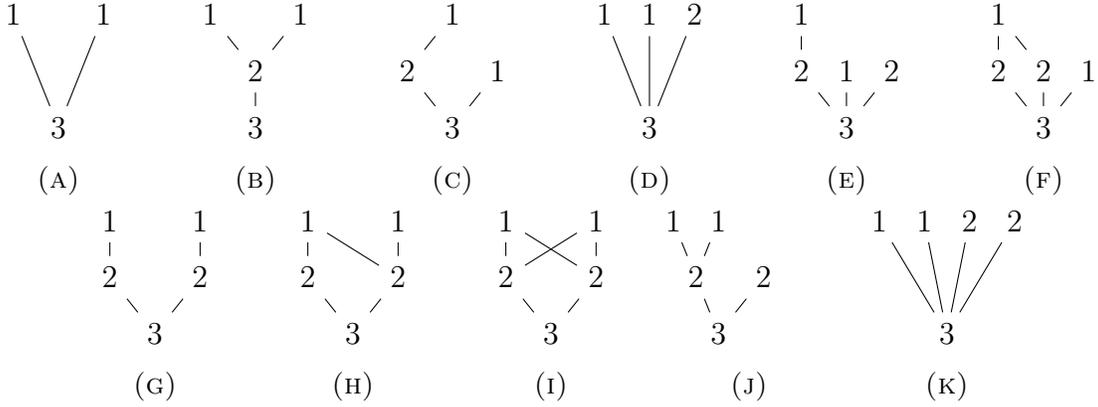
\begin{figure}
\begin{subfigure}{0.16\textwidth}
\centering
\begin{tikzpicture}[scale=0.6]
\tikzstyle{state}=[thick,minimum size=2mm]
		\node [state] (1) at (-9, 3) {$1$};
		\node [state] (2) at (-7, 3) {$1$};  
		\node [state] (3) at (-8, 0.5) {$3$};
		\draw (1) to (3);
		\draw (2) to (3);  
\end{tikzpicture}
\caption{} 
\end{subfigure}
\begin{subfigure}{0.16\textwidth}
\centering
\begin{tikzpicture}[scale=0.6]
\tikzstyle{state}=[thick,minimum size=2mm]
		\node [state] (1) at (-9, 3) {$1$};
		\node [state] (2) at (-7, 3) {$1$};  
		\node [state] (3) at (-8, 1.75) {$2$};  
		\node [state] (4) at (-8, 0.5) {$3$};
		\draw (1) to (3);
		\draw (2) to (3);  
		\draw (3) to (4);  
\end{tikzpicture}
\caption{} 
\end{subfigure}
\begin{subfigure}{0.16\textwidth}
\centering
\begin{tikzpicture}[scale=0.6]
\tikzstyle{state}=[thick,minimum size=2mm]
		\node [state] (1) at (-8, 3) {$1$};
		\node [state] (2) at (-9, 1.75) {$2$};  
		\node [state] (3) at (-7, 1.75) {$1$};  
		\node [state] (4) at (-8, 0.5) {$3$};
		\draw (1) to (2);   
		\draw (2) to (4);  
		\draw (3) to (4);  
\end{tikzpicture}
\caption{} 
\end{subfigure}
\begin{subfigure}{0.16\textwidth}
\centering
\begin{tikzpicture}[scale=0.6]
\tikzstyle{state}=[thick,minimum size=2mm]
		\node [state] (1) at (-9, 3) {$1$};
		\node [state] (2) at (-8, 3) {$1$};  
		\node [state] (3) at (-7, 3) {$2$};  
		\node [state] (4) at (-8, 0.5) {$3$};
		\draw (1) to (4);
		\draw (2) to (4);  
		\draw (3) to (4);   
\end{tikzpicture}
\caption{} 
\end{subfigure}
\begin{subfigure}{0.16\textwidth}
\centering
\begin{tikzpicture}[scale=0.6]
\tikzstyle{state}=[thick,minimum size=2mm]
		\node [state] (1) at (-9, 3) {$1$};
		\node [state] (2) at (-9, 1.75) {$2$};  
		\node [state] (3) at (-8, 1.75) {$1$};
		\node [state] (4) at (-7, 1.75) {$2$};  
		\node [state] (5) at (-8, 0.5) {$3$};
		\draw (1) to (2);
		\draw (2) to (5);  
		\draw (3) to (5);   
		\draw (4) to (5);   
\end{tikzpicture}
\caption{} 
\end{subfigure}
\begin{subfigure}{0.16\textwidth}
\centering
\begin{tikzpicture}[scale=0.6]
\tikzstyle{state}=[thick,minimum size=2mm]
		\node [state] (1) at (-9, 3) {$1$};
		\node [state] (2) at (-9, 1.75) {$2$};  
		\node [state] (3) at (-8, 1.75) {$2$};
		\node [state] (4) at (-7, 1.75) {$1$};  
		\node [state] (5) at (-8, 0.5) {$3$};
		\draw (1) to (2);
		\draw (1) to (3);
		\draw (2) to (5);  
		\draw (3) to (5);   
		\draw (4) to (5);   
\end{tikzpicture}
\caption{} 
\end{subfigure}
\begin{subfigure}{0.16\textwidth}
\centering
\begin{tikzpicture}[scale=0.6]
\tikzstyle{state}=[thick,minimum size=2mm]
		\node [state] (1) at (-9, 3) {$1$};
		\node [state] (2) at (-7, 3) {$1$};  
		\node [state] (3) at (-9, 1.75) {$2$};
		\node [state] (4) at (-7, 1.75) {$2$};  
		\node [state] (5) at (-8, 0.5) {$3$};
		\draw (1) to (3);
		\draw (2) to (4);
		\draw (3) to (5);
		\draw (4) to (5);   
\end{tikzpicture}
\caption{} 
\end{subfigure}
\begin{subfigure}{0.16\textwidth}
\centering
\begin{tikzpicture}[scale=0.6]
\tikzstyle{state}=[thick,minimum size=2mm]
		\node [state] (1) at (-9, 3) {$1$};
		\node [state] (2) at (-7, 3) {$1$};  
		\node [state] (3) at (-9, 1.75) {$2$};
		\node [state] (4) at (-7, 1.75) {$2$};  
		\node [state] (5) at (-8, 0.5) {$3$};
		\draw (1) to (3);
		\draw (1) to (4);
		\draw (2) to (4);
		\draw (3) to (5);
		\draw (4) to (5);   
\end{tikzpicture}
\caption{} 
\end{subfigure}
\begin{subfigure}{0.16\textwidth}
\centering
\begin{tikzpicture}[scale=0.6]
\tikzstyle{state}=[thick,minimum size=2mm]
		\node [state] (1) at (-9, 3) {$1$};
		\node [state] (2) at (-7, 3) {$1$};  
		\node [state] (3) at (-9, 1.75) {$2$};
		\node [state] (4) at (-7, 1.75) {$2$};  
		\node [state] (5) at (-8, 0.5) {$3$};
		\draw (1) to (3);
		\draw (1) to (4);
		\draw (2) to (3);
		\draw (2) to (4);
		\draw (3) to (5);
		\draw (4) to (5);   
\end{tikzpicture}
\caption{} 
\end{subfigure}
\begin{subfigure}{0.16\textwidth}
\begin{tikzpicture}[scale=0.6]
\tikzstyle{state}=[thick,minimum size=2mm]
		\node [state] (1) at (-9, 3) {$1$};
		\node [state] (2) at (-8, 3) {$1$};  
		\node [state] (3) at (-8.5, 1.75) {$2$};
		\node [state] (4) at (-7, 1.75) {$2$};  
		\node [state] (5) at (-8, 0.5) {$3$};
		\draw (1) to (3); 
		\draw (2) to (3);
		\draw (3) to (5);
		\draw (4) to (5);   
\end{tikzpicture}
\caption{} 
\end{subfigure}
\begin{subfigure}{0.16\textwidth}
\centering
\begin{tikzpicture}[scale=0.6]
\tikzstyle{state}=[thick,minimum size=2mm]
		\node [state] (1) at (-10, 3) {$1$};
		\node [state] (2) at (-9, 3) {$1$};  
		\node [state] (3) at (-8, 3) {$2$};
		\node [state] (4) at (-7, 3) {$2$};  
		\node [state] (5) at (-8.5, 0.5) {$3$};
		\draw (1) to (5);
		\draw (2) to (5);  
		\draw (3) to (5);   
		\draw (4) to (5);   
\end{tikzpicture}
\caption{} 
\end{subfigure}
\caption{All possible subspace configurations of $L_I|L_J|L_K$, $L_I, L_J$ form two quasi-boxes and $L_J, L_K$ form at most two quasi-boxes.}
\label{fig:IJ-2-boxes_JK_up_to_2}
\end{figure}

\subsection{Rank $3$ indecomposable modules in ${\rm CM}(B_{3,n})$}\label{ssec:rk3-roots}

If we restrict to $k=3$, we can say even more about the profile of a rank 3 indecomposable module, as we will show now. 

\begin{lemma} \label{lem:k=3_IJK-indec_no_G_K} 
Suppose that $M=L_I|L_J|L_K \in {\rm CM}(B_{3,n})$ is indecomposable. Then the subspace 
configuration of $M$ cannot be (G) of Figure \ref{fig:IJ_one_box_JK_up_to_3} or (K) of 
Figure \ref{fig:IJ-2-boxes_JK_up_to_2}. 
\end{lemma}

\begin{proof}
The subspace configurations (G) of Figure~\ref{fig:IJ_one_box_JK_up_to_3} and (K) of 
Figure~\ref{fig:IJ-2-boxes_JK_up_to_2} both correspond to a partition of a $k$-subset into four 
non-trivial parts, hence only occur for $k\ge 4$. 
\end{proof}

\begin{lemma} \label{lem:k3_rk3_interlacing-r1-r2}
If $M=L_I|L_J|L_K \in {\rm CM}(B_{3,n})$ is indecomposable, $I$ and $J$ are interlacing. 
\end{lemma}

\begin{proof}
Suppose that $M=L_I|L_J|L_K \in {\rm CM}(B_{3,n})$ is indecomposable. Since $k=3$, 
$J$ and $K$ form at most three quasi-boxes. 
Suppose that $I$ and $J$ are not interlacing, then they form at most two quasi-boxes (as $k=3$). 
By Lemma \ref{lm:IJK-indec-1-4-or-G}, $I$ and $J$ then form exactly two quasi-boxes 
(as the subspace configuration 
of $M$ cannot be Figure~\ref{fig:IJ_one_box_JK_up_to_3} (G) by 
Lemma~\ref{lem:k=3_IJK-indec_no_G_K}).

Now we use Lemma~\ref{lm:IJK-indec-2-3-or-K} to see that $J$ and $K$ form at least three quasi-boxes 
(as the subspace configuration cannot be Figure~\ref{fig:IJ-2-boxes_JK_up_to_2} (K) 
by Lemma~\ref{lem:k=3_IJK-indec_no_G_K}). As $k=3$, $J$ and $K$ form exactly three quasi-boxes.

Since $k=3$, one of the quasi-boxes has size 1 and one has size 2. Since $I$ and $J$ are not interlacing, 
the quasi-box of size 2 must be a $2 \times 2$ square. 
But then the rim of $L_J$ has exactly two valleys, and since the number of valleys is 
an upper bound for the number of quasi-boxes, $J$ and $K$ cannot form three  quasi-boxes. A contradiction. 
\end{proof}

\begin{theorem}\label{thm:rank3-in-CM-3-n}
Let $M\in{\rm CM}(B_{3,n})$ be indecomposable of rank $3$. Then $q(M)\in \{0,2\}$ 
and $\varphi(M)$ is a real or imaginary root of $J_{3,n}$. 
\end{theorem}

\begin{proof}
Suppose that $M=L_I|L_J|L_K$ is an indecomposable module of rank $3$ in ${\rm CM}(B_{3,n})$. 
If $n<8$, there are no indecomposable rank 3 modules. If $n=8$, all indecomposables correspond to real roots, \cite{JKS}. So let $n\ge 9$. 
It suffices to show that the content of $M$ consists of $8$ or $9$ different 
numbers. If this is the case, we get, up to reordering the $x_i$, that the vector of multiplicities of $M$ is 
$x(M)=(2,1,1,1,1,1,1,0,\dots,0)\in \ZZ^n$ in the first case, with 
$q(M)=2$ and $x(M)=(1,1,1,1,1,1,1,1,1,0,\dots,0)\in \ZZ^n$ in the second case, with $q(M)=0$. 

We consider the quasi-boxes formed by the rims. Note that since $k=3$, any two successive rims form at 
most three  quasi-boxes. We go through these cases. 

\begin{itemize}
\item 
If $I=J$, 
$M$ decomposes by Lemma~\ref{lm:IJK-indec-0-4-boxes}, a contradiction. 

\item
Suppose next that $I$ and $J$ form one quasi-box. Since $J$ and $K$ form at most three quasi-boxes, 
Lemma~\ref{lm:IJK-indec-1-4-or-G} tells that the subspace configuration of $M$ is (G) in 
Figure~\ref{fig:IJ_one_box_JK_up_to_3}. This is not possible by 
Lemma~\ref{lem:k=3_IJK-indec_no_G_K}. 

\item
Assume now that $I$ and $J$ form two quasi-boxes. Then using Lemma~\ref{lm:IJK-indec-2-3-or-K} 
and Lemma~\ref{lem:k=3_IJK-indec_no_G_K} we find that $J$ and $K$ have to form three quasi-boxes 
(each of size 1). In particular, they are 3-interlacing, $|J\cup K|=6$ and $J\cap K=\emptyset$. 

By Lemma~\ref{lem:k3_rk3_interlacing-r1-r2}, $I$ and $J$ are interlacing and since we have two quasi-boxes, 
$I$ and $J$ are 2-interlacing. So $|I\cap J|=1$. This implies that $I \cup J \cup K$ has at least 7 
different elements.

If $|I\cup J\cup K|=7$, $I$ and $K$ have one element in common. 
This implies that in the subspace configuration, there is a 2 mapping into a 3 
(from the common element of $I$ and $J$). There is also a 1 mapping into two 
different 2's. 
Using the reduction from Section~\ref{ssec:subspace}, we can remove this 1. 
The resulting diagram has two vertices with a 2, each with a single edge mapping into the 3. 
Such a subspace configuration is never indecomposable. 

So we get $|I\cup J\cup K|\ge 8$ as claimed.

\item
Suppose now that $L_I,L_J$ form three quasi-boxes. In particular, as $k=3$, 
they are 3-interlacing, $|I\cup J|=6$ and $|I\cap J|=\emptyset$.

The subspace configurations for the cases where 
$J$ and $K$ form at most one quasi-box are not indecomposable, see 
Figure~\ref{fig:subspace configuration, ij 3 boxes, jk 0 or 1 box}. 
So $J$ and $K$ form at least two quasi-boxes. 

If $|I\cup J\cup K|=6$, we have $K\subset I\cup J$ and since $J$ and $K$ form at least two quasi-boxes, 
$|J\cap K|\le 1$. So either $K=I$ or two elements of $K$ are in $I$, one is in $J$. In the first case, the 
subspace configuration consists of three 1's included simultaneously in two 2's. 
So they can be removed by the reduction of Section~\ref{ssec:subspace}. This leaves three 2's mapping into one 3, 
not an indecomposable configuration. 
In the second case, we have two 1's mapping simultaneously into two 2's, they can be removed. The 
resulting subspace configuration contains two 2's with a single edge into the 3. Again, this is not 
indecomposable. 

If $|I \cup J \cup K|=7$, two elements of $K$ are in $I\cup J$. These can be both in $I$ or one in $I$ and 
one in $J$. If they are both in $I$, we have, as before, two 1's in the subspace configuration mapping 
into two 2's each. Removing them leaves two 2's, each with only one edge to the 3. This is 
not an indecomposable subspace configuration. 

If one is in $I$ and one in $J$, we can also reduce two 1's mapping into two 2's each and get a subspace 
configuration which is not indecomposable. 

Therefore $I \cup J \cup K$ has at least $8$ different numbers.
\end{itemize}

We now prove that $\varphi(M)$ is a real or imaginary root in $J_{k,n}$. We will use a similar method 
as in the proof of Lemma \ref{lm:rk2-condition}.

First we consider the case when $q(M)=2$. Up to Weyl group action, we may assume that $x(M) = (2,1,1,1,1,1,1,1, 0, \ldots, 0)$. Then
\begin{align*}
\varphi(M) = 3\beta+\alpha_1+3\alpha_2+5\alpha_3+4\alpha_4+3\alpha_5+2\alpha_6+\alpha_7.
\end{align*}
We have that  
\begin{align*}
s_3s_4s_5s_2s_3s_4s_1s_2s_3s_{\beta}s_6s_7s_5s_6s_4s_5s_3s_4s_2s_3s_{\beta}(\varphi(M)) = \beta
\end{align*}
which is a simple root. Therefore $\varphi(M)$ is a real root of $J_{3,n}$. 

Now we consider the case when $q(M)<2$. Up to Weyl group action, we may assume that $x(M) = (1,1,1,1,1,1,1,1,1, 0, \ldots, 0)$. Therefore 
\begin{align*}
\varphi(M) = 3\beta+2\alpha_1+4\alpha_2+6\alpha_3+5\alpha_4+4\alpha_5+3\alpha_6+2\alpha_7+\alpha_8.
\end{align*}
We have that $\langle \varphi(M), \alpha_{9}^{\vee} \rangle = -1$, $\langle \varphi(M), \alpha_{i}^{\vee} \rangle = 0$ for $i \in [n] \setminus \{9\}$ (we denote $\beta = \alpha_n$). Moreover, ${\rm supp}(\varphi(M))$ is connected. 
By \cite[Theorem 5.4]{Kac}, $\varphi(M)$ is  an imaginary root of $J_{3,n}$.\end{proof}

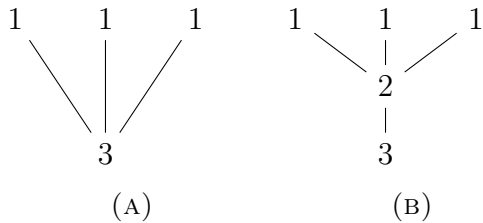
\begin{figure}
\begin{subfigure}{0.23\textwidth}
\begin{tikzpicture}[scale=0.6]
\tikzstyle{state}=[thick,minimum size=2mm]
		\node [state] (1) at (-10, 3) {$1$};
		\node [state] (2) at (-8, 3) {$1$};
		\node [state] (3) at (-6, 3) {$1$}; 
		\node [state] (4) at (-8, 0) {$3$};
		\draw (1) to (4);
		\draw (2) to (4);
		\draw (3) to (4);
\end{tikzpicture}
\caption{} 
\end{subfigure} 
\begin{subfigure}{0.23\textwidth}
\begin{tikzpicture}[scale=0.6]
\tikzstyle{state}=[thick,minimum size=2mm]
		\node [state] (1) at (-10, 3) {$1$};
		\node [state] (2) at (-8, 3) {$1$};
		\node [state] (3) at (-6, 3) {$1$}; 
		\node [state] (4) at (-8, 1.5) {$2$};
		\node [state] (5) at (-8, 0) {$3$};
		\draw (1) to (4);
		\draw (2) to (4);
		\draw (3) to (4);
		\draw (4) to (5);
\end{tikzpicture}
\caption{} 
\end{subfigure} 
\caption{Subspace configurations for $I,J$ forming three quasi-boxes and 
(A) $J=K$ or (B) $J$ and $K$ form one quasi-box} 
\label{fig:subspace configuration, ij 3 boxes, jk 0 or 1 box}
\end{figure}

Recall from Definition~\ref{def:canonical-real-profile} that a profile $P=(P_{ij})\in\mathcal P_{k,n}$ 
is said to be canonical if its entries are weakly column decreasing and if for all $j\ge 2$, 
$P_{m,j}\ge P_{1,j-1}$. The set of canonical profiles with real root is denoted by 
$C_{k,n}^{\rm re}$. 
Recall also that a cyclic permutation of a profile is obtained by permuting its rows cyclically 
(Definition~\ref{def:cyclic-perm}). 

\begin{theorem} \label{thm:k=3_rk=3_real_canonical} 
Let $M\in {\rm CM}(B_{3,n})$ be an indecomposable module of rank $3$.  
If $q(M)=2$, then $P_M$ is a cyclic permutation of some $P$ in 
$C_{3,n}^{\rm re}$. 
\end{theorem}

\begin{proof}
Write $M=L_I|L_J|L_K$ with $3$-subsets $I,J,K$. Recall that $k\ge 3$ as $M$ has rank 3. 
We already know from the proof of Theorem~\ref{thm:rank3-in-CM-3-n} that 
$x(M)=(2,1,\dots, 1,0,\dots, 0)$, up to reordering the entries (the entry $1$ appearing exactly 7 times).

In particular, $I,$ $J,$ and $K$ are pairwise different. If any two of them are interlacing, they have to be 
2-interlacing or 3-interlacing (if they were 1-interlacing, $x(M)$ would contain at least two entries equal 
to 2).

By Lemma \ref{lem:k3_rk3_interlacing-r1-r2}, $I$ and $J$ are interlacing. 
Assume first they are 2-interlacing. Then the rims of $L_I$ and $L_J$ form two quasi-boxes. By 
Lemma~\ref{lm:IJK-indec-2-3-or-K}, $L_J$ and $L_K$ have to form three quasi-boxes, as otherwise, the 
subspace configuration of $M$ would be the one in Figure~\ref{fig:IJ-2-boxes_JK_up_to_2} (K), 
which does not occur for $k=3$ (Lemma~\ref{lem:k=3_IJK-indec_no_G_K}). So $J$ and $K$ 
are 3-interlacing.

Let $\{i_1,\dots, i_9\}$ be the entries of the profile $P_M$ of $M$. 
Without loss of generality we can assume 
$i_1=i_2<i_3<\dots<i_9$ as the other cases are analogous. 

Since $I$ and $J$ have the element $i_1$ in common, $I$ and $J$ are 2-interlacing and $J$ and 
$K$ are 3-interlacing, the only possible choices for the rank 1 modules are the following: 

\begin{align*}
(a)\hskip.2cm & L_I = i_1 i_5 i_8, \ L_J=i_1 i_4 i_7, \ L_K=i_3 i_6 i_9, \\
(b)\hskip.2cm  & L_I=i_1 i_5 i_9, \ L_J=i_1 i_4 i_7, \ L_K=i_3 i_6 i_8, \\
(c)\hskip.2cm & L_I=i_1 i_6 i_8, \ L_J=i_1 i_4 i_7, \ L_K=i_3 i_5 i_9, \\
(d)\hskip.2cm & L_I=i_1 i_6 i_9, \ L_J=i_1 i_4 i_7, \ L_K=i_3 i_5 i_8, \\
(e)\hskip.2cm & L_I=i_1 i_4 i_7, \ L_J=i_1 i_5 i_8, \ L_K=i_3 i_6 i_9, \\
(f)\hskip.2cm & L_I=i_1 i_3 i_7, \ L_J=i_1 i_5 i_8, \ L_K=i_4 i_6 i_9, \\
(g)\hskip.2cm & L_I=i_1 i_4 i_6, \ L_J=i_1 i_5 i_8, \ L_K=i_3 i_7 i_9, \\
(h)\hskip.2cm & L_I=i_1 i_3 i_6, \ L_J=i_1 i_5 i_8, \ L_K=i_4 i_7 i_9.
\end{align*}
The only indecomposable subspace configuration among these is the one in (a) as one can 
check. 

If $M$ is as in (a), the profile $P_M$ is a cyclic permutation of the canonical profile 
$\thfrac{i_3 i_6 i_9}{i_1 i_5 i_8}{i_1 i_4 i_7}$.

Assume now that $I$ and $J$ are 3-interlacing. Then $J$ and $K$ are interlacing as otherwise, 
the subspace configuration would be three leaves with 1 mapping into a vertex with a 2 and from there one 
edge into a vertex with 3. 
Since $x(M)$ has only one entry 2, we must have that either $I,K$ are 2-interlacing with $J,K$ 
3-interlacing, case (i), or that $I,K$ are 3-interlacing with $J,K$ 2-interlacing, case (ii). 

As before, we assume without loss of generality that the entries of $I\cup J\cup K$ 
are of the form $i_1=i_2<i_3<\dots<i_9$. 
Then the only indecomposable subspace configurations are 
$\thfrac{i_1 i_4 i_7}{i_3 i_6 i_9}{i_1 i_5 i_8}$ for (i), $\thfrac{i_3 i_6 i_9}{i_1 i_5 i_8}{i_1 i_4 i_7}$ 
for (ii).  
\end{proof}

We expect that Theorem~\ref{thm:k=3_rk=3_real_canonical} is true for any indecomposable module 
in ${\rm CM}(B_{k,n})$, in arbitrary rank.

\begin{conjecture} \label{conj:real-roots-interlacing}
Let $M\in {\rm CM}(B_{k,n})$ be rigid indecomposable and assume that $\varphi(M)$ is a real root in $J_{k,n}$. 
Then $P_M$ is a cyclic permutation of a real canonical profile.
\end{conjecture}
 
\begin{theorem} \label{thm:k=3-mod-imaginary}
Let $M\in {\rm CM}(B_{3, n})$ be indecomposable of rank $3$ and assume that $\varphi(M)$ 
is imaginary. 
Then the profile $P_M$ of $M$ is one of the following:
\begin{align*}
\thfrac{ i_1   i_5   i_7}{ i_3   i_6   i_9}{ i_2   i_4   i_8 },\thfrac{ i_2   i_6   i_8}{ i_1   i_4   i_7}{ i_3   i_5   i_9 },\thfrac{ i_3   i_7   i_9}{ i_2   i_5   i_8}{ i_1   i_4   i_6 },\thfrac{ i_1   i_4   i_8}{ i_3   i_6   i_9}{ i_2   i_5   i_7 },\thfrac{ i_2   i_5   i_9}{ i_1   i_4   i_7}{ i_3   i_6   i_8 },\thfrac{ i_1   i_3   i_6}{ i_2   i_5   i_8}{ i_4   i_7   i_9 },\thfrac{ i_2   i_4   i_7}{ i_3   i_6   i_9}{ i_1   i_5   i_8 },\thfrac{ i_3   i_5   i_8}{ i_1   i_4   i_7}{ i_2   i_6   i_9 },\thfrac{ i_4   i_6   i_9}{ i_2   i_5   i_8}{ i_1   i_3   i_7 },
\thfrac{i_1 i_4 i_7}{i_3 i_6 i_9}{i_2 i_5 i_8}, \thfrac{i_2 i_5 i_8}{i_1 i_4 i_7}{i_3 i_6 i_9}, \thfrac{i_3 i_6 i_9}{i_2 i_5 i_8}{i_1 i_4 i_7},
\end{align*}
where $i_1<\cdots <i_9 \in \ZZ_{\ge 1}$.
\end{theorem} 

\begin{proof}
Denote by $i_1 \le \ldots \le i_9$ the entries of $P_M$. 
Since $\varphi(M)$ is imaginary, we have $q(M)=0$ and the vector $x(M)$ of multiplicities of $M$ 
is $(1,1,\dots, 1,0,\dots, 0)$, up to reordering the $x_i$, with $1$ occurring nine times, as seen in the 
proof of Theorem~\ref{thm:rank3-in-CM-3-n}.  
Therefore $I,$ $J,$ and $K$ are pairwise different and if any two of them interlace, they 
are $3$-interlacing. 

By Lemma~\ref{lem:k3_rk3_interlacing-r1-r2}, $I$ and $J$ are thus 3-interlacing. 
If $L_J, L_K$ are not interlacing, then $k=3$ implies that $L_J, L_K$ form one quasi-box. 
But then the subspace configuration of $M$ is a graph with 3 leaves with label 1, mapping into 
a vertex with label 2 and the latter into a vertex with label 3, not indecomposable. 
So both $I,J$ and $J,K$ are 3-interlacing.

Assume that  $i_1\in I$. Then the three 3-subsets have to be as follows:  
\begin{align*}
& (A) \qquad I = i_1 i_3 i_6, \ J=i_2 i_5 i_8, \ K=i_4 i_7 i_9, \\
& (B) \qquad I=i_1 i_3 i_7, \ J=i_2 i_5 i_8, \ K=i_4 i_6 i_9, \\
& (C) \qquad I=i_1 i_4 i_6, \ J=i_2 i_5 i_8, \ K=i_3 i_7 i_9, \\
& (D) \qquad I=i_1 i_4 i_7, \ J=i_2 i_5 i_8, \ K=i_3 i_6 i_9, \\
& (E) \qquad I=i_1 i_4 i_7, \ J=i_3 i_6 i_9, \ K=i_2 i_5 i_8, \\
& (F) \qquad I=i_1 i_4 i_8, \ J=i_3 i_6 i_9, \ K=i_2 i_5 i_7, \\
& (G) \qquad I=i_1 i_5 i_7, \ J=i_3 i_6 i_9, \ K=i_2 i_4 i_8, \\
& (H) \qquad I=i_1 i_5 i_8, \ J=i_3 i_6 i_9, \ K=i_2 i_4 i_7.
\end{align*}
Only modules (A), (E), (F), (G) have indecomposable subspace configurations, as one can check. 

Similary, if $i_1$ appears in the second row of $P_M$, then $P_M$ is one of the following:
\begin{align*}
\thfrac{ i_2  i_6  i_8}{ i_1  i_4  i_7}{ i_3  i_5  i_9 }, \ \thfrac{ i_2  i_5  i_9}{ i_1  i_4  i_7}{ i_3  i_6  i_8 }, \ \thfrac{ i_3  i_5  i_8}{ i_1  i_4  i_7}{ i_2  i_6  i_9 }, \ \thfrac{ i_2  i_5  i_8}{ i_1  i_4  i_7}{ i_3  i_6  i_9 } \,.
\end{align*}
If $i_1$ appears in the third row of $P_M$, then $P_M$ is one of the following:
\begin{align*}
\thfrac{ i_3  i_7  i_9}{ i_2  i_5  i_8}{ i_1  i_4  i_6 }, \ \thfrac{ i_2  i_4  i_7}{ i_3  i_6  i_9}{ i_1  i_5  i_8 }, \ \thfrac{ i_4  i_6  i_9}{ i_2  i_5  i_8}{ i_1  i_3  i_7 }, \ \thfrac{ i_3  i_6  i_9 }{ i_2  i_5  i_8}{ i_1  i_4  i_7 }\,.
\end{align*}
\end{proof}

Recall that if we add a fixed number to every element of a profile 
$P$, reducing modulo $n$, the result is called a shift of $P$ (Definition~\ref{def:shift}). 

\begin{corollary}\label{cor:39-imaginary}
In ${\rm CM}(B_{3,9})$, the profile of an indecomposable rank $3$ module with imaginary root has 
to be a shift of one of the following two:   
{\small $$\thfrac{157}{369}{248}, \quad \thfrac{147}{369}{258}.
$$ }
\end{corollary}

We will see in Proposition~\ref{proposition:non-rigid modules} that modules with 
the second profile are not rigid. 

%
\section{Auslander-Reiten quiver} \label{sec:Auslander-Reiten quiver}

The goal of this section is to characterise part of the Grassmannian cluster category  
${\rm CM}(B_{3,n})$. 
In particular, we will give all Auslander-Reiten sequences where the end terms 
are a rank 1 and a rank 2 module respectively. We will use these results to show that there are 
exactly 225 rigid indecomposable rank 3 modules in the case $n=9$.

%
\subsection{Profiles of Auslander-Reiten translates}

We will need to determine profiles of (inverses of) 
Auslander-Reiten translates of indecomposable modules. 
This is done as follows. \\
Let $M$ be a rigid non-projective indecomposable module in ${\rm CM}(B_{k,n})$ with 
filtration  $M=L_{I_1}| L_{I_2}| \dots | L_{I_s}$. 
Recall that a peak of a rank 1 module $L_I$  is an element $i\notin I$ such that  $i+1\in I$ 
(Definition~\ref{def:peak}). 

For $j\in [n]$ let $P_j=L_{\{j+1,j+2,\dots, j+k\}}$ be the projective rank 1 module in 
${\rm CM}(B_{k,n})$ with peak at $j$. 

\begin{remark}\label{rem:profile-tau} 
Let $M\in \underline{{\rm CM}}(B_{k,n})$ be rigid indecomposable. Recall that $\tau^{-1}(M)$ is 
also rigid (Remark~\ref{rem:rigid-tau}). 
To find the profile of $\tau^{-1}(M)$ we use a strategy to compute the first syzygy of 
$M$, which is the same as $\tau^{-1}(M)$. In practice, already finding the 
minimal projective  cover of $M$ is difficult. 
We will only need to do this in small ranks (ranks 1,2,3) where it is still feasible. 

Let $\oplus_{j\in U}P_j$ be the minimal projective 
cover of $M$, where $U$ is a multiset with entries in $\{0,1,\dots, n-1\}$. 
We consider the lattice diagrams of $\oplus_{j\in U}P_j$ and of $M$. Then 
the dimensions in the corresponding lattice diagram of the syzygy of $M$, i.e., of 
$\tau^{-1}(M)$,  are given as 
$\dim\oplus_{j\in U} P_j-\dim M$.  
We will call this infinite tuple the {\bf dimension lattice} of $\tau^{-1}(M)$. 
Assume that $\tau^{-1}(M)$ has filtration $L_{J_1}|L_{J_2}|\dots |L_{J_t}$. 

From the dimension lattice $\dim\oplus_{j\in U} P_j-\dim M$ we can find the filtration factors of 
$\tau^{-1}(M)$ iteratively. Draw the lattice diagram for the dimension lattice of $\tau^{-1}(M)$. Then the set $J_1$ consists of the $i\in[n]$ such that the arrow $x_i$ is in the top rim of the lattice diagram. 
The set $J_2$ consists of  the $i\in[n]$ such that $x_i$ is in the top rim of the lattice diagram obtained by 
removing the rim of $L_{J_1}$, etc. 
\end{remark}

\begin{example}\label{ex:tau-of-rk-1}
For $I=\{i,j,j+1,\dots, j+k-2\}$, Remark~\ref{rem:profile-tau} gives 
$\tau^{-1}(L_I)=L_J$ with $J=\{i+1,i+2,\dots, i+k-1,j+k-1\}$. 
\end{example}
%
\subsection{Auslander-Reiten sequences}

Observe that if $L_I$ has $m$ peaks, then by Remark~\ref{rem:rigid-tau}, 
$\tau^{-1}(L_I)$ is also a rigid indecomposable module and since the rank is additive on short exact sequences, 
this module has rank $m-1$. 
We also recall that 
in the case where $L_I$ has 2 peaks and where the $k$-subset $I$ consists of two intervals, 
one with a single element, Auslander-Reiten sequences with these end terms have been determined 
in~\cite[Theorem 3.12]{BBG}. Here, we give a more general result for Auslander-Reiten sequences 
in ${\rm CM}(B_{3,n})$ with end terms a rank 1 and a rank 2 module. 

\begin{theorem} \label{thm:AR-sequences-3-peaks}
Let $L_I$ be a rank $1$ module in ${\rm CM}(B_{3,n})$, where $I=\{i_1, i_2, i_3\}$ 
with $i_1<i_2<i_3<i_1$ in the cyclic order 
is a $3$-subset with three peaks. 
Then the Auslander-Reiten sequence starting at $L_I$ is as follows: 
\[
L_I\hookrightarrow M\twoheadrightarrow \frac{L_X}{L_Y},
\]
where $X=\{i_1+1, i_2+1, i_3+1\}$ and $Y=\{i_1+2, i_2+2, i_3+2\}$. 

If the middle term $M$ is indecomposable, then $P_M=X|I|Y$. 
\end{theorem}

\begin{proof}
Let $L_I$ be a rank $1$ module in ${\rm CM}(B_{3,n})$, where 
$I=\{i_1, i_2, i_3\}$ is a $3$-subset with three peaks. Note that this implies $n\ge 6$. 
Using 
Remark~\ref{rem:profile-tau} we find 
$\tau^{-1}(L_I)=\frac{L_X}{L_Y}$, where $X=\{i_1+1, i_2+1, i_3+1\}$, $Y=\{i_1+2, i_2+2, i_3+2\}$.
Therefore, the Auslander-Reiten sequence starting at $L_I$ is 
$L_I \to M \to \frac{L_X}{L_Y}$, where $M$ is a rank $3$ module. 

We now show that if $M$ is indecomposable, its profile is as claimed. For that, we will study the content 
of $M$ and use it to find candidates for $M$. 

So suppose that $M$ is indecomposable. This implies $n\ge 8$, since there are no rank 3 
indecomposables for $n=6,7$. 
The content ${\rm con}(M)$ of $M$ is the union of the contents of the end terms 
$L_I$ and $L_X|L_Y$, so ${\rm con}(M)=\{i_1,i_2,i_3, i_1+1,i_2+1,i_3+1, i_1+2,i_2+2,i_3+2\}$. 
By the proof of Theorem~\ref{thm:rank3-in-CM-3-n}, this content consists of 8 or 9 different 
elements. If these numbers are all different, we have $i_1+2<i_2,\ \ i_2+2<i_3\ \mbox{and} \ i_3+2<i_1$. 
In this case, 
$q(M)=0$, i.e., $M$ corresponds to an imaginary root. 
By Theorem~\ref{thm:k=3-mod-imaginary}, there are 12 possibilities for the profile $P_M$ of $M$. 
The projective cover of the sum $L_I\oplus \frac{L_X}{L_Y}$ of the end terms 
has to contain the projective cover of the middle term as a summand. 
Comparing with these 12 profiles, we see that only the module with filtration $L_X|L_I|L_Y$ works, 
as claimed.

If the content of $M$ has exactly $8$ different numbers, then we have one of the following situations:
\begin{align*}
(i)\hskip.2cm & i_1+2=i_2, \ i_2+2<i_3,\ \mbox{ and } i_3+2<i_1, \\
(ii)\hskip.2cm & i_1+2<i_2 ,\  i_2+2=i_3, \ \mbox{ and } i_3+2<i_1, \\
(iii)\hskip.2cm & i_1+2<i_2, \ i_2+2<i_3,\ \mbox{ and } i_3+2=i_1.
\end{align*}
In this case, $q(M)=2$, i.e., $M$ corresponds to a real root of $J_{3,n}$. 
By Theorem~\ref{thm:k=3_rk=3_real_canonical}, the corresponding profile $P_M$ is one of the 
three cyclic permutations of a canonical one. 
In (i), 
the canonical profile is 
$\begin{small}
\begin{array}{ccc}
i_2 & i_2+2 & i_3+2  \\ 
\hline 
i_1+1 & i_2+1 & i_3+1 \\
\hline 
i_1 & i_2 & i_3 
\end{array}
\end{small}
$.
Comparing projective covers, as above, 
we find that $P_M$ is in fact equal to this, which 
is $X|I|Y$ as claimed.  
Similarly, for (ii) and (iii),  where the profiles of $M$ are the following two: 
\[
\begin{small}
\begin{array}{ccc}
i_1+1& i_2+1& i_3+1\\
\hline
i_1 & i_2 & i_3\\ 
\hline
i_1+2 & i_3 & i_3+2 
\end{array},
\hskip1.5cm
\begin{array}{ccc}
i_1+1 & i_2+1 & i_3+1 \\ 
\hline
i_1 & i_2 & i_3\\ 
\hline
i_1 & i_3 +2 & i_3+2 
\end{array}
\end{small}
\]
Therefore $P_M$ is $X|I|Y$ as claimed, in all cases. 
\end{proof}

\begin{lemma}\label{lm:AR-direct-summand}
Let $I=\{i,i+2,i+4\}$ (entries taken modulo $n$). 
Let $L_I\hookrightarrow M\twoheadrightarrow L_X|L_Y$ be as in 
Theorem~\ref{thm:AR-sequences-3-peaks}, with $X=\{i+1,i+3,i+5\}$ 
and $Y=\{i+2,i+4,i+6\}$. Assume that 
$M=L_A\oplus N$ with $A=\{i+1,i+2,i+4\}$. The module $N$ is indecomposable 
if and only if $n\ge 7$ and in this case, 
$N=L_B|L_Y$ for $B=I\cup X\setminus A=\{i,i+3,i+5\}$. 
\end{lemma}

\begin{proof}
If $n=6$, the Auslander-Reiten quiver is well-known, see \cite[Figure 10]{JKS} 
and the only two 3-subsets with three peaks are 
$\{1,3,5\}$ and $\{2,4,6\}$. For $I=\{1,3,5\}$, we have $X=\{2,4,6\}$, 
$Y=\{1,3,5\}$, $A=\{2,3,5\}$, and $N=L_{\{1,4,5\}}\oplus L_{\{1,3,6\}}$. 

So assume $n\ge 7$. 
If $N$ is indecomposable, we are in case (A) of Figure~\ref{fig:AR-sequence-M-decomposable}. 
For $N$ to be indecomposable, the two 3-subsets of the filtration of $N$ have to be 3-interlacing. 
Thus the module $N$ must be either $L_B|L_Y$ or $L_Y|L_B$ with $B=\{i,i+3,i+5\}$. 
Comparing the projective covers shows that $N=L_B|L_Y$. 

If $N=C \oplus E$ is a direct sum of two rank $1$ modules, the Auslander-Reiten quiver locally is as in (B) 
of Figure~\ref{fig:AR-sequence-M-decomposable} with modules $C',C''$ $E',E''$, and where 
$E'$, $E''$ may be trivial modules. 
We have $\con(C)\cup\con(E)=\{i,i+2,i+3,i+4,i+5,i+6\}$. These six elements are all different since $n\ge 7$. 
The six entries have to be partitioned into two $3$-subsets $J_1, J_2$, with 
$C=L_{J_1}$ and $E=L_{J_2}$ in a way that $C$ and $E$ combined 
have the four peaks $\{i+1,i+3,i+5,i+6\}$. There are three ways to do so: 
(i) $\{i,i+2,i+6\}\cup \{i+3,i+4,i+5\}$, 
(ii) $\{i,i+4,i+5\}\cup \{i+2,i+3,i+6\}$, or 
(iii) $\{i,i+3,i+4\}\cup \{i+2,i+5,i+6\}$. Only in case (iii), the direct sum $C\oplus E$ has the 
peaks $\{i,i+2,i+3,i+5\}$ which are needed in order to match the projective covers. 
So assume 
\[
C=L_{\{i,i+3,i+4\}}\ \mbox{ and }\ E=L_{\{i+2,i+5,i+6\}}. 
\]
We then find $E''=\tau^{-1}(E)=L_{\{i+3,i+4,i+7\}}$, cf.\ Example~\ref{ex:tau-of-rk-1}. 
Since the content of $L_X|L_Y$ has to be contained in 
$\con(E)\cup \con(E'')$, this would imply $i+7\equiv i+1$ modulo $n$, 
i.e., $n=6$, a contradiction to the assumptions. 
\end{proof} 

\begin{figure}\scalebox{.6}{
\hspace{-3cm}
\begin{subfigure}{0.33\textwidth} 
\begin{tikzpicture}  
		\node  (1) at (0, 5) {$A'$};
		\node  (2) at (4, 5) {$L_A$};
		\node  (3) at (8, 5) {$A''$};
		\node  (4) at (2, 3.5) {$L_I$}; 
		\node  (5) at (6, 3.5) {$L_X|L_Y$}; 		
		\node  (6) at (0, 2) {$N'$};
		\node  (7) at (4, 2) {$N$};
		\node  (8) at (8, 2) {$N''$};

		\draw (1) to (4);
		\draw (2) to (4);
		\draw (2) to (5);
		\draw (3) to (5);
		\draw (4) to (6);
		\draw (4) to (7);
		\draw (5) to (7);
		\draw (5) to (8);
\end{tikzpicture}
\caption{}
\end{subfigure}
\hspace{5cm}
\begin{subfigure} {0.5\textwidth}
\begin{tikzpicture}  
		\node  (1) at (0, 5) {$A'$};
		\node  (2) at (4, 5) {$L_A$}; 
		\node  (3) at (8, 5) {$A''$};
		\node  (4) at (0, 3.5) {$C'$};
		\node  (5) at (2, 3.5) {$L_I$};
		\node  (6) at (4, 3.5) {$C$};
		\node  (7) at (6, 3.5) {$L_X|L_Y$}; 		
		\node  (8) at (8, 3.5) {$C''$};
		\node  (9) at (0, 2) {$E'$};
		\node  (10) at (4, 2) {$E$};
		\node  (11) at (8, 2) {$E''$};

		\draw (1) to (5);
		\draw (2) to (5);
		\draw (2) to (7);
		\draw (3) to (7);
		\draw (4) to (5);
		\draw (5) to (6);
		\draw (5) to (9);
		\draw (5) to (10);
		\draw (6) to (7);
		\draw (7) to (8);
		\draw (7) to (10);
		\draw (7) to (11);
\end{tikzpicture}
\caption{}
\end{subfigure}
}
\caption{Part of the Auslander-Reiten quiver of ${\rm CM}(B_{3,n})$.}
\label{fig:AR-sequence-M-decomposable}
\end{figure}
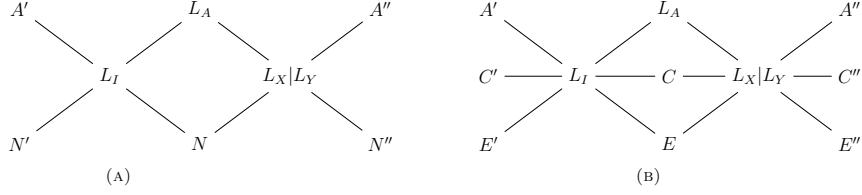

We also have  dual versions of Theorem~\ref{thm:AR-sequences-3-peaks} and  
Lemma~\ref{lm:AR-direct-summand}. 
We leave out their proofs as they are analogous to the above.  

\begin{theorem} \label{thm:AR-sequence-3-peaks-dual}
Let $L_I$ be a rank $1$ module in ${\rm CM}(B_{3,n})$, $n \in \ZZ_{\ge 7}$, where $I=\{i_1, i_2, i_3\}$ 
($i_1<i_2<i_3<i_1$ in the cyclic order) is a $3$-subset with three peaks. 
Then the Auslander-Reiten sequence ending at $L_I$ is 
\[
\frac{L_X}{L_Y}\hookrightarrow M\twoheadrightarrow L_I,
\]
for $X=\{i_1-2, i_2-2, i_3-2\}$ and $Y=\{i_1-1, i_2-1, i_3-1\}$. 

If $M$ is indecomposable then $P_M=X|I|Y$. 
\end{theorem}

\begin{lemma}\label{lm:AR-direct-summand-dual}
In the situation of Theorem~\ref{thm:AR-sequence-3-peaks-dual}, 
if $I=\{i,i+2,i+4\}$ and $M=L_A\oplus N$ for $A=\{1,i+2,i+3\}$, then 
$N$ is indecomposable if and only if $n\ge 7$ and in this case, 
$N=L_X|L_B$ for $B=\{i-1,i+1,i+4\}=I\cup Y\setminus A$.  
\end{lemma}

\subsection{Rank $3$ rigid indecomposable modules in ${\rm CM}(B_{3,9})$}
 
We now characterise the rigid indecomposable rank 3 modules of 
${\rm CM}(B_{3,9})$.

\begin{theorem} \label{thm:rigid-ind-rk3-39}
Let $M$ be a rigid indecomposable rank $3$ module in ${\rm CM}(B_{3,9})$. 
Then either $q(M)=2$ and $M$ is a cyclic permutation of a real canonical profile, 
or $q(M)=0$ and the profile of $M$ is a shift of 
$\begin{small}\begin{array}{c} 157\\ \hline 369\\ \hline 248\end{array}\end{small}$. 
Furthermore, every cyclic permutation of a real canonical profile and every shift of the 
above profile yields a rigid indecomposable module.
\end{theorem}

There are 72 real canonical profiles of rank 3, listed in Table~\ref{table:all-canon-real-rk3}. 
Thus the following is immediate.

\begin{corollary}\label{cor:count-rigid-3-9}
In ${\rm CM}(B_{3,9})$, there are $225$ rigid indecomposable rank $3$ modules. 
Out of them, $216$ correspond to a real root and $9$ to an imaginary root. 
\end{corollary}

\begin{proposition} \label{proposition:non-rigid modules}
If $M$ is an indecomposable module and if its profile is a cyclic permutation of 
the profile $\begin{array}{c} 369\\ \hline 258 \\ \hline 147\end{array}$,
then $M$ is not rigid.  
\end{proposition}

\begin{proof}
If $M$ is indecomposable, with the above profile, then $\tau^{-1}(M)$ has the same profile. 
So $M$ is not rigid by Lemma~\ref{lem:tauM and M have the same profile implies that M is non-rigid}. 
The claim then follows from Lemma~\ref{lem:shift-rigid}
since any cyclic permutation of one of the above profiles is also equal to a shift of this profile.
\end{proof}

Our strategy to prove Theorem~\ref{thm:rigid-ind-rk3-39} is as follows: we will show that for each of the 
predicted profiles there exists a rigid indecomposable module of rank 3. These are the three cyclic 
permutations of the 72 canonical real profiles (Theorem~\ref{thm:k=3_rk=3_real_canonical})
and the nine shifts of the profile of the imaginary root (Corollary~\ref{cor:39-imaginary} 
and Proposition~\ref{proposition:non-rigid modules}), i.e., of 
$\begin{small} 
\begin{array}{c} 157\\ \hline 369 \\ \hline 248 \end{array} 
\end{small}. 
$

The 72 real rank 3 profiles are listed in Table~\ref{table:all-canon-real-rk3}. They give 
216 candidates for rigid indecomposable rank 3 modules. 
By Lemma~\ref{lem:shift-rigid}, it is enough 
to consider them up to a shift (up to adding a constant number to each entry in the profile). 
Since the content of every such profile has seven distinct numbers and one number appearing twice, 
all 9 shifts of a given profile are different. So up to a shift, we only have to consider 24 candidates. 
These candidates are listed in Tables~\ref{table:candidates-39-I}, 
~\ref{table:candidates-39-II}, and ~\ref{table:candidates-39-III}. 
The first entry in Table~\ref{table:candidates-39-I} 
is the representative of the 9 imaginary profiles (which are all shifts of each other). 

We will show that each of these candidates appears in the region of rigid objects in a 
tube of the Auslander-Reiten quiver of ${\rm CM}(B_{3,9})$.

\begin{table}
\begin{tabular}{l}
$
\thfrac{258}{147}{136}, 
\thfrac{259}{147}{136}, 
\thfrac{258}{247}{136}, 
\thfrac{358}{247}{136}, 
\thfrac{259}{247}{136}, 
\thfrac{259}{148}{136}, 
\thfrac{359}{247}{136}, 
\thfrac{259}{248}{136}, 
\thfrac{358}{247}{146}, 
\thfrac{259}{148}{137}, 
\thfrac{359}{248}{136}, 
\thfrac{359}{247}{146}, 
\thfrac{358}{257}{146}, 
\thfrac{269}{148}{137}, 
\thfrac{259}{248}{137}, 
\thfrac{368}{257}{146}, 
\thfrac{359}{257}{146}, 
\thfrac{359}{248}{146}, 
$ \\
\\
$
\thfrac{359}{248}{137}, 
\thfrac{269}{248}{137}, 
\thfrac{269}{158}{137}, 
\thfrac{369}{257}{146}, 
\thfrac{359}{258}{146}, 
\thfrac{369}{248}{137}, 
\thfrac{269}{258}{137}, 
\thfrac{368}{257}{147}, 
\thfrac{359}{248}{147}, 
\thfrac{269}{158}{147}, 
\thfrac{369}{258}{146}, 
\thfrac{369}{258}{137}, 
\thfrac{369}{257}{147}, 
\thfrac{369}{248}{147}, 
\thfrac{369}{158}{147}, 
\thfrac{368}{258}{147}, 
\thfrac{359}{258}{147}, 
\thfrac{269}{258}{147}, 
$ \\
\\ 
$
\thfrac{469}{258}{147}, 
\thfrac{379}{258}{147}, 
\thfrac{369}{358}{147}, 
\thfrac{369}{268}{147}, 
\thfrac{369}{259}{147}, 
\thfrac{369}{258}{247}, 
\thfrac{369}{258}{157}, 
\thfrac{369}{258}{148}, 
\thfrac{469}{358}{147}, 
\thfrac{379}{268}{147}, 
\thfrac{369}{358}{247}, 
\thfrac{469}{258}{157}, 
\thfrac{369}{268}{157}, 
\thfrac{379}{258}{148}, 
\thfrac{369}{259}{148}, 
\thfrac{469}{358}{247}, 
\thfrac{469}{358}{157}, 
\thfrac{469}{268}{157}, 
$ \\
\\ 
$
\thfrac{379}{268}{157}, 
\thfrac{379}{268}{148}, 
\thfrac{379}{259}{148}, 
\thfrac{479}{268}{157}, 
\thfrac{469}{368}{157}, 
\thfrac{379}{269}{148}, 
\thfrac{469}{358}{257}, 
\thfrac{379}{268}{158}, 
\thfrac{479}{368}{157}, 
\thfrac{469}{368}{257}, 
\thfrac{479}{268}{158}, 
\thfrac{379}{269}{158}, 
\thfrac{479}{368}{257}, 
\thfrac{479}{368}{158}, 
\thfrac{479}{269}{158}, 
\thfrac{479}{369}{158}, 
\thfrac{479}{368}{258}, 
\thfrac{479}{369}{258}
$
\end{tabular}
\caption{The canonical profiles of rank $3$ in $C_{3,9}^{\rm re}$.}
\label{table:all-canon-real-rk3}
\end{table}
 
\begin{table}
\[
\begin{array}{c} 4   6   9\\ \hline 2  5  8\\ \hline 1  3  7 \end{array};\hskip 1cm 
\begin{array}{c} 3 6 9\\ \hline 2 5 8\\ \hline 1 3 7 \end{array} ,  \hskip.2cm
\begin{array}{c} 3 7 9\\ \hline 2 5 8\\ \hline 1 4 7 \end{array} , \hskip.2cm
\begin{array}{c} 3 5 9\\ \hline 2 4 7\\ \hline 1 4 6 \end{array} , \hskip.2cm
 \begin{array}{c} 3 5 8\\ \hline 2 5 7\\ \hline 4 6 9 \end{array} , \hskip.2cm
\begin{array}{c} 4 6 9\\ \hline 3 5 8\\ \hline 1 5 7 \end{array}, \hskip.2cm
\begin{array}{c} 4 6 9\\ \hline 2 6 8\\ \hline 1 5 7 \end{array} , \hskip.2cm
\begin{array}{c} 2  6  9\\ \hline 1  5  8\\ \hline 1  3  7 \end{array}, \hskip.2cm
\begin{array}{c} 3 5 9\\ \hline 2 5 7\\ \hline 1 4 6 \end{array}.
\]
\caption{Candidates for rigid indecomposable rank 3 modules in ${\rm CM}(B_{3,9})$, 
cf. Figures~\ref{fig:tube_L124}--\ref{fig:tube_L158}. The first one  is imaginary.} 
\label{table:candidates-39-I}
\end{table}

\begin{table}
\[
\begin{array}{c} 3 5 9\\ \hline 2 4 7\\ \hline 1 3 6 \end{array} , \hskip.2cm
\begin{array}{c} 3 6 9\\ \hline 2 4 8\\ \hline 1 3 7 \end{array} , \hskip.2cm
\begin{array}{c} 4 6 9\\ \hline 3 5 8\\ \hline 1 4 7 \end{array} , \hskip.2cm
\begin{array}{c} 3 5 9\\ \hline 2 4 8\\ \hline 1 3 6 \end{array}.
\]
\caption{Candidates for rigid indecomposable rank 3 modules in ${\rm CM}(B_{3,9})$, 
cf. Figures~\ref{fig:tube_L258L146}--\ref{fig:tube_L268L147}.} 
\label{table:candidates-39-II}
\end{table}

{\small
\begin{table}
\[
\begin{array}{c} 2  6  9\\ \hline 1  5  8\\ \hline 1  4  7 \end{array}, \hskip.2cm
\begin{array}{c} 3  6 9\\\hline  1  5  8\\ \hline 1  4 7 \end{array}, \hskip.2cm
\begin{array}{c} 3 5 9\\\hline  2 5 8\\ \hline 1 4 6 \end{array}, \hskip.2cm
\begin{array}{c} 2 6 9\\\hline  2 5 8\\ \hline 1 4 7 \end{array}, \hskip.2cm
\begin{array}{c} 4 6 9\\ \hline 3 5 8\\ \hline 2 5 7 \end{array} , \hskip.2cm
\begin{array}{c} 3 6 9\\ \hline 3 5 8\\ \hline 1 4 7 \end{array}, \hskip.2cm
\begin{array}{c} 4 6 9\\ \hline 2 5 8 \\ \hline 1 5 7 \end{array} ,  \hskip.2cm
\begin{array}{c} 2 5 8\\ \hline 1 4 7\\ \hline 1 3 6 \end{array} , \hskip.2cm
\begin{array}{c} 3 6 9\\ \hline 2 6 8\\ \hline 1 5 7 \end{array} , \hskip.2cm
\begin{array}{c} 3  7 9\\ \hline 2  6  9\\ \hline 1 5  8 \end{array}, \hskip.2cm
\begin{array}{c} 3 6 8\\ \hline 2 5 7\\ \hline 1 4 6 \end{array} , \hskip.2cm
\begin{array}{c} 3 5 8\\ \hline 2 4 7\\ \hline 1 3 6 \end{array} .
\]
\caption{Candidates for rigid indecomposable rank 3 modules in ${\rm CM}(B_{3,9})$, 
cf. Figures~\ref{fig:tube_L269L158L147}--\ref{fig:tube_L358L247L136}.}
\label{table:candidates-39-III}
\end{table}
}

\subsection{Proof of Theorem~\ref{thm:rigid-ind-rk3-39}} \label{sec:theorem-rigid-ind}

The category ${\rm CM}(B_{3,9})$ is known to be tubular with standard components (see Proposition \ref{propos:AR-tubes}), it means 
that when we study the indecomposable objects and the irreducible maps between them, 
they form tubes which are bounded on one end. 
Consider a tube in the Auslander-Reiten quiver of the category ${\rm CM}(B_{3,9})$. 
This has the shape of an infinite cylinder, bounded on one end. In general, on each level (row), 
it has the same number of vertices which correspond to the indecomposable objects of 
the category. The only exceptions are the tubes which contain projective-injective indecomposables 
at the end, such a row only has half of the entries of the other rows in the tube. 
We will consider the row with projective-injective objects as row 0. Then row 1 is 
called the {\bf mouth} of the tube. The {\bf rank} of a tube is its width. It is a fact that the indecomposable 
objects in rows $1,2,\dots, r-1$ (and in row 0, when present) are all rigid 
(e.g., see \cite{BM}).
 
So what we will prove is that for every candidate in the Tables~\ref{table:candidates-39-I}, 
~\ref{table:candidates-39-II}, and ~\ref{table:candidates-39-III} there exists 
a rank 3 module in the rigid region in a tube of 
the Auslander-Reiten quiver of ${\rm CM}(B_{3,9})$. 
Note that the rank of a module is additive on short exact sequences. This means that in any tube, 
the lowest rank modules are at the mouth and that this rank grows when moving away from the mouth. 

It will thus be enough to consider the set of tubes which have rank 1 modules in their mouth,  the set of 
tubes which have rigid rank 2 modules in their mouth but no rank 1 modules, and the set of tubes which 
have rigid rank 3 modules in their mouth but no rank 2 modules. We call these sets of 
tubes ${\mathcal R}_{\ge 1}$, ${\mathcal R}_{\ge 2}$ and ${\mathcal R}_{\ge 3}$, respectively.

The tubes in the sets ${\mathcal R}_{\ge 1}$ and ${\mathcal R}_{\ge 2}$ have been determined 
in~\cite{BBG}:  the tubes of the subfigures (A), (B), (C), and (D) of Figure 6 in~\cite{BBG} all belong to
$\mathcal R_{\ge 1}$. There are 81 different rigid indecomposable modules of rank 3 in such tubes. 

The tubes of the subfigures (G), (H), (I),  and (J) of Figure 6 in~\cite{BBG} 
belong to $\mathcal R_{\ge 2}$; these tubes contain 36 different rigid indecomposable 
rank 3 modules. 

We will give more details here. For ${\mathcal R}_{\ge 1}$, 
we will give the filtrations of representatives of these rank 3 modules in 
Figures~\ref{fig:tube_L124}--\ref{fig:tube_L158} explicitly and explain how they can be computed. 
For ${\mathcal R}_{\ge 2}$, we give the filtrations of representatives of these rank 3 modules 
in Figures~\ref{fig:tube_L258L146}--
\ref{fig:tube_L268L147}.  

If we can prove that the tubes in $\mathcal R_{\ge 3}$ contain 108 rigid indecomposable 
rank 3 modules, we are done. Since all rank 3 modules are at the mouths of the tubes of 
$\mathcal R_{\ge 3}$, it is enough to compute the first $\tau$-orbit (the mouth) in each tube 
of rank $>1$. For this, we first have to determine the modules for the candidates of 
Tables~\ref{table:candidates-39-I},~\ref{table:candidates-39-II}, 
and~\ref{table:candidates-39-III}
in the 
tubes of $\mathcal R_{\ge 1}$ and of $\mathcal R_{\ge 2}$: they are the ones remaining 
after finding the rank 3 modules in $\mathcal R_{\ge 1}$ and $\mathcal R_{\ge 2}$. 

We give all the remaining representatives in Figures~\ref{fig:tube_L269L158L147}--
\ref{fig:tube_L358L247L136}. 
In total, there are 108 rigid indecomposable rank 3 modules in such tubes, adding up to 
225 rigid indecomposable rank 3 modules in ${\rm CM}(B_{3,9})$ as claimed 
and this will finish the proof. 

In all the figures, we indicate the representatives of the profiles 
(for Tables~\ref{table:candidates-39-I},~\ref{table:candidates-39-II},
~\ref{table:candidates-39-III}) 
by writing them in blue.

\subsection{Tubes with rank 1 modules}

There are four types of tubes in ${\mathcal R}_{\ge 1}$, shown in Figures~\ref{fig:tube_L124}--\ref{fig:tube_L158}. 
We explain how to obtain Figure~\ref{fig:tube_L124}, determining the 
profiles of the rank 3 modules works similarly in the other cases. 

Consider the Auslander-Reiten sequence $M\to M'\oplus N'\to N$ in Figure~\ref{fig:tube_L124} 
with 
\[
M=\begin{array}{c}359\\ \hline 246\end{array}, \hskip .5cm 
M'=\begin{array}{c}135\\ \hline 246\end{array}, \hskip .3cm \mbox{and} \hskip.3cm
N=\begin{array}{c}135\\ \hline 247\end{array}\,\, .
\]
Since the category is tubular, $N'$ is indecomposable, of rank 2. 
It is rigid as it is in the rigid range in this tube. 
By Theorem~\ref{thm:rank2-bound} (1), 
the two $3$-subsets of $N'$ form three quasi-boxes, i.e., are 3-interlacing. 
So the module $N'$ 
is either $L_{359}|L_{247}$ or it is $L_{247}|L_{359}$. Comparing the projective covers 
of $M\oplus N$ and $M'\oplus N'$ shows that $N'=L_{359}|L_{247}$. 
Then we compute $\tau^{-1}(N')$ and the rest of this row using Remark~\ref{rem:profile-tau}. 

To find the profiles of the rank 3 modules in row 5, we consider the Auslander-Reiten sequence 
$M\to M'\oplus N'\to N$ for 
\[
M=\begin{array}{c} 359\\ \hline 247  \end{array}, \hskip.5cm
M'=\begin{array}{c} 135 \\ \hline 247 \end{array}, \hskip.3cm \mbox{ and } \hskip.3cm 
N=\begin{array}{c} 136 \\ \hline 258 \\ \hline 479 \end{array} \,\, .
\]

\begin{figure}
\[
\xymatrix@=0.4em{
&& & 456\ar@{-}[rd]\ar@{-}[ld]& & &&789\ar@{-}[rd] \ar@{-}[ld]&&  &&123\ar@{-}[ld]\ar@{-}[rd]& \\
124\ar@{.}[rr]\ar@{--}[dd]\ar@{-}[rd] && 356 \ar@{.}[rr]\ar@{-}[rd] 
 && 457 \ar@{.}[rr]\ar@{-}[rd] 
 &&689\ar@{.}[rr]\ar@{-}[rd] && 
178\ar@{.}[rr]\ar@{-}[rd]  && 239 \ar@{.}[rr]\ar@{-}[rd] && 124\ar@{--}[dd]  \\
\ar@{.}[r]& \frac{135}{246}\ar@{-}[rd]\ar@{-}[ld]\ar@{.}[rr]\ar@{-}[ru] && 357\ar@{.}[rr]\ar@{-}[ru] 
 && \frac{468}{579}\ar@{-}[rd]\ar@{-}[ld]\ar@{.}[rr]\ar@{-}[ru] && 
 168 \ar@{-}[rd]\ar@{-}[ld]\ar@{.}[rr]\ar@{-}[ru] && \frac{279}{138}\ar@{-}[rd]\ar@{-}[ld]\ar@{.}[rr]\ar@{-}[ru] 
  && 249 \ar@{-}[rd]\ar@{-}[ld]\ar@{.}[r]\ar@{-}[ru] & \\
  \frac{359}{246}\ar@{.}[rr]\ar@{--}[dd]\ar@{-}[rd] && \frac{135}{247}\ar@{-}[ru] \ar@{-}[ld]\ar@{.}[rr]\ar@{-}[rd] 
 && \frac{368}{579}\ar@{-}[lu]\ar@{-}[ld] \ar@{.}[rr]\ar@{-}[rd] 
 &&\frac{468}{157}\ar@{-}[ld]\ar@{.}[rr]\ar@{-}[rd] && 
\frac{269}{138}\ar@{-}[ld]\ar@{.}[rr]\ar@{-}[rd]  &&\frac{279}{148} \ar@{-}[ld]\ar@{.}[rr]\ar@{-}[rd] && \frac{359}{246}\ar@{-}[ld]\ar@{--}[dd]  \\
\ar@{.}[r]& \frac{359}{247}\ar@{-}[rd]\ar@{-}[ld]\ar@{.}[rr]\ar@{-}[ru] && {\tiny\thfrac{136}{258}{479}}\ar@{.}[rr]\ar@{-}[ru] 
 && \frac{368}{157}\ar@{-}[rd]\ar@{-}[ld]\ar@{.}[rr]\ar@{-}[ru] && 
 \textcolor{blue}{{\tiny\thfrac{469}{258}{137}}}\ar@{-}[rd]\ar@{-}[ld]\ar@{.}[rr]\ar@{-}[ru] && \frac{269}{148}\ar@{-}[rd]\ar@{-}[ld]\ar@{.}[rr]\ar@{-}[ru] 
  && {\tiny\thfrac{379}{258}{146}} \ar@{-}[rd]\ar@{-}[ld]\ar@{.}[r]\ar@{-}[ru] & \\
{\tiny\thfrac{379}{258}{147}}\ar@{.}[rr]\ar@{--}[dd]\ar@{-}[ru] && {\tiny\thfrac{369}{258}{479}} \ar@{.}[rr]\ar@{-}[ru] 
 && {\tiny\thfrac{136}{258}{147}}\ar@{.}[rr]\ar@{-}[ru] \ar@{-}[lu]
 &&\textcolor{blue}{{\tiny\thfrac{369}{258}{137}}}\ar@{.}[rr]\ar@{-}[ru] && 
{{\tiny\thfrac{469}{258}{147}}}\ar@{.}[rr]\ar@{-}[ru]  && {\tiny\thfrac{369}{258}{146}} \ar@{.}[rr]\ar@{-}[ru] && \textcolor{blue}{{\tiny\thfrac{379}{258}{147}}}\ar@{--}[dd]  \\
&& && && && && && \\
&& && && && && && \\
}
\]
\caption{The tube containing $L_{124}$.}
\label{fig:tube_L124}
\end{figure}
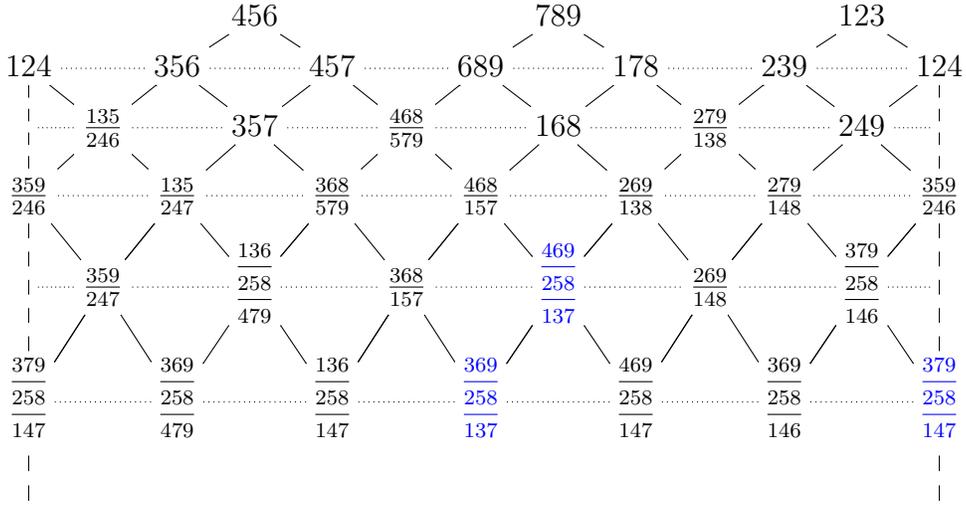

Similarly as before, $N'$ is rigid indecomposable, it is of rank 3. 
The content of $N'$ is $\{2,3,4,5,6,7,8,9,9\}$, so $q(N')=2$, i.e.,  $N'$ corresponds to a real root. 
By Theorem~\ref{thm:k=3_rk=3_real_canonical}, $P_{N'}$ is a cyclic permutation of a canonical profile. 
The only possible choices of $N'$ are the following:
\[
\begin{array}{c} 3 6 9\\ \hline 2 5 8\\ \hline 4 7 9 \end{array}, \hskip.5cm
\begin{array}{c} 4 7 9 \\\hline  3 6 9\\ \hline  2 5 8 \end{array}, \hskip.3cm  \mbox{ and } \hskip.3cm 
\begin{array}{c} 2 5 8 \\ \hline  4 7 9 \\ \hline  3 6 9 \end{array}\,\,.
\]
Comparing the projective covers, 
we find that $N'=L_{369}|L_{258}|L_{479}$. 
Then we use Remark~\ref{rem:profile-tau} to find $\tau^{-1}{N'}$ and so on. 

In Figures~\ref{fig:tube_L124}, \ref{fig:tube_L256}, \ref{fig:tube_L268}, and \ref{fig:tube_L158}, 
we have chosen 
representatives of the rigid rank 3 modules up to a shift. They are coloured in blue.

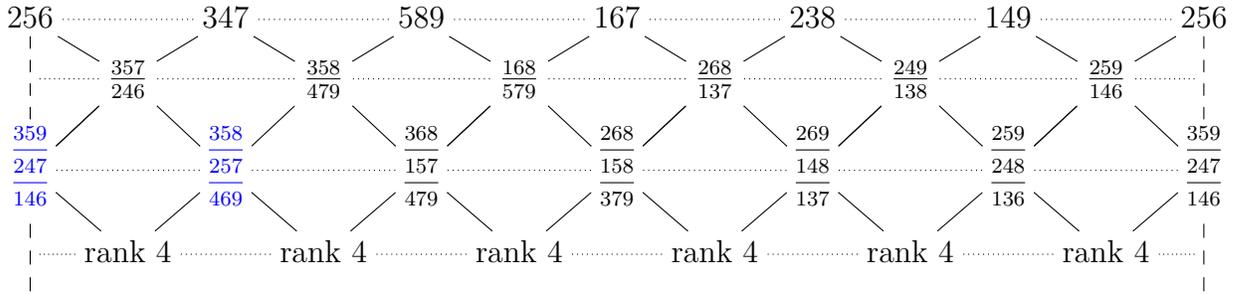
\begin{figure}
\[
\xymatrix@=0.4em{
256\ar@{.}[rr]\ar@{--}[dd]\ar@{-}[rd] && 347 \ar@{.}[rr]\ar@{-}[rd] 
 &&  589 \ar@{.}[rr]\ar@{-}[rd] 
 &&167\ar@{.}[rr]\ar@{-}[rd] && 
238\ar@{.}[rr]\ar@{-}[rd]  && 149 \ar@{.}[rr]\ar@{-}[rd] && 256\ar@{--}[dd]  \\
\ar@{.}[r]& \frac{357}{246}\ar@{-}[rd]\ar@{-}[ld]\ar@{.}[rr]\ar@{-}[ru] && \frac{358}{479}\ar@{.}[rr]\ar@{-}[ru] 
 && \frac{168}{579}\ar@{-}[rd]\ar@{-}[ld]\ar@{.}[rr]\ar@{-}[ru] && 
 \frac{268}{137}\ar@{-}[rd]\ar@{-}[ld]\ar@{.}[rr]\ar@{-}[ru] && \frac{249}{138}\ar@{-}[rd]\ar@{-}[ld]\ar@{.}[rr]\ar@{-}[ru] 
  && \frac{259}{146} \ar@{-}[rd]\ar@{-}[ld]\ar@{.}[r]\ar@{-}[ru] & \\
\textcolor{blue}{{\tiny\thfrac{359}{247}{146}}}\ar@{.}[rr]\ar@{--}[dd]\ar@{-}[ru] && \textcolor{blue}{{\tiny\thfrac{358}{257}{469}}} \ar@{.}[rr]\ar@{-}[ru] 
 && {\tiny\thfrac{368}{157}{479}}\ar@{.}[rr]\ar@{-}[ru] \ar@{-}[lu]
 &&{{\tiny\thfrac{268}{158}{379}}}\ar@{.}[rr]\ar@{-}[ru] && 
{{\tiny\thfrac{269}{148}{137}}}\ar@{.}[rr]\ar@{-}[ru]  && {\tiny\thfrac{259}{248}{136}} \ar@{.}[rr]\ar@{-}[ru] && {{\tiny\thfrac{359}{247}{146}}}\ar@{--}[dd]  \\
\ar@{.}[r]& \mbox{rank 4}   \ar@{-}[lu]\ar@{.}[rr]\ar@{-}[ru] && \mbox{rank 4}  \ar@{.}[rr]\ar@{-}[ru]\ar@{-}[lu] 
 && \mbox{rank 4} \ar@{-}[lu]\ar@{.}[rr]\ar@{-}[ru] && 
\mbox{rank 4} \ar@{-}[lu]\ar@{.}[rr]\ar@{-}[ru] &&  \mbox{rank 4} 
\ar@{-}[lu]\ar@{.}[rr]\ar@{-}[ru] 
  && \mbox{rank 4}  \ar@{-}[lu]\ar@{.}[r]\ar@{-}[ru] & \\
&& && && && && && \\
}
\]
\caption{The tube containing $L_{256}$.}
\label{fig:tube_L256}
\end{figure}

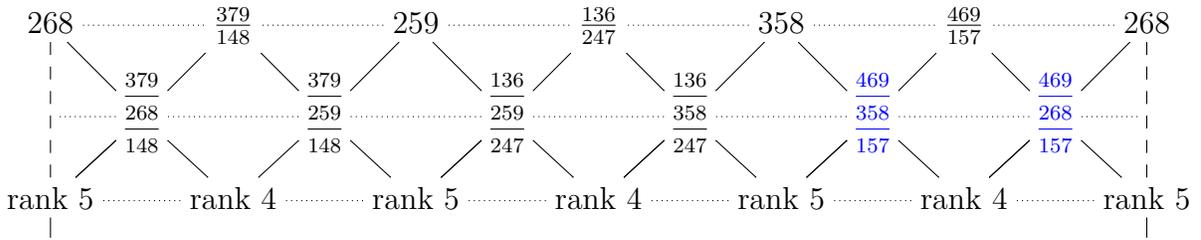
\begin{figure}
\[
\xymatrix@=0.4em{
268\ar@{.}[rr]\ar@{--}[dd]\ar@{-}[rd] && \frac{379}{148} \ar@{.}[rr]\ar@{-}[rd] 
 &&  259 \ar@{.}[rr]\ar@{-}[rd] 
 &&\frac{136}{247}\ar@{.}[rr]\ar@{-}[rd] && 
358\ar@{.}[rr]\ar@{-}[rd]  && \frac{469}{157} \ar@{.}[rr]\ar@{-}[rd] && 268\ar@{--}[dd]  \\
\ar@{.}[r]& {\tiny\thfrac{379}{268}{148}} \ar@{-}[rd]\ar@{-}[ld]\ar@{.}[rr]\ar@{-}[ru] && {\tiny\thfrac{379}{259}{148}}\ar@{.}[rr]\ar@{-}[ru] 
 && {\tiny\thfrac{136}{259}{247}}\ar@{-}[rd]\ar@{-}[ld]\ar@{.}[rr]\ar@{-}[ru] && 
 {\tiny\thfrac{136}{358}{247}}\ar@{-}[rd]\ar@{-}[ld]\ar@{.}[rr]\ar@{-}[ru] && \textcolor{blue}{{\tiny\thfrac{469}{358}{157}}}\ar@{-}[rd]\ar@{-}[ld]\ar@{.}[rr]\ar@{-}[ru] 
  && \textcolor{blue}{{\tiny\thfrac{469}{268}{157}}} \ar@{-}[rd]\ar@{-}[ld]\ar@{.}[r]\ar@{-}[ru] & \\
 \mbox{rank 5} \ar@{.}[rr]\ar@{--}[d]\ar@{-}[ru] &&  \mbox{rank 4} \ar@{.}[rr]\ar@{-}[ru] 
 && \mbox{rank 5}\ar@{.}[rr]\ar@{-}[ru] \ar@{-}[lu]
 &&  \mbox{rank 4} \ar@{.}[rr]\ar@{-}[ru] && 
 \mbox{rank 5} \ar@{.}[rr]\ar@{-}[ru]  &&  \mbox{rank 4} \ar@{.}[rr]\ar@{-}[ru] && 
  \mbox{rank 5} \ar@{--}[d]  \\
&& && && && && && \\
}
\]
\caption{The tube containing $L_{268}$.}
\label{fig:tube_L268}
\end{figure}

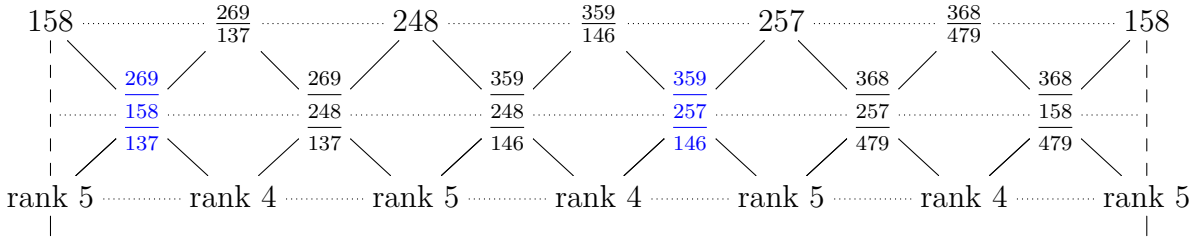
\begin{figure}
\[
\xymatrix@=0.4em{
158\ar@{.}[rr]\ar@{--}[dd]\ar@{-}[rd] && \frac{269}{137} \ar@{.}[rr]\ar@{-}[rd] 
 &&  248 \ar@{.}[rr]\ar@{-}[rd] 
 &&\frac{359}{146}\ar@{.}[rr]\ar@{-}[rd] && 
257\ar@{.}[rr]\ar@{-}[rd]  && \frac{368}{479} \ar@{.}[rr]\ar@{-}[rd] && 158\ar@{--}[dd]  \\
\ar@{.}[r]& \textcolor{blue}{{\tiny\thfrac{269}{158}{137}}} \ar@{-}[rd]\ar@{-}[ld]\ar@{.}[rr]\ar@{-}[ru] && {\tiny\thfrac{269}{248}{137}}\ar@{.}[rr]\ar@{-}[ru] 
 && {\tiny\thfrac{359}{248}{146}}\ar@{-}[rd]\ar@{-}[ld]\ar@{.}[rr]\ar@{-}[ru] && 
 \textcolor{blue}{{\tiny\thfrac{359}{257}{146}}}\ar@{-}[rd]\ar@{-}[ld]\ar@{.}[rr]\ar@{-}[ru] &&{{\tiny\thfrac{368}{257}{479}}}\ar@{-}[rd]\ar@{-}[ld]\ar@{.}[rr]\ar@{-}[ru] 
  && {{\tiny\thfrac{368}{158}{479}}} \ar@{-}[rd]\ar@{-}[ld]\ar@{.}[r]\ar@{-}[ru] & \\
 \mbox{rank 5} \ar@{.}[rr]\ar@{--}[d]\ar@{-}[ru] &&  \mbox{rank 4} \ar@{.}[rr]\ar@{-}[ru] 
 && \mbox{rank 5}\ar@{.}[rr]\ar@{-}[ru] \ar@{-}[lu]
 &&  \mbox{rank 4} \ar@{.}[rr]\ar@{-}[ru] && 
 \mbox{rank 5} \ar@{.}[rr]\ar@{-}[ru]  &&  \mbox{rank 4} \ar@{.}[rr]\ar@{-}[ru] && 
  \mbox{rank 5} \ar@{--}[d]  \\
&& && && && && && \\
}
\]
\caption{The tube containing $L_{158}$.}
\label{fig:tube_L158}
\end{figure}

\subsection{Rigid rank $3$ modules in $\mathcal R_{\ge 2}$}

In this section, we have the rigid rank 3 modules from tubes in $\mathcal R_{\ge 2}$, 
as shown in Figures~\ref{fig:tube_L258L146}--\ref{fig:tube_L268L147}.  
Part of these tubes appear in~\cite[Figure 6]{BBG}. 
Here, we have determined the rows with 
rigid rank 3 modules.

\begin{figure}[h]
\[
\xymatrix@=0.4em{
\frac{258}{146}\ar@{.}[rr]\ar@{--}[dd]\ar@{-}[rd] && \textcolor{blue}{{\tiny\thfrac{359}{247}{136}}} \ar@{.}[rr]\ar@{-}[rd] 
 &&  \frac{258}{479} \ar@{.}[rr]\ar@{-}[rd] 
 &&{\tiny\thfrac{368}{157}{469}}\ar@{.}[rr]\ar@{-}[rd] && 
\frac{258}{137}\ar@{.}[rr]\ar@{-}[rd]  && {\tiny \thfrac{269}{148}{379} }\ar@{.}[rr]\ar@{-}[rd] && \frac{258}{146}\ar@{--}[dd]  \\
\ar@{.}[r]& \mbox{rank 5}   \ar@{-}[lu]\ar@{.}[rr]\ar@{-}[ru] && \mbox{rank 5}  \ar@{.}[rr]\ar@{-}[ru]\ar@{-}[lu] 
 && \mbox{rank 5} \ar@{-}[lu]\ar@{.}[rr]\ar@{-}[ru] && 
\mbox{rank 5} \ar@{-}[lu]\ar@{.}[rr]\ar@{-}[ru] &&  \mbox{rank 5} 
\ar@{-}[lu]\ar@{.}[rr]\ar@{-}[ru] 
  && \mbox{rank 5}  \ar@{-}[lu]\ar@{.}[r]\ar@{-}[ru] & \\
&& && && && && && \\
}
\]
\caption{The tube containing $L_{258}|L_{146}$.}
\label{fig:tube_L258L146}
\end{figure}
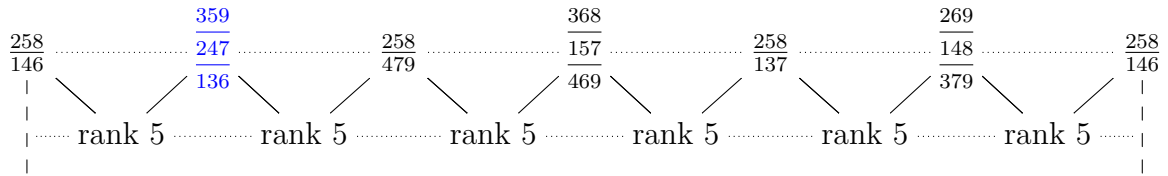

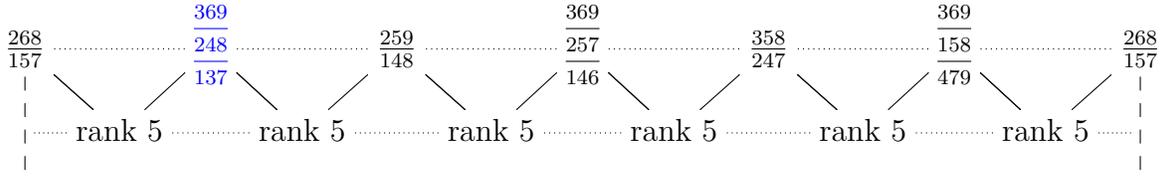
\begin{figure}[h]
\[
\xymatrix@=0.4em{
\frac{268}{157}\ar@{.}[rr]\ar@{--}[dd]\ar@{-}[rd] && \textcolor{blue}{{\tiny\thfrac{369}{248}{137}}} \ar@{.}[rr]\ar@{-}[rd] 
 &&  \frac{259}{148} \ar@{.}[rr]\ar@{-}[rd] 
 &&{\tiny\thfrac{369}{257}{146}}\ar@{.}[rr]\ar@{-}[rd] && 
\frac{358}{247}\ar@{.}[rr]\ar@{-}[rd]  && {\tiny \thfrac{369}{158}{479} }\ar@{.}[rr]\ar@{-}[rd] && \frac{268}{157}\ar@{--}[dd]  \\
\ar@{.}[r]& \mbox{rank 5}   \ar@{-}[lu]\ar@{.}[rr]\ar@{-}[ru] && \mbox{rank 5}  \ar@{.}[rr]\ar@{-}[ru]\ar@{-}[lu] 
 && \mbox{rank 5} \ar@{-}[lu]\ar@{.}[rr]\ar@{-}[ru] && 
\mbox{rank 5} \ar@{-}[lu]\ar@{.}[rr]\ar@{-}[ru] &&  \mbox{rank 5} 
\ar@{-}[lu]\ar@{.}[rr]\ar@{-}[ru] 
  && \mbox{rank 5}  \ar@{-}[lu]\ar@{.}[r]\ar@{-}[ru] & \\
&& && && && && && \\
}
\]
\caption{The tube containing $L_{268}|L_{157}$.}
\label{fig:tube_L268L146}
\end{figure}

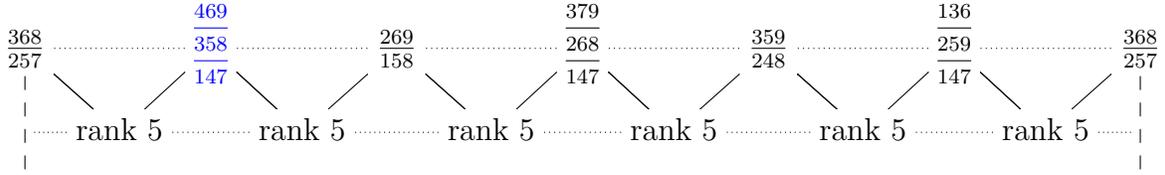
\begin{figure}[h]
\[
\xymatrix@=0.4em{
\frac{368}{257}\ar@{.}[rr]\ar@{--}[dd]\ar@{-}[rd] && \textcolor{blue}{{\tiny\thfrac{469}{358}{147}}} \ar@{.}[rr]\ar@{-}[rd] 
 &&  \frac{269}{158} \ar@{.}[rr]\ar@{-}[rd] 
 &&{\tiny\thfrac{379}{268}{147}}\ar@{.}[rr]\ar@{-}[rd] && 
\frac{359}{248}\ar@{.}[rr]\ar@{-}[rd]  && {\tiny \thfrac{136}{259}{147} }\ar@{.}[rr]\ar@{-}[rd] && \frac{368}{257}\ar@{--}[dd]  \\
\ar@{.}[r]& \mbox{rank 5}   \ar@{-}[lu]\ar@{.}[rr]\ar@{-}[ru] && \mbox{rank 5}  \ar@{.}[rr]\ar@{-}[ru]\ar@{-}[lu] 
 && \mbox{rank 5} \ar@{-}[lu]\ar@{.}[rr]\ar@{-}[ru] && 
\mbox{rank 5} \ar@{-}[lu]\ar@{.}[rr]\ar@{-}[ru] &&  \mbox{rank 5} 
\ar@{-}[lu]\ar@{.}[rr]\ar@{-}[ru] 
  && \mbox{rank 5}  \ar@{-}[lu]\ar@{.}[r]\ar@{-}[ru] & \\
&& && && && && && \\
}
\]
\caption{The tube containing $L_{368}|L_{257}$.}
\label{fig:tube_L368L257}
\end{figure}

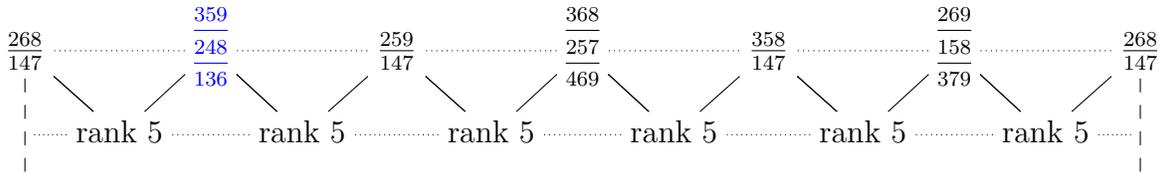
\begin{figure}[h]
\[
\xymatrix@=0.4em{
\frac{268}{147}\ar@{.}[rr]\ar@{--}[dd]\ar@{-}[rd] && \textcolor{blue}{{\tiny\thfrac{359}{248}{136}}} \ar@{.}[rr]\ar@{-}[rd] 
 &&  \frac{259}{147} \ar@{.}[rr]\ar@{-}[rd] 
 &&{\tiny\thfrac{368}{257}{469}}\ar@{.}[rr]\ar@{-}[rd] && 
\frac{358}{147}\ar@{.}[rr]\ar@{-}[rd]  && {\tiny \thfrac{269}{158}{379} }\ar@{.}[rr]\ar@{-}[rd] && \frac{268}{147}\ar@{--}[dd]  \\
\ar@{.}[r]& \mbox{rank 5}   \ar@{-}[lu]\ar@{.}[rr]\ar@{-}[ru] && \mbox{rank 5}  \ar@{.}[rr]\ar@{-}[ru]\ar@{-}[lu] 
 && \mbox{rank 5} \ar@{-}[lu]\ar@{.}[rr]\ar@{-}[ru] && 
\mbox{rank 5} \ar@{-}[lu]\ar@{.}[rr]\ar@{-}[ru] &&  \mbox{rank 5} 
\ar@{-}[lu]\ar@{.}[rr]\ar@{-}[ru] 
  && \mbox{rank 5}  \ar@{-}[lu]\ar@{.}[r]\ar@{-}[ru] & \\
&& && && && && && \\
}
\]
\caption{The tube containing $L_{268}|L_{147}$.}
\label{fig:tube_L268L147}
\end{figure}

\subsection{Rigid rank 3 modules in $\mathcal R_{\ge 3}$}

For the tubes in Figures \ref{fig:tube_L269L158L147}--\ref{fig:tube_L358L247L136}, 
we only need to establish the first row, using the strategy from Remark~\ref{rem:profile-tau}, 
starting with the remaining candidates for rank 3 modules from 
Tables~\ref{table:candidates-39-I}, 
~\ref{table:candidates-39-II}, and~\ref{table:candidates-39-III}.

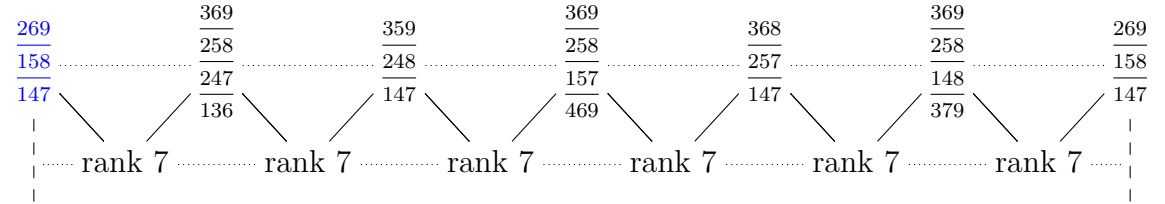
\begin{figure}
\[
\xymatrix@=0.4em{
\textcolor{blue}{\tiny \thfrac{269}{158}{147}}\ar@{.}[rr]\ar@{--}[dd]\ar@{-}[rd] && {\tiny\ffrac{369}{258}{247}{136}} \ar@{.}[rr]\ar@{-}[rd] 
 && {\tiny \thfrac{359}{248}{147}} \ar@{.}[rr]\ar@{-}[rd] 
 &&{\tiny\ffrac{369}{258}{157}{469}}\ar@{.}[rr]\ar@{-}[rd] && 
{\tiny \thfrac{368}{257}{147}} \ar@{.}[rr]\ar@{-}[rd]  && {\tiny\ffrac{369}{258}{148}{379}}\ar@{.}[rr]\ar@{-}[rd] && {\tiny \thfrac{269}{158}{147}}\ar@{--}[dd]  \\
\ar@{.}[r]& \mbox{rank 7}   \ar@{-}[lu]\ar@{.}[rr]\ar@{-}[ru] && \mbox{rank 7}  \ar@{.}[rr]\ar@{-}[ru]\ar@{-}[lu] 
 && \mbox{rank 7} \ar@{-}[lu]\ar@{.}[rr]\ar@{-}[ru] && 
\mbox{rank 7} \ar@{-}[lu]\ar@{.}[rr]\ar@{-}[ru] &&  \mbox{rank 7} 
\ar@{-}[lu]\ar@{.}[rr]\ar@{-}[ru] 
  && \mbox{rank 7}  \ar@{-}[lu]\ar@{.}[r]\ar@{-}[ru] & \\
&& && && && && && \\
}
\]
\caption{The tube containing $L_{269}|L_{158}|L_{147}$.}
\label{fig:tube_L269L158L147}
\end{figure}

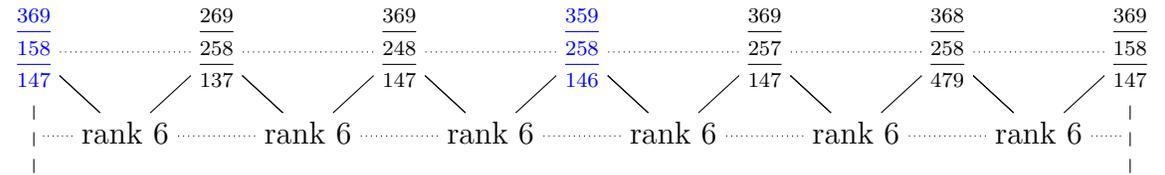
\begin{figure}
\[
\xymatrix@=0.4em{
\textcolor{blue}{\tiny \thfrac{369}{158}{147}}\ar@{.}[rr]\ar@{--}[dd]\ar@{-}[rd] && {\tiny \thfrac{269}{258}{137}} \ar@{.}[rr]\ar@{-}[rd] 
 && {\tiny \thfrac{369}{248}{147}} \ar@{.}[rr]\ar@{-}[rd] 
 &&\textcolor{blue}{{\tiny \thfrac{359}{258}{146}}}\ar@{.}[rr]\ar@{-}[rd] && 
{\tiny \thfrac{369}{257}{147}} \ar@{.}[rr]\ar@{-}[rd]  && {\tiny \thfrac{368}{258}{479}}\ar@{.}[rr]\ar@{-}[rd] && {\tiny \thfrac{369}{158}{147}}\ar@{--}[dd]  \\
\ar@{.}[r]& \mbox{rank 6}   \ar@{-}[lu]\ar@{.}[rr]\ar@{-}[ru] && \mbox{rank 6}  \ar@{.}[rr]\ar@{-}[ru]\ar@{-}[lu] 
 && \mbox{rank 6} \ar@{-}[lu]\ar@{.}[rr]\ar@{-}[ru] && 
\mbox{rank 6} \ar@{-}[lu]\ar@{.}[rr]\ar@{-}[ru] &&  \mbox{rank 6} 
\ar@{-}[lu]\ar@{.}[rr]\ar@{-}[ru] 
  && \mbox{rank 6}  \ar@{-}[lu]\ar@{.}[r]\ar@{-}[ru] & \\
&& && && && && && \\
}
\]
\caption{The tube containing $L_{369}|L_{158}|L_{147}$.}
\label{fig:tube start with L369L158L147}
\label{fig:tube_L369L158L147}
\end{figure}

\begin{figure}
\[
\xymatrix@=0.4em{
\textcolor{blue}{\tiny \thfrac{269}{258}{147}}\ar@{.}[rr]\ar@{--}[dd]\ar@{-}[rd] && {\tiny\ffrac{369}{358}{247}{146}} \ar@{.}[rr]\ar@{-}[rd] 
 && {\tiny \thfrac{359}{258}{147}} \ar@{.}[rr]\ar@{-}[rd] 
 &&{\tiny\ffrac{369}{268}{157}{479}}\ar@{.}[rr]\ar@{-}[rd] && 
{\tiny \thfrac{368}{258}{147}} \ar@{.}[rr]\ar@{-}[rd]  && {\tiny\ffrac{369}{259}{148}{137}}\ar@{.}[rr]\ar@{-}[rd] && {\tiny \thfrac{269}{258}{147}}\ar@{--}[dd]  \\
\ar@{.}[r]& \mbox{rank 7}   \ar@{-}[lu]\ar@{.}[rr]\ar@{-}[ru] && \mbox{rank 7}  \ar@{.}[rr]\ar@{-}[ru]\ar@{-}[lu] 
 && \mbox{rank 7} \ar@{-}[lu]\ar@{.}[rr]\ar@{-}[ru] && 
\mbox{rank 7} \ar@{-}[lu]\ar@{.}[rr]\ar@{-}[ru] &&  \mbox{rank 7} 
\ar@{-}[lu]\ar@{.}[rr]\ar@{-}[ru] 
  && \mbox{rank 7}  \ar@{-}[lu]\ar@{.}[r]\ar@{-}[ru] & \\
&& && && && && && \\
}
\]
\caption{The tube containing $L_{269}|L_{258}|L_{147}$.}
\label{fig:tube_L269L258L147}
\end{figure}
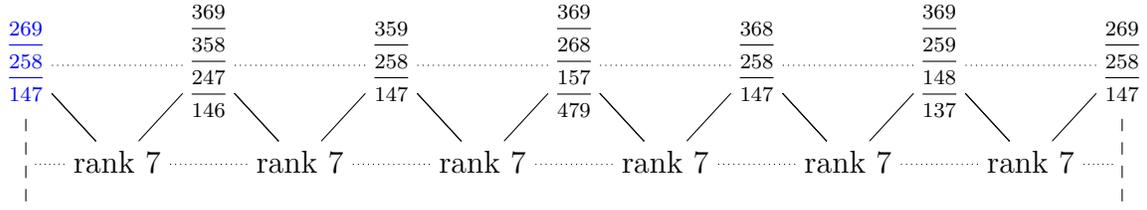

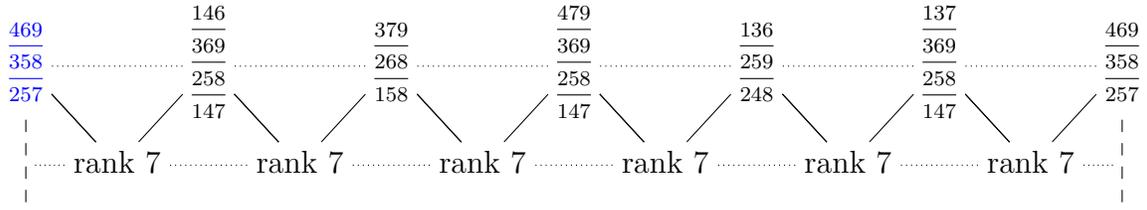
\begin{figure}
\[
\xymatrix@=0.4em{
\textcolor{blue}{\tiny \thfrac{469}{358}{257}}\ar@{.}[rr]\ar@{--}[dd]\ar@{-}[rd] && {\tiny\ffrac{146}{369}{258}{147}} \ar@{.}[rr]\ar@{-}[rd] 
 && {\tiny \thfrac{379}{268}{158}} \ar@{.}[rr]\ar@{-}[rd] 
 &&{\tiny\ffrac{479}{369}{258}{147}}\ar@{.}[rr]\ar@{-}[rd] && 
{\tiny \thfrac{136}{259}{248}} \ar@{.}[rr]\ar@{-}[rd]  && {\tiny\ffrac{137}{369}{258}{147}}\ar@{.}[rr]\ar@{-}[rd] && {\tiny \thfrac{469}{358}{257}}\ar@{--}[dd]  \\
\ar@{.}[r]& \mbox{rank 7}   \ar@{-}[lu]\ar@{.}[rr]\ar@{-}[ru] && \mbox{rank 7}  \ar@{.}[rr]\ar@{-}[ru]\ar@{-}[lu] 
 && \mbox{rank 7} \ar@{-}[lu]\ar@{.}[rr]\ar@{-}[ru] && 
\mbox{rank 7} \ar@{-}[lu]\ar@{.}[rr]\ar@{-}[ru] &&  \mbox{rank 7} 
\ar@{-}[lu]\ar@{.}[rr]\ar@{-}[ru] 
  && \mbox{rank 7}  \ar@{-}[lu]\ar@{.}[r]\ar@{-}[ru] & \\
&& && && && && && \\
}
\]
\caption{The tube containing $L_{469}|L_{358}|L_{257}$.}
\label{fig:tube_L469L358L257}
\end{figure}

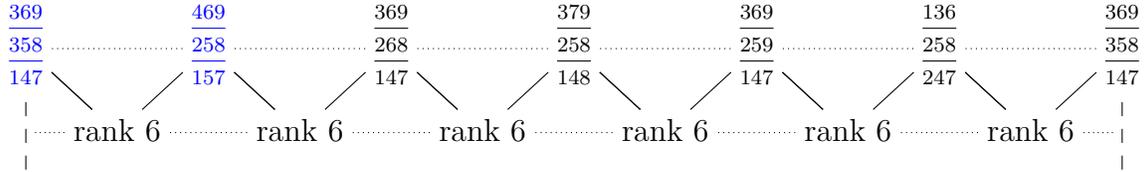
\begin{figure}
\[
\xymatrix@=0.4em{
\textcolor{blue}{\tiny \thfrac{369}{358}{147}}\ar@{.}[rr]\ar@{--}[dd]\ar@{-}[rd] && 
\textcolor{blue}{{\tiny \thfrac{469}{258}{157}}} \ar@{.}[rr]\ar@{-}[rd] 
 && {\tiny \thfrac{369}{268}{147}} \ar@{.}[rr]\ar@{-}[rd] 
 &&{{\tiny \thfrac{379}{258}{148}}}\ar@{.}[rr]\ar@{-}[rd] && 
{\tiny \thfrac{369}{259}{147}} \ar@{.}[rr]\ar@{-}[rd]  && {\tiny \thfrac{136}{258}{247}}\ar@{.}[rr]\ar@{-}[rd] && {\tiny \thfrac{369}{358}{147}}\ar@{--}[dd]  \\
\ar@{.}[r]& \mbox{rank 6}   \ar@{-}[lu]\ar@{.}[rr]\ar@{-}[ru] && \mbox{rank 6}  \ar@{.}[rr]\ar@{-}[ru]\ar@{-}[lu] 
 && \mbox{rank 6} \ar@{-}[lu]\ar@{.}[rr]\ar@{-}[ru] && 
\mbox{rank 6} \ar@{-}[lu]\ar@{.}[rr]\ar@{-}[ru] &&  \mbox{rank 6} 
\ar@{-}[lu]\ar@{.}[rr]\ar@{-}[ru] 
  && \mbox{rank 6}  \ar@{-}[lu]\ar@{.}[r]\ar@{-}[ru] & \\
&& && && && && && \\
}
\]
\caption{The tube containing $L_{369}|L_{358}|L_{147}$.}
\label{fig:tube_L369L358L147}
\end{figure}

\begin{figure}
\[
\xymatrix@=0.4em{
\textcolor{blue}{\tiny \thfrac{258}{147}{136}}\ar@{.}[rr]\ar@{--}[dd]\ar@{-}[rd] && {\tiny\ffrac{259}{248}{137}{369}} \ar@{.}[rr]\ar@{-}[rd] 
 && {\tiny \thfrac{258}{147}{469}} \ar@{.}[rr]\ar@{-}[rd] 
 &&{\tiny\ffrac{358}{257}{146}{369}}\ar@{.}[rr]\ar@{-}[rd] && 
{\tiny \thfrac{258}{147}{379}} \ar@{.}[rr]\ar@{-}[rd]  && {\tiny\ffrac{268}{158}{479}{369}}\ar@{.}[rr]\ar@{-}[rd] && {\tiny \thfrac{258}{147}{136}}\ar@{--}[dd]  \\
\ar@{.}[r]& \mbox{rank 7}   \ar@{-}[lu]\ar@{.}[rr]\ar@{-}[ru] && \mbox{rank 7}  \ar@{.}[rr]\ar@{-}[ru]\ar@{-}[lu] 
 && \mbox{rank 7} \ar@{-}[lu]\ar@{.}[rr]\ar@{-}[ru] && 
\mbox{rank 7} \ar@{-}[lu]\ar@{.}[rr]\ar@{-}[ru] &&  \mbox{rank 7} 
\ar@{-}[lu]\ar@{.}[rr]\ar@{-}[ru] 
  && \mbox{rank 7}  \ar@{-}[lu]\ar@{.}[r]\ar@{-}[ru] & \\
&& && && && && && \\
}
\]
\caption{The tube containing $L_{258}|L_{147}|L_{136}$.}
\label{fig:tube_L258L147L136}
\end{figure}
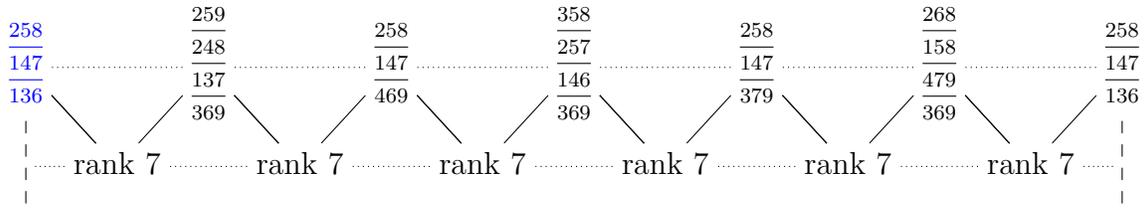

\begin{figure}
\[
\xymatrix@=0.4em{
\textcolor{blue}{\tiny \thfrac{369}{268}{157}}\ar@{.}[rr]\ar@{--}[dd]\ar@{-}[rd] && {\tiny\ffrac{479}{368}{258}{147}} \ar@{.}[rr]\ar@{-}[rd] 
 && {\tiny \thfrac{369}{259}{148}} \ar@{.}[rr]\ar@{-}[rd] 
 &&{\tiny\ffrac{137}{269}{258}{147}}\ar@{.}[rr]\ar@{-}[rd] && 
{\tiny \thfrac{369}{358}{247}} \ar@{.}[rr]\ar@{-}[rd]  && {\tiny\ffrac{146}{359}{258}{147}}\ar@{.}[rr]\ar@{-}[rd] && {\tiny \thfrac{369}{268}{157}}\ar@{--}[dd]  \\
\ar@{.}[r]& \mbox{rank 7}   \ar@{-}[lu]\ar@{.}[rr]\ar@{-}[ru] && \mbox{rank 7}  \ar@{.}[rr]\ar@{-}[ru]\ar@{-}[lu] 
 && \mbox{rank 7} \ar@{-}[lu]\ar@{.}[rr]\ar@{-}[ru] && 
\mbox{rank 7} \ar@{-}[lu]\ar@{.}[rr]\ar@{-}[ru] &&  \mbox{rank 7} 
\ar@{-}[lu]\ar@{.}[rr]\ar@{-}[ru] 
  && \mbox{rank 7}  \ar@{-}[lu]\ar@{.}[r]\ar@{-}[ru] & \\
&& && && && && && \\
}
\]
\caption{The tube containing $L_{369}|L_{268}|L_{157}$.}
\label{fig:tube start with L369L268L157}
\label{fig:tube_L369L268L157}
\end{figure}
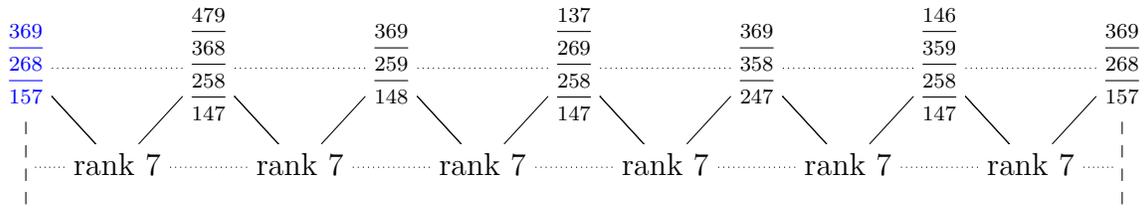

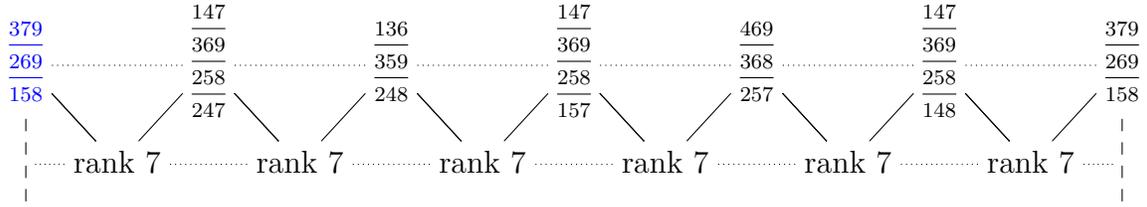
\begin{figure}
\[
\xymatrix@=0.4em{
\textcolor{blue}{\tiny \thfrac{379}{269}{158}}\ar@{.}[rr]\ar@{--}[dd]\ar@{-}[rd] && {\tiny\ffrac{147}{369}{258}{247}} \ar@{.}[rr]\ar@{-}[rd] 
 && {\tiny \thfrac{136}{359}{248}} \ar@{.}[rr]\ar@{-}[rd] 
 &&{\tiny\ffrac{147}{369}{258}{157}}\ar@{.}[rr]\ar@{-}[rd] && 
{\tiny \thfrac{469}{368}{257}} \ar@{.}[rr]\ar@{-}[rd]  && {\tiny\ffrac{147}{369}{258}{148}}\ar@{.}[rr]\ar@{-}[rd] && {\tiny \thfrac{379}{269}{158}}\ar@{--}[dd]  \\
\ar@{.}[r]& \mbox{rank 7}   \ar@{-}[lu]\ar@{.}[rr]\ar@{-}[ru] && \mbox{rank 7}  \ar@{.}[rr]\ar@{-}[ru]\ar@{-}[lu] 
 && \mbox{rank 7} \ar@{-}[lu]\ar@{.}[rr]\ar@{-}[ru] && 
\mbox{rank 7} \ar@{-}[lu]\ar@{.}[rr]\ar@{-}[ru] &&  \mbox{rank 7} 
\ar@{-}[lu]\ar@{.}[rr]\ar@{-}[ru] 
  && \mbox{rank 7}  \ar@{-}[lu]\ar@{.}[r]\ar@{-}[ru] & \\
&& && && && && && \\
}
\]
\caption{The tube containing $L_{379}|L_{269}|L_{158}$.}
\label{fig:tube_L379L269L158}
\end{figure}

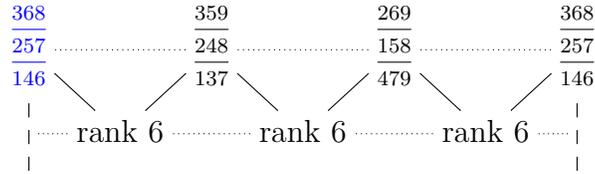
\begin{figure}
\[
\xymatrix@=0.4em{
\textcolor{blue}{\tiny \thfrac{368}{257}{146}}\ar@{.}[rr]\ar@{--}[dd]\ar@{-}[rd] && {\tiny \thfrac{359}{248}{137}}\ar@{-}[rd] \ar@{.}[rr]
 && {\tiny \thfrac{269}{158}{479}} \ar@{.}[rr] \ar@{-}[rd]
 &&{{\tiny \thfrac{368}{257}{146}}}\ar@{--}[dd]\\
\ar@{.}[r]&  \mbox{rank 6} \ar@{.}[rr]\ar@{-}[ru] &&  \mbox{rank 6} \ar@{.}[rr]\ar@{-}[ru] 
 &&  \mbox{rank 6} \ar@{-}[ru]\ar@{.}[r]&  \\
&& && && \\
}
\]
\caption{The tube containing $L_{368}|L_{257}|L_{146}$.}
\label{fig:tube_L368L257L146}
\end{figure}

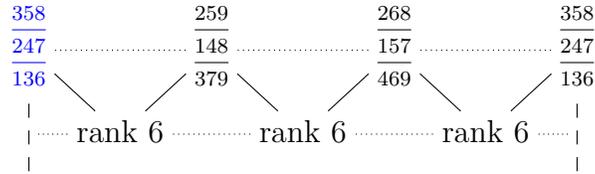
\begin{figure}
\[
\xymatrix@=0.4em{
\textcolor{blue}{\tiny \thfrac{358}{247}{136}}\ar@{.}[rr]\ar@{--}[dd]\ar@{-}[rd] && {\tiny \thfrac{259}{148}{379}}\ar@{-}[rd] \ar@{.}[rr]
 && {\tiny \thfrac{268}{157}{469}} \ar@{.}[rr] \ar@{-}[rd]
 &&{{\tiny \thfrac{358}{247}{136}}}\ar@{--}[dd]\\
\ar@{.}[r]&  \mbox{rank 6} \ar@{.}[rr]\ar@{-}[ru] &&  \mbox{rank 6} \ar@{.}[rr]\ar@{-}[ru] 
 &&  \mbox{rank 6} \ar@{-}[ru]\ar@{.}[r]&  \\
&& && && \\
}
\]
\caption{The tube containing $L_{358}|L_{247}|L_{136}$.}
\label{fig:tube_L358L247L136}
\end{figure}

\clearpage
\appendix

\section{Reduction techniques}\label{sec:reductions}

Here we present tools for studying rigidity of modules. We define maps on modules in ${\rm CM}(B_{k,n})$, 
on the one hand to modules for $B_{k+1,n+1}$ and $B_{k,n+1}$, Definition~\ref{def:1-increase}, 
on the other hand to modules for $B_{k-1,n-1}$ and to $B_{k,n-1}$, Definition~\ref{def:1-decrease}. 
The latter is related to the collapse map from Definition~\ref{def:collapse-a-shift}. 

Let $M\in {\rm CM}(B_{k,n})$ be a rank $s$ module. For $i=1,\dots, n$, let $V_i=\mathbb{C}[[t]]^s$. 
We write $M$ as a representation of $B_{k,n}$ with matrices $x_i,y_i$ of size $s\times s$: 

\begin{equation}\label{eq:M-rep}
\xymatrix@=0.7em{
 V_n \ar@<3pt>[rr]^{x_1} && V_1\ar@<3pt>[rr]^{x_2}\ar@<3pt>[ll]^{y_1} && V_2\ar@<3pt>[rr]^{x_3}\ar@<3pt>[ll]^{y_2} && 
 V_3\ar@<3pt>[rr]\ar@<3pt>[ll]^{y_3}  && \cdots \ar@<3pt>[ll] \ar@<3pt>[rr] && 
 V_{n-1}\ar@<3pt>[rr]^{x_n}\ar@<3pt>[ll] && V_n\ar@<3pt>[ll]^{y_n} 
}
\end{equation}

The first map we define adds a vertex to the circular graph $C$ from Section~\ref{ssec:CM-setup} 
(see Figure~\ref{fig:graph C and quiver Q_C}). 
The map changes the algebra $B_{k,n}$ by increasing $n$ by 1. 
There are two versions of this - we will either increase the size of the $k$-subsets, hence going from 
$B_{k,n}$ to $B_{k+1,n+1}$, or we keep $k$ fixed and increase the size of the complements, going from 
$B_{k,n}$ to $B_{k,n+1}$. 
Since we will move between these circular graphs of varying sizes, we write 
$C(n)$ for the circular graph on $n$ vertices and $Q_{C(n)}$ for the corresponding quiver.

Note that since in all our modules, the maps $x_i$ and $y_i$ satisfy $x_iy_i=t\Id$, they are invertible as matrices over  $\mathbb{C}((t))$. 

Note also that the central element $t=\sum x_iy_i$ will change when increasing or decreasing the number of vertices of the circular quiver. As we identify $\mathbb{C}[[\sum_{i=1}^n x_iy_i]]$ with $\mathbb{C}[[\sum_{i=1}^{n+1} x_iy_i]]$, by abuse of notation, $t$ will always denote the element $\sum_{i=1}^m x_i y_i$ of $Q(m)$.

\begin{definition}[1-increase and 1-co-increase]\label{def:1-increase}
Let $M\in {\rm CM}(B_{k,n})$ be as in (\ref{eq:M-rep}) with $V_i=\mathbb{C}[[t]]^s$ for $i=1,\dots, n$. Define 
$V_{n+1}=\mathbb{C}[[t]]^s$. 

(a) The {\bf 1-increase of $M$ at $j$} is the representation 
$$\inc_j(M)=(V_i, i\in [n+1]; x_1,\dots, x_{j}, \Id, x_{j+1},,\dots, x_{n}, 
y_1,\dots, y_{j},t\Id, y_j,\dots, y_n)$$ 
of $Q_{C(n+1)}$: 
\begin{equation}\label{eq:inc-j-M}
\xymatrix@=0.5em{
 V_{n+1} \ar@<3pt>[rr]^{x_1} && V_1\ar@<3pt>[rr]^{x_2}\ar@<3pt>[ll]^{y_1} && V_2\ar@<3pt>[rr]^{x_3}\ar@<3pt>[ll]^{y_2} && 
 V_3\ar@<3pt>[rr]\ar@<3pt>[ll]^{y_3}  && \cdots \ar@<3pt>[ll] \ar@<3pt>[rr]^{x_j}
 && V_j \ar@<3pt>[ll]^{y_j} \ar@<3pt>[rr]^{\Id}
 && V_{j+1} \ar@<3pt>[ll]^{t\Id} \ar@<3pt>[rr]^{x_{j+1}}
 && \cdots \ar@<3pt>[ll]^{y_{j+1}} \ar@<3pt>[rr] && 
 V_{n}\ar@<3pt>[rr]^{x_n}\ar@<3pt>[ll] && V_{n+1}\ar@<3pt>[ll]^{y_n} 
}
\end{equation}

(b) 
The {\bf 1-co-increase of $M$ at $j$} is the representation 
$$\inc_j^c(M)=(V_i, i\in [n+1]; 
x_1,\dots, x_{j}, t\Id,x_{j+1},\dots, x_{n+1}, y_1,\dots, y_{j},\Id, y_{j+1},\dots, y_{n+1} )$$ 
of $Q_{C(n+1)}$: 
\begin{equation}\label{eq:coinc-j-M}
\xymatrix@=0.5em{
 V_{n+1} \ar@<3pt>[rr]^{x_1} && V_1\ar@<3pt>[rr]^{x_2}\ar@<3pt>[ll]^{y_1} && V_2\ar@<3pt>[rr]^{x_3}\ar@<3pt>[ll]^{y_2} && 
 V_3\ar@<3pt>[rr]\ar@<3pt>[ll]^{y_3}  && \cdots \ar@<3pt>[ll] \ar@<3pt>[rr]^{x_j}
 && V_j \ar@<3pt>[ll]^{y_j} \ar@<3pt>[rr]^{t\Id}
 && V_{j+1} \ar@<3pt>[ll]^{\Id} \ar@<3pt>[rr]^{x_{j+1}}
 && \cdots \ar@<3pt>[ll]^{y_{j+1}} \ar@<3pt>[rr] && 
 V_{n}\ar@<3pt>[rr]^{x_n}\ar@<3pt>[ll] && V_{n+1}\ar@<3pt>[ll]^{y_n} 
}
\end{equation}
\end{definition}

\begin{lemma}\label{A2}
Let $M$ be a $B_{k,n}$-module and $j\in[n]$. Then 
$\inc_j(M)$ is a $B_{k+1,n+1}$-module and 
$\inc_j^c(M)$ is a $B_{k,n+1}$-module.
\end{lemma}

\begin{proof}
We prove the statement for the increase $\inc_j(M)$ of $M$, 
the one about the co-increase follows 
similarly. 
We check that this module satisfies all the relations for $B_{k+1,n+1}$. 
Without loss of generality let $j=n$ and so as a representation, 
$$
\inc_n(M):\quad 
\xymatrix@=0.7em{
 V_{n+1} \ar@<3pt>[rr]^{x_1} && V_1\ar@<3pt>[rr]^{x_2}\ar@<3pt>[ll]^{y_1} && V_2\ar@<3pt>[rr]^{x_3}\ar@<3pt>[ll]^{y_2} && 
 V_3\ar@<3pt>[rr]\ar@<3pt>[ll]^{y_3}  && \cdots \ar@<3pt>[ll] \ar@<3pt>[rr] && 
 V_{n-1}\ar@<3pt>[rr]^{x_n}\ar@<3pt>[ll] && V_n\ar@<3pt>[rr]^{\Id}\ar@<3pt>[ll]^{y_n} && V_{n+1}\ar@<3pt>[ll]^{t\Id} 
}
$$
To see that $\inc_n(M)$ is in 
${\rm CM}(B_{k+1,n+1})$, we have to check that the relations $xy=yx$ and $x^{k+1}=y^{n+1-{k+1}}=y^{n-k}$ hold everywhere. 
The relations $xy=yx$ are clear as for the maps of $M$, $x_iy_i=t\Id$, for $i=1,\dots, n$, and as 
the only new maps are $x_{n+1}=\Id$ and $y_{n+1}=t\Id$ which satisfy $x_{n+1}y_{n+1}=t\Id$, too.  To prove that $x^{k+1}=y^{n-k}$, note that this relation is equivalent to $x^{n+1}=t^{n-k}\Id$. Since for the module $M$ we know that $x^n=t^{n-k}\Id$, and since $x_{n+1}=\Id$, it follows that, for the module  $\inc_n(M)$, $x^{n+1}=t^{n-k}\Id.$
So $\inc_n(M)$ is indeed a representation of $B_{k+1,n+1}$. 
\end{proof}

The following lemma follows directly 
from the construction. 
\begin{lemma}\label{lm:freeness}
Let $M\in {\rm CM}(B_{k,n})$. 
Then for every $j\in[n]$, the $1$-increase $\inc_j(M)$ is free 
over the center of $B_{k+1,n+1}$ and the 
$1$-co-increase $\inc_j^c(M)$ is free over 
the center of $B_{k,n+1}$. 
\end{lemma}

\begin{lemma}\label{lm:hom-iso}
Let $M$ and $N$ be in ${\rm CM}(B_{k,n})$. 
Then for every $j\in[n]$ there are isomorphisms 
$$
\begin{array}{lcl}
 \Hom_{B_{k,n}}(M,N)  & \cong  & \Hom_{B_{k+1,n+1}}(\inc_j(M),\inc_j(N)) \\
 \Hom_{B_{k,n}}(M,N)  & \cong  & \Hom_{B_{k,n+1}}(\inc_j^c(M),\inc_j^c(N))
\end{array}
$$
\end{lemma}

\begin{proof}
It is enough to prove the claim for the increase as the other 
one follows similarly. 
Without loss of generality, let $j=n$. Let $r=\rk M$ and $s=\rk N$. 
By construction, 
$x_{n+1}=\Id_r$ and $y_{n+1}=t\Id_r$ in  $\inc_n M$, 
$x_{n+1}=\Id_s$ and $y_{n+1}=t\Id_s$ in $\inc_n N$. So if $F=(F_i)_{i=1}^{n+1}$ is a homomorphism from $\inc_n M$ to $\inc_n N$, then 
$F_n=F_{n+1}$, and the restriction $F=(F_i)_{i=1}^{n}$ is a homomorphism from $M$ to $N$. Conversely, each morphism from $M$ to $N$ is defined by $n$ maps $(F_i)_{i=1}^n$ and can be extended accordingly repeating the map $F_n$ at the vertex $n+1$. Thus, there is a bijection between the homomorphisms from $M$ to $N$ and the homomorphisms from $\inc_n M$ to $\inc_n N$.
\end{proof}

\begin{lemma}\label{lm:inc-indec}
Let $M\in {\rm CM}(B_{k,n})$ be indecomposable with filtration $L_{I_1}\mid \dots \mid L_{I_s}$, let $j\in[n]$. 
We have the following: 

(1) The $1$-increase $\inc_j(M)$ of $M$ 
is indecomposable. 
Furthermore, $\inc_j(M)$ has filtration $L_{I_1(j)}\mid\dots \mid L_{I_s(j)}$, where 
$I_m(j)=\{i\mid i\in I_m, i\le  j\}\cup \{j+1 \} \cup \{i+1\mid i\in I_m,i> j\}$ for $m=1,\dots, s$. 

(2) The $1$-co-increase $\inc_j^c(M)$ of $M$ 
is indecomposable. It has 
filtration $L_{I_1}^c\mid \dots \mid L_{I_s}^c$ 
where 
$I_m(j)^c=\{i\mid i\in I_m, i\le  j\}\cup \{i+1\mid i\in I_m,i> j\}$ for $m=1,\dots, s$.
\end{lemma}

\begin{proof}
It is enough to prove (1). \\
(a) For the indecomposability, 
we use Lemma~\ref{lm:hom-iso} for $N=M$ and the fact that a module is 
indecomposable if and 
only if its endomorphism ring is local. 

\noindent
(b) That the filtrations are as claimed is clear: the increase operation adds a common element to all $k$-subsets 
and inserts a common parallel step to all profiles, see Definition \ref{def:rank1}. \end{proof} 

We next define a (refined) version of the collapse map for modules in ${\rm CM}(B_{k,n})$. 
Let $M$ be in ${\rm CM}(B_{k,n})$, write $M$ as representation $M=(V_i, i \in [n]; x_i, y_i, i \in [n])$ 
and assume $M$ has filtration $L_{I_1}\mid \dots \mid L_{I_s}$. 

\begin{remark}\label{rem:diag-form} 
Let $M=(V_i;x_i,y_i)$ be in ${\rm CM}(B_{k,n})$, assume it 
has filtration 
$L_{I_1}\mid \dots \mid L_{I_s}$. If there exists 
$j\in I_1\cap \dots\cap I_s$, then up to base change we can assume that $x_j=\Id_s$ and 
$y_j=t\Id_s$. 
Similarly, if there exists $j\in I_1^c\cap\dots\cap I_s^c$ then we 
can assume that $x_j=t\Id_s$ and that $y_j=\Id_s$.  
\end{remark}

\begin{definition}[1-decrease and 1-co-decrease]\label{def:1-decrease}
Let $M\in {\rm CM}(B_{k,n})$. 
Assume that $M$ has filtration $L_{I_1}\mid \dots \mid L_{I_s}$, and let 
$M=(V_i, i \in [n]; x_i, y_i, i \in [n])$. 

(a) Let $j \in I_1\cap \dots \cap I_s\ne \emptyset$ and assume $x_j=\Id_s$, $y_j=t\Id_s$. 
The {\bf 1-decrease of $M$ at $j$} is the representation 
$\dec_j(M)$ of $Q_{C(n-1)}$ obtained by removing vertex $j$ of $Q_{C(n)}$ and the 
arrows $x_j$ and $y_j$, so that the arrow $x_{j+1}$ goes from $j-1$ to $j+1$ and the arrow $y_{j+1}$ 
from $j+1$ to $j-1$. 

(b) Assume that $j\in I_1^c\cap \dots \cap I_s^c\ne \emptyset$ and assume that $x_j=t\Id_s$ and 
$y_j=\Id_s$. The the {\bf 1-co-decrease of $M$ at $j$} is the representation of $Q_{C(n-1)}$ 
obtained by removing the vertex $j$ of $Q_{C(n)}$ and the arrows $x_j$ and $y_j$, 
with arrow $x_{j+1}:j-1\to j+1$ and arrow $y_{j+1}: j+1\to j-1$. 

As representations of $Q_{C(n-1)}$, $\dec_j(M)$ and $\dec_j^c(M)$ are as follows: 

$$
\xymatrix@=0.4em{
 V_{n-1} \ar@<3pt>[rr]^{x_1} && V_{1} \ar@<3pt>[rr]^{x_2}\ar@<3pt>[ll]^{y_1} && 
 V_{2}  \ar@<3pt>[rr]^{x_3}\ar@<3pt>[ll]^{y_2} && 
 \cdots \ar@<3pt>[rr]^{x_{j-1}}\ar@<3pt>[ll]^{y_3}   && V_{j-1}\ar@<3pt>[rr]^{x_{j+1}}\ar@<3pt>[ll]^{y_{j-1}}
  && V_{j+1}\ar@<3pt>[rr]^{x_{j+2}}\ar@<3pt>[ll]^{y_{j+1}}  && 
 \cdots\ar@<3pt>[rr]\ar@<3pt>[ll]^{y_{j+2}} &&   
 V_{n-1}\ar@<3pt>[rr]^{x_{n}}\ar@<3pt>[ll] &&  V_{n}\ar@<3pt>[ll]^{y_{n}} 
}
$$

\end{definition}

\begin{lemma}\label{lm:decrease}
Let $M\in {\rm CM}(B_{k,n})$. Then any $1$-decrease of $M$ is a module for $B_{k-1,n-1}$ which is free over the center of 
$B_{k-1,n-1}$ and any $1$-co-decrease 
of $M$ is a module for $B_{k,n-1}$ which is free over the center of $B_{k,n-1}$. 
If $M$ is indecomposable, then its $1$-decreases and $1$-co-decreases are 
indecomposable. 
\end{lemma}

\begin{proof}
We prove the statement about 1-decreases, the claim on 1-co-decreases follows similarly. 
Let $M=L_{I_1}\mid\dots \mid L_{I_s}$. 
Without loss of generality we assume that $n\in I_1\cap \dots\cap I_s$. 

(a) By using analogous arguments as in the proof of Lemma~\ref{A2}, one can easily prove that the relations of $B_{k-1,n-1}$ hold. So $\dec_n(M)$ is indeed a $B_{k-1,n-1}$-module.  

(b) The modules $\dec_j(M)$ are free over the centre of $B_{k-1,n-1}$ by construction. 

(c) Since the endomorphism rings of a module and its 1-decrease are isomorphic, by an 
analogue of Lemma~\ref{lm:hom-iso} for the decreasing maps, 
the claim about indecomposability follows. 
\end{proof}

\begin{remark}
Using $\inc$ and $\dec$ we can view every module in ${\rm CM}(B_{k,n})$ as a module in ${\rm CM}(B_{k+1,n+1})$ or ${\rm CM}(B_{k,n+1})$ as follows: 
Set 
$$ {\rm CM}(B_{k+1,n+1})^j =\{M\in {\rm CM}(B_{k+1,n+1})\mid \mbox{ $j$ is in every filtration factor of $M$}\}$$ and 
$${\rm CM}(B_{k,n+1})^{\widehat{j}}=\{M\in {\rm CM}(B_{k,n+1})\mid \mbox{ none of the filtration 
factors of $M$ contains $j$}\}.$$ Then we have diagrams 
\[
\xymatrix{
 {\rm CM}(B_{k,n})\ar@<3pt>[rr]^{\inc_j}  && {\rm CM}(B_{k+1,n+1})^{j+1} \ar@<3pt>[ll]^{\dec_{j+1}} \\
 {\rm CM}(B_{k,n})\ar@<3pt>[rr]^{\inc_j^c} && {\rm CM}(B_{k,n+1})^{\widehat{j+1}}\ar@<3pt>[ll]^{\dec_{j+1}^c}&
}
\]
\end{remark}

\begin{lemma}\label{lm:increase-rigid}
Let $M\in {\rm CM}(B_{k,n})$ be indecomposable rigid. Then $\inc_j(M)$ is rigid in ${\rm CM}(B_{k+1,n+1})$ 
for any $j\in [n]$ and $\inc_j^c(M)$ is rigid in ${\rm CM}(B_{k,n+1})$ for any $j\in [n]$. 
\end{lemma}

\begin{proof}
We show this for $\inc_j(M)$, the claim about co-increasing follows similarly. Let $M$ be of rank $s$, with 
filtration $L_{I_1}\mid\dots\mid L_{I_s}$. 
Without loss of generality, we consider $j=n$ and let 
$\inc_n(M)$ be the following representation with $V_i=\mathbb{C}[[t]]^s$ for $i=1,\dots, n+1$. 
$$
\xymatrix@=0.7em{
 V_{n+1} \ar@<3pt>[rr]^{x_1} && V_1\ar@<3pt>[rr]^{x_2}\ar@<3pt>[ll]^{y_1} && V_2\ar@<3pt>[rr]^{x_3}\ar@<3pt>[ll]^{y_2} && 
 V_3\ar@<3pt>[rr]\ar@<3pt>[ll]^{y_3}  && \cdots \ar@<3pt>[ll] \ar@<3pt>[rr] && 
 V_{n-1}\ar@<3pt>[rr]^{x_n}\ar@<3pt>[ll] && V_n\ar@<3pt>[rr]^{\Id}\ar@<3pt>[ll]^{y_n} && V_{n+1}\ar@<3pt>[ll]^{t\Id} 
}
$$
Note that $\inc_n(M)$ has filtration 
$L_{I_1(j)}\mid\dots \mid L_{I_s(j)}$ (Lemma~\ref{lm:inc-indec}) where $j\in I_1(j)\cap\dots\cap I_s(j)$.

Assume for contradiction that $\inc_n(M)$ is not rigid, i.e., that there exists a module $N'$ in ${\rm CM}(B_{k+1,n+1})$ and 
a non-trivial short exact sequence 
$$0 \to \inc_n(M)\stackrel{F}{\to} N'\stackrel{G}{\to} \inc_n(M) \to 0.$$ 
The module $N'$, as a representation of the circular 
quiver $Q_{C(n+1)}$, is of the form
$$N'=(U_i, i\in [n+1]; X_i,Y_i, i\in [n+1]),$$ 
with $\mathbb{C}[[t]]^{2s}$ for $i=1,\dots, n+1$. 
The module $\inc_n(M)$ has a label $n+1$ in every rank 1 filtration factor and therefore, 
in any filtration of $N'$, the 
element $n+1$ appears in every $(k+1)$-subset of the rank 1 filtration factors as the content of the middle term is the union of the contents of the end terms of the short exact sequence. 

Furthermore, we can assume $X_{n+1}=\Id_{2s}$ and $Y_{n+1}=t\Id_{2s}$ (by Remark~\ref{rem:diag-form}). 
In particular, we  
can apply the decrease map at $n+1$ to $N'$ and to $\inc_n(M)$ to construct a 
self-extension of $M$ as we will do now.  

Note that in the homomorphism $F=(F_i)_i$, each $F_i$ is $2s\times s$ matrix, $i=1,\dots, n-1$. 
And each of the $G_i$ in $G=(G_i)_i$ is a $s\times 2s$ matrix. 
Since $X_{n+1}=\Id_{2s}$ and $Y_{n+1}=t\Id_{2s}$, we get 
$F_{n+1}=F_n$ and $G_{n+1}=G_n$.

Let $N=\dec_{n+1}(N')$ and recall that by construction, 
$M=\dec_{n+1}\inc_n(M)$. 
By abuse of  notation we write $F$ for $(F_i)_{1\le i\le n}$ and 
$G$ for $(G_i)_{1\le i\le n}$.
Then 
$$
0 \to M\stackrel{F}{\to} N\stackrel{G}{\to} M \to 0
$$
is a short exact sequence in ${\rm CM}(B_{k,n})$. If 
$N$ was the direct sum $M\oplus M$, then we would have 
$N'=\inc_n(M)\oplus \inc_n(M)$, which is not the case. 
\end{proof}

The following lemma can be proved with the same arguments 
as the previous one. 

\begin{lemma}
Let $M\in {\rm CM}(B_{k,n})$ be indecomposable rigid. 
Then any $1$-decrease of $M$ is rigid in ${\rm CM}(B_{k-1,n-1})$ and any $1$-co-decrease is rigid in ${\rm CM}(B_{k,n-1})$. 
\end{lemma}

Let $M$ be in ${\rm CM}(B_{k,n})$. If $M'\in {\rm CM}(B_{k',n'})$ is obtained 
from $M$ by using all possible (co-) decreases, we call $M'$ 
the full reduction of $M$. 
Then we have the following.
\begin{corollary}\label{cor:reducing}
Let $M$ be a rank $2$ module in ${\rm CM}(B_{k,n})$ with filtration 
$L_I\mid L_J$. Then $M$ is indecomposable rigid if and only 
if the full reduction of $M$ is indecomposable rigid. 
In particular, if the rims of an indecomposable module form exactly three boxes, it is rigid. 
\end{corollary}

\begin{proof}
It only remains to prove the statement about the three boxes. 
Assume that $M$ is a fully reduced indecomposable module 
and that its rims form three boxes. Then $\tau(M)$ is a rank 1 module, hence $M$ is rigid by Remark \ref{rem:rigid-tau}. The statement follows. 
\end{proof}

\begin{remark} The statement of Corollary~\ref{cor:reducing} is expected by Le and Y\i ld\i r\i m, \cite{LY}. 
They use an approach via webs to deduce it. Let us mention that King also expects this statement to hold, as was told to us in private communication. 
\end{remark}

To complete the characterisation of rigid indecomposable rank 2 modules, it remains to show that if $M$ is an indecomposable rank 2 module where the rims form 
at least 3 quasi-boxes which are not exactly three boxes, then $M$ is not rigid. By  Corollary~\ref{cor:reducing}, it is 
enough to consider fully reduced modules. 

\begin{conjecture}\label{conj:4-peaks-not-rigid} 
Let $M\in {\rm CM}(B_{k,2k})$ be a fully reduced rank $2$ module where the rims form at least three quasi-boxes but 
not exactly three boxes. Then $M$ is not rigid. 
\end{conjecture}
Note that in case $(k,n)=(4,8)$, the claim is true as such a module is in the non-rigid range of a tube of rank 4. Also, one can check that the  profile of such a module appears repeatedly in the corresponding tube.


\end{document}